\documentclass[preprint,3p,sort&compress]{elsarticle}
\journal{{\tt arXiv.org}}

\usepackage[dvips]{epsfig}
\usepackage{graphicx}
\usepackage{pdfpages}
\usepackage{booktabs} 
\usepackage{latexsym}
\usepackage{verbatim}
\usepackage{amsmath,amsthm}
\usepackage{amssymb}
\usepackage{bbm}
\usepackage{mathtools}
\usepackage{fonts}
\usepackage{enumerate}
\usepackage{placeins}
\usepackage{standalone}
\usepackage{paralist}
\usepackage{subcaption}
\usepackage{blindtext}
\usepackage{relsize}

\definecolor{greenyellow}   {cmyk}{0.15, 0   , 0.69, 0   }
\definecolor{yellow}        {cmyk}{0   , 0   , 1   , 0   }
\definecolor{goldenrod}     {cmyk}{0   , 0.10, 0.84, 0   }
\definecolor{dandelion}     {cmyk}{0   , 0.29, 0.84, 0   }
\definecolor{apricot}       {cmyk}{0   , 0.32, 0.52, 0   }
\definecolor{peach}         {cmyk}{0   , 0.50, 0.70, 0   }
\definecolor{melon}         {cmyk}{0   , 0.46, 0.50, 0   }
\definecolor{yelloworange}  {cmyk}{0   , 0.42, 1   , 0   }
\definecolor{orange}        {cmyk}{0   , 0.61, 0.87, 0   }
\definecolor{burntorange}   {cmyk}{0   , 0.51, 1   , 0   }
\definecolor{bittersweet}   {cmyk}{0   , 0.75, 1   , 0.24}
\definecolor{redorange}     {cmyk}{0   , 0.77, 0.87, 0   }
\definecolor{mahogany}      {cmyk}{0   , 0.85, 0.87, 0.35}
\definecolor{maroon}        {cmyk}{0   , 0.87, 0.68, 0.32}
\definecolor{brickred}      {cmyk}{0   , 0.89, 0.94, 0.28}
\definecolor{red}           {cmyk}{0   , 1   , 1   , 0   }
\definecolor{orangered}     {cmyk}{0   , 1   , 0.50, 0   }
\definecolor{rubinered}     {cmyk}{0   , 1   , 0.13, 0   }
\definecolor{wildstrawberry}{cmyk}{0   , 0.96, 0.39, 0   }
\definecolor{salmon}        {cmyk}{0   , 0.53, 0.38, 0   }
\definecolor{carnationpink} {cmyk}{0   , 0.63, 0   , 0   }
\definecolor{magenta}       {cmyk}{0   , 1   , 0   , 0   }
\definecolor{violetred}     {cmyk}{0   , 0.81, 0   , 0   }
\definecolor{rhodamine}     {cmyk}{0   , 0.82, 0   , 0   }
\definecolor{mulberry}      {cmyk}{0.34, 0.90, 0   , 0.02}
\definecolor{redviolet}     {cmyk}{0.07, 0.90, 0   , 0.34}
\definecolor{fuchsia}       {cmyk}{0.47, 0.91, 0   , 0.08}
\definecolor{lavender}      {cmyk}{0   , 0.48, 0   , 0   }
\definecolor{thistle}       {cmyk}{0.12, 0.59, 0   , 0   }
\definecolor{orchid}        {cmyk}{0.32, 0.64, 0   , 0   }
\definecolor{darkorchid}    {cmyk}{0.40, 0.80, 0.20, 0   }
\definecolor{purple}        {cmyk}{0.45, 0.86, 0   , 0   }
\definecolor{plum}          {cmyk}{0.50, 1   , 0   , 0   }
\definecolor{violet}        {cmyk}{0.79, 0.88, 0   , 0   }
\definecolor{royalpurple}   {cmyk}{0.75, 0.90, 0   , 0   }
\definecolor{blueviolet}    {cmyk}{0.86, 0.91, 0   , 0.04}
\definecolor{periwinkle}    {cmyk}{0.57, 0.55, 0   , 0   }
\definecolor{cadetblue}     {cmyk}{0.62, 0.57, 0.23, 0   }
\definecolor{cornflowerblue}{cmyk}{0.65, 0.13, 0   , 0   }
\definecolor{midnightblue}  {cmyk}{0.98, 0.13, 0   , 0.43}
\definecolor{navyblue}      {cmyk}{0.94, 0.54, 0   , 0   }
\definecolor{royalblue}     {cmyk}{1   , 0.50, 0   , 0   }
\definecolor{blue}          {cmyk}{1   , 1   , 0   , 0   }
\definecolor{cerulean}      {cmyk}{0.94, 0.11, 0   , 0   }
\definecolor{cyan}          {cmyk}{1   , 0   , 0   , 0   }
\definecolor{processblue}   {cmyk}{0.96, 0   , 0   , 0   }
\definecolor{skyblue}       {cmyk}{0.62, 0   , 0.12, 0   }
\definecolor{turquoise}     {cmyk}{0.85, 0   , 0.20, 0   }
\definecolor{tealblue}      {cmyk}{0.86, 0   , 0.34, 0.02}
\definecolor{aquamarine}    {cmyk}{0.82, 0   , 0.30, 0   }
\definecolor{bluegreen}     {cmyk}{0.85, 0   , 0.33, 0   }
\definecolor{emerald}       {cmyk}{1   , 0   , 0.50, 0   }
\definecolor{junglegreen}   {cmyk}{0.99, 0   , 0.52, 0   }
\definecolor{seagreen}      {cmyk}{0.69, 0   , 0.50, 0   }
\definecolor{green}         {cmyk}{1   , 0   , 1   , 0   }
\definecolor{forestgreen}   {cmyk}{0.91, 0   , 0.88, 0.12}
\definecolor{pinegreen}     {cmyk}{0.92, 0   , 0.59, 0.25}
\definecolor{limegreen}     {cmyk}{0.50, 0   , 1   , 0   }
\definecolor{yellowgreen}   {cmyk}{0.44, 0   , 0.74, 0   }
\definecolor{springgreen}   {cmyk}{0.26, 0   , 0.76, 0   }
\definecolor{olivegreen}    {cmyk}{0.64, 0   , 0.95, 0.40}
\definecolor{rawsienna}     {cmyk}{0   , 0.72, 1   , 0.45}
\definecolor{sepia}         {cmyk}{0   , 0.83, 1   , 0.70}
\definecolor{brown}         {cmyk}{0   , 0.81, 1   , 0.60}
\definecolor{tan}           {cmyk}{0.14, 0.42, 0.56, 0   }
\definecolor{gray}          {cmyk}{0   , 0   , 0   , 0.50}
\definecolor{black}         {cmyk}{0   , 0   , 0   , 1   }
\definecolor{white}         {cmyk}{0   , 0   , 0   , 0   }

 \usepackage[]{hyperref}
\usepackage[sort&compress,capitalize,noabbrev]{cleveref}

\usepackage{pgfplots}
\pgfplotsset{compat=newest}

\newcommand{\externaltikz}[2]{\includegraphics{#1}} 
\newtheorem{theorem}{Theorem}[section]
\newtheorem{definition}[theorem]{Definition}
\newtheorem{remark}[theorem]{Remark}

\newtheorem{lemma}[theorem]{Lemma}

\numberwithin{equation}{section}

\newcommand{\bdm}{\begin{displaymath}}
\newcommand{\edm}{\end{displaymath}}
\newcommand{\beq}{\begin{equation}}
\newcommand{\eeq}{\end{equation}}
\newcommand{\beqa}{\begin{eqnarray}}
\newcommand{\eeqa}{\end{eqnarray}}

\title{First-order continuous- and discontinuous-Galerkin moment models for a linear kinetic
	equation: realizability-preserving splitting scheme and numerical analysis}
\author[fs]{Florian Schneider}
\address[fs]{Fachbereich Mathematik, TU Kaiserslautern, Erwin-Schr\"odinger-Str., 67663 Kaiserslautern, Germany, {\tt schneider@mathematik.uni-kl.de}}
\author[tl]{Tobias Leibner\fnref{fn1}}
\address[tl]{Fachbereich Mathematik und Informatik, WWU M\"unster, Einsteinstrasse 62, 48149 M\"unster, {\tt tobias.leibner@uni-muenster.de}}
\fntext[fn1]{\textbf{Funding:} The author acknowledges funding by the Deutsche Forschungsgemeinschaft (DFG, German Research Foundation)
under Germany's Excellence Strategy EXC 2044--390685587, Mathematics Münster: Dynamics–Geometry–Structure.}
\date{}

\newcommand*{\Cpp}{C\nolinebreak[4]\hspace{-.05em}\raisebox{.4ex}{\relsize{-3}{\textbf{++}}}\ }

\def\quand{\quad \mbox{and} \quad}

\newcommand*{\secref}[1]{Section~\ref{#1}}

\newcommand*{\defnref}[1]{Definition~\ref{#1}}

\newcommand*{\figref}[1]{Figure~\ref{#1}}
\newcommand*{\tabref}[1]{Table~\ref{#1}}

\DeclareMathOperator{\minmod}{minmod}
\DeclareMathOperator{\argmin}{argmin}

\DeclareMathOperator{\matexp}{\mathbf{exp}}

\newcommand*{\closure}[1]{\operatorname{cl}\left(#1\right)}
\newcommand*{\convexhull}[1]{\ensuremath{\operatorname{conv}\left(#1\right)}}

\newcommand*{\distancebd}{\ensuremath{d}}
\newcommand*{\floor}[1]{\ensuremath{\left\lfloor #1 \right\rfloor}}

\newcommand*{\interior}[1]{\ensuremath{\operatorname{int}\left(#1\right)}}

\newcommand*{\minmodfunc}[1]{\ensuremath{\minmod\left(#1\right)}}

\newcommand*{\outernormal}{\ensuremath{\Vn}}

\newcommand*{\indicator}[1][\checkerboardmaterialindex]{\ensuremath{\mathbbm{1}_{#1}}}

\providecommand\given{}
\newcommand\SetSymbol[1][]{%
  \nonscript\:#1\vert
  \allowbreak
  \nonscript\:
  \mathopen{}}
\DeclarePairedDelimiterX\Set[1]\{\}{%
\renewcommand\given{\,\SetSymbol[\delimsize]\,}
\,#1\,
}

\newcommand*{\R}{\mathbb{R}}
\newcommand*{\Rpos}{\R^{+}}

\newcommand*{\Lp}[1]{L^{#1}}

\newcommand*{\sphere}[1][2]{\ensuremath{\mathcal{S}^{#1}}}

\newcommand*{\abs}[1]{\ensuremath{\left\vert #1 \right\vert}}
\newcommand*{\norm}[2]{\ensuremath{\left\lVert #1 \right\rVert_{#2}}}

\newcommand*{\eyematrix}{\ensuremath{\MI}}

\newcommand*{\gridindex}{\ensuremath{i}}
\newcommand*{\gridindexalt}{\ensuremath{j}}
\newcommand*{\cellindex}{\gridindex}

\newcommand*{\gridsize}{\ensuremath{I}}
\newcommand*{\timeindex}{\ensuremath{\kappa}}
\newcommand*{\dimension}{\ensuremath{d}}

\newcommand*{\momentindex}{\ensuremath{l}} 
\newcommand*{\momentindexvar}{\ensuremath{m}} 
\newcommand*{\basisindex}{\momentindex} 
\newcommand*{\basisindexvar}{\momentindexvar} 
\newcommand*{\momentorder}{\ensuremath{N}} 
\newcommand*{\momentnumber}{\ensuremath{n}} 
\newcommand*{\altvariable}[1]{\tilde{#1}} 

\newcommand*{\dx}[1][]{\ensuremath{\partial_{\x_{#1}}}}
\newcommand*{\dy}{\ensuremath{\partial_{\y}}}
\newcommand*{\dz}{\ensuremath{\partial_{\z}}}

\newcommand*{\dt}{\ensuremath{\partial_\timevar}}
\newcommand*{\spatialGradient}{\ensuremath{\nabla_\spatialvar}}
\newcommand*{\jacobianop}{\MD}
\newcommand*{\multipliersGradient}{\nabla_{\multipliers}}

\newcommand*{\x}{\ensuremath{x}}
\newcommand*{\y}{\ensuremath{y}}
\newcommand*{\z}{\ensuremath{z}}
\newcommand*{\domain}{\ensuremath{X}}
\newcommand*{\spatialvar}{\ensuremath{\Vx}}

\newcommand*{\gridwidth}{\ensuremath{\Delta\x}}
\newcommand*{\gridwidthx}{\ensuremath{\Delta\x}}
\newcommand*{\gridwidthy}{\ensuremath{\Delta\y}}
\newcommand*{\gridwidthz}{\ensuremath{\Delta\z}}

\newcommand*{\cell}[1]{\ensuremath{I_{#1}}}
\newcommand*{\ccell}[1]{\cell{{#1}}}
\newcommand*{\cellmean}[2][\cellindex]{\ensuremath{\overline{#2}_{#1}}}

\newcommand*{\timevar}{\ensuremath{t}} 
\newcommand*{\timeint}{\ensuremath{T}} 
\newcommand*{\timestep}{\ensuremath{\Delta\timevar}}

\newcommand*{\rkstages}{\ensuremath{s}}

\newcommand*{\tf}{\ensuremath{t_f}} 

\makeatletter
\@tfor\next:=abcdefghijklmnopqrstuvwxyzPQ\do{%
  \def\command@factory#1{%
    \expandafter\def\csname V#1\endcsname{\mathbf{#1}}
  }
  \expandafter\command@factory\next
}
\makeatother

\makeatletter
\def\greekvectors#1{%
\@for\next:=#1\do{%
\def\X##1;{%
\expandafter\def\csname V##1\endcsname{\boldsymbol{\csname##1\endcsname}}
}
\expandafter\X\next;
}
}
\makeatother

\greekvectors{alpha,beta,delta,iota,gamma,lambda,theta,mu,nu,eta,chi,phi,psi,Gamma,Chi,varsigma,varphi}

\makeatletter
\@tfor\next:=ABCDEFGHIJKLMOPQRSTUVWXYZ\do{%
  \def\command@factory#1{%
    \expandafter\def\csname M#1\endcsname{\mathbf{#1}}
  }
  \expandafter\command@factory\next
}
\makeatother

\makeatletter
\def\greekmatrices#1{%
\@for\next:=#1\do{%
\def\X##1;{%
\expandafter\def\csname M##1\endcsname{\boldsymbol{\csname##1\endcsname}}
}
\expandafter\X\next;
}
}
\makeatother

\greekmatrices{Gamma,Chi,Omega}

\newcommand*{\imaginaryunit}{\mathrm{i}}

\newcommand*{\Source}{\ensuremath{\Vs}}
\newcommand*{\collisionop}{\ensuremath{\mathcal{C}}}
\newcommand*{\collision}[1]{\ensuremath{\collisionop\left(#1\right)}}
\newcommand*{\collisionkernel}{\ensuremath{K}}

\newcommand{\SHl}{{\ensuremath{l}}} 
\newcommand{\SHm}{{\ensuremath{m}}} 
\newcommand{\Slmmod}[2]{{\ensuremath{S_{#1}^{#2}}}} 
\newcommand{\Slm}{\Slmmod{\SHl}{\SHm}} 

\newcommand*{\zL}{\ensuremath{\z_{L}}} 
\newcommand*{\zR}{\ensuremath{\z_{R}}} 
\newcommand*{\SC}{\ensuremath{\boldsymbol{\Omega}}} 
\newcommand*{\SCx}{\ensuremath{\Omega_\x}} 
\newcommand*{\SCy}{\ensuremath{\Omega_\y}} 
\newcommand*{\SCz}{\ensuremath{\Omega_\z}} 

\newcommand*{\SCheight}{\ensuremath{\mu}}
\newcommand*{\SCangle}{\ensuremath{\varphi}}

\newcommand*{\Flux}{\ensuremath{\Vf}}

\newcommand*{\numericalflux}{\ensuremath{\Vg}}
\newcommand*{\numericalFlux}{\numericalflux}
\newcommand*{\kineticflux}{\numericalflux^{kin}}
\newcommand*{\kineticFlux}{\kineticflux}
\newcommand*{\viscosityconstant}{\ensuremath{C}}
\newcommand*{\limitervariable}{\ensuremath{\theta}}
\newcommand*{\limitereps}{\ensuremath{\varepsilon_{\mathcal{R}}}}
\newcommand*{\limiterepstilde}{\ensuremath{\tilde{\varepsilon}}}

\newcommand*{\eigenvectormatrix}{\ensuremath{\MV}}

\newcommand*{\eigenvalue}{\ensuremath{\lambda}}

\newcommand*{\scattering}{\ensuremath{\sigma_s}}

\newcommand*{\absorption}{\ensuremath{\sigma_a}}

\newcommand*{\crosssection}{\ensuremath{\sigma_t}}

\newcommand*{\source}{\ensuremath{Q}}

\newcommand*{\moments}[1][ ]{\ensuremath{\Vu_{#1}}} 
\newcommand*{\momentspv}[1]{\ensuremath{\moments[#1]}}
\newcommand*{\momentscellmean}[1]{\ensuremath{\cellmean[#1]{\moments}}} 
\newcommand*{\momentcompcellmean}[1]{\ensuremath{\cellmean[#1]{\momentcomp{}}}}
\newcommand*{\momentschar}{\ensuremath{\Vu^{\text{c}}}} 
\newcommand*{\momentscharcomp}[1][\basisindex]{\ensuremath{u^{\text{c}}_{#1}}}
\newcommand*{\momentscharmeancomp}[1]{\ensuremath{\cellmean[#1]{u}^{\text{c}}}}

\newcommand*{\momentstime}[2]{\moments[#1]^{\left(#2\right)}}
\newcommand*{\momentslimited}[2][\limitervariable]{\moments[#2]^{#1}}
\newcommand*{\isotropicmoment}{\moments[\text{iso}]}

\newcommand*{\momentcomp}[1]{\ensuremath{u_{#1}}}

\newcommand*{\normalizedmoments}[1][ ]{\ensuremath{\Vphi_{#1}}}

\newcommand*{\normalizedmomentcomp}[1]{\ensuremath{\phi_{#1}}}

\newcommand*{\multipliers}[1][ ]{\Valpha_{#1}}
\newcommand*{\numericalmultipliers}{\widetilde{\Valpha}}
\newcommand*{\multipliersone}[1][\basis]{\Valpha_{#1}^{\mathbbm{1}}}

\newcommand*{\density}{\ensuremath{\rho}} 

\newcommand*{\densityvacuum}{\ensuremath{\rho}_{vac}}
\newcommand*{\densityMult}{\ensuremath{\varrho}} 

\newcommand*{\distribution}[1][ ]{\ensuremath{\psi_{#1}}}
\newcommand*{\distributiontzero}{\ensuremath{\distribution[\timevar=0]}}
\newcommand*{\distributionboundary}{\ensuremath{\distribution[b]}}
\newcommand*{\distributionvacuum}{\ensuremath{\distribution[\text{vac}]}}

\newcommand*{\basis}[1][ ]{{\ensuremath{\Vb_{#1}}}} 
\newcommand*{\basiscomp}[1][\basisindex]{\ensuremath{b_{#1}}}
\newcommand*{\fmbasis}[1][\momentorder]{\ensuremath{\Vf_{#1}}}

\newcommand*{\pmbasis}[1][\momentnumber]{\ensuremath{\Vp_{#1}}}

\newcommand*{\hfbasis}[1][\momentnumber]{\ensuremath{\Vh_{#1}}}

\newcommand*{\hfbasiscomp}[1][\basisindex]{\ensuremath{h_{#1}}}
\newcommand*{\PN}[1][\momentorder]{\ensuremath{\text{P}_{#1}}}
\newcommand*{\Mn}[1][\momentorder]{\ensuremath{\text{M}_{#1}}}
\newcommand*{\MN}[1][\momentorder]{\ensuremath{\text{M}_{#1}}}

\newcommand*{\HFMN}[1][\momentnumber]{\ensuremath{\text{HFM}_{#1}}}

\newcommand*{\HFPN}[1][\momentnumber]{\ensuremath{\text{HFP}_{#1}}}

\newcommand*{\PMMN}[1][\momentnumber]{\ensuremath{\text{PMM}_{#1}}}

\newcommand*{\PMPN}[1][\momentnumber]{\ensuremath{\text{PMP}_{#1}}}

\newcommand*{\RD}[2]{\ensuremath{\mathcal{R}_{#1}^{#2}}}

\newcommand*{\RDpos}[1]{\ensuremath{\mathcal{R}^{+}_{#1}}}
\newcommand*{\RQ}[1]{\RD{#1}{\mathcal{Q}}}

\newcommand*{\RQeps}[1]{\RD{#1}{\mathcal{Q},\limitereps}}
\newcommand*{\RQone}[1]{\left. \RQ{#1} \right|_{\density \leq 1}}

\newcommand*{\RQonehalfspaceveccomp}{d}
\newcommand*{\RQonehalfspacematcol}{\Vc}
\newcommand*{\facetindex}{i}

\newcommand*{\hankelhalfind}{k}

\newcommand*{\angularDomain}{\ensuremath{V}}

\newcommand*{\angularQuadrature}{\ensuremath{\mathcal{Q}}}
\newcommand*{\angularQuadratureFunction}{\ensuremath{f}}
\newcommand*{\angularQuadratureIndex}{\ensuremath{i}}
\newcommand*{\angularQuadratureNumber}[1][\angularQuadrature]{\ensuremath{{n_{#1}}}}

\newcommand*{\angularQuadratureNodesS}[1]{\ensuremath{\SC_{#1}}}
\newcommand*{\angularQuadratureWeights}[1]{\ensuremath{w_{#1}}}

\newcommand*{\ints}[1]{\ensuremath{\left<#1\right>}} 
\newcommand*{\intA}[2]{\ensuremath{\left<#1\right>_{#2}}}

\newcommand*{\intQ}[1]{\ensuremath{\intA{#1}{\angularQuadrature}}} 
\newcommand*{\intp}[1]{\intA{#1}{+}}
\newcommand*{\intm}[1]{\intA{#1}{-}}

\newcommand*{\ld}[1]{\ensuremath{{#1}_*}} 
\newcommand*{\entropy}{\ensuremath{\eta}} 
\newcommand*{\entropyFunctional}{\ensuremath{\mathcal{H}}} 
\newcommand*{\ansatz}[1][ ]{\ensuremath{\hat{\psi}_{#1}}}

\newcommand*{\opttol}{\ensuremath{\tau}}
\newcommand*{\opttolmod}{\ensuremath{\tau'}}
\newcommand*{\opttoleps}{\ensuremath{{\varepsilon_{\gamma}}}}

\newcommand*{\optObjective}{p}
\newcommand*{\optGradient}{\ensuremath{\Vg}}
\newcommand*{\optHessian}{\ensuremath{\MH}}
\newcommand*{\newtonDirection}{\ensuremath{\Vd}}
\newcommand*{\regularizationParameter}[1][ ]{\ensuremath{r_{#1}}}

\newcommand*{\regularizedmoments}[1][ ]{\ensuremath{\moments[#1]^{\regularizationParameter}}}

\newcommand*{\multipliersmod}{\ensuremath{\Vbeta'}}
\newcommand*{\multiplierstilde}{\ensuremath{\Vbeta}}
\newcommand*{\isomatrix}{\ensuremath{\MG}}
\newcommand*{\optNewtonDirection}{\newtonDirection}
\newcommand*{\velocityHessian}{\MJ}

\makeatletter
\DeclareFontFamily{U}{tipa}{}
\DeclareFontShape{U}{tipa}{m}{n}{<->tipa10}{}
\newcommand*{\arc@char}{{\usefont{U}{tipa}{m}{n}\symbol{62}}}%
\newcommand*{\arc}[1]{\mathpalette\arc@arc{#1}}
\newcommand*{\arc@arc}[2]{%
  \sbox0{$\m@th#1#2$}%
  \vbox{
    \hbox{\resizebox{\wd0}{\height}{\arc@char}}
    \nointerlineskip
    \box0
  }%
}
\makeatother
\newcommand*{\TriangulationSphere}{\ensuremath{\mathcal{T}_h}}
\newcommand*{\sphericaltriangle}{\ensuremath{\arc{K}}}

\newcommand*{\refinementnumber}{\ensuremath{r}} 
\newcommand*{\nvertex}{\ensuremath{n_v}}
\newcommand*{\nentity}{\ensuremath{n_t}}

\newcommand*{\generalpartition}{\ensuremath{\mathcal{P}}}

\newcommand*{\numfacets}{\ensuremath{n_{f}}} 

\definecolor{greenyellow}   {cmyk}{0.15, 0   , 0.69, 0   }
\definecolor{yellow}        {cmyk}{0   , 0   , 1   , 0   }
\definecolor{goldenrod}     {cmyk}{0   , 0.10, 0.84, 0   }
\definecolor{dandelion}     {cmyk}{0   , 0.29, 0.84, 0   }
\definecolor{apricot}       {cmyk}{0   , 0.32, 0.52, 0   }
\definecolor{peach}         {cmyk}{0   , 0.50, 0.70, 0   }
\definecolor{melon}         {cmyk}{0   , 0.46, 0.50, 0   }
\definecolor{yelloworange}  {cmyk}{0   , 0.42, 1   , 0   }
\definecolor{orange}        {cmyk}{0   , 0.61, 0.87, 0   }
\definecolor{burntorange}   {cmyk}{0   , 0.51, 1   , 0   }
\definecolor{bittersweet}   {cmyk}{0   , 0.75, 1   , 0.24}
\definecolor{redorange}     {cmyk}{0   , 0.77, 0.87, 0   }
\definecolor{mahogany}      {cmyk}{0   , 0.85, 0.87, 0.35}
\definecolor{maroon}        {cmyk}{0   , 0.87, 0.68, 0.32}
\definecolor{brickred}      {cmyk}{0   , 0.89, 0.94, 0.28}
\definecolor{red}           {cmyk}{0   , 1   , 1   , 0   }
\definecolor{orangered}     {cmyk}{0   , 1   , 0.50, 0   }
\definecolor{rubinered}     {cmyk}{0   , 1   , 0.13, 0   }
\definecolor{wildstrawberry}{cmyk}{0   , 0.96, 0.39, 0   }
\definecolor{salmon}        {cmyk}{0   , 0.53, 0.38, 0   }
\definecolor{carnationpink} {cmyk}{0   , 0.63, 0   , 0   }
\definecolor{magenta}       {cmyk}{0   , 1   , 0   , 0   }
\definecolor{violetred}     {cmyk}{0   , 0.81, 0   , 0   }
\definecolor{rhodamine}     {cmyk}{0   , 0.82, 0   , 0   }
\definecolor{mulberry}      {cmyk}{0.34, 0.90, 0   , 0.02}
\definecolor{redviolet}     {cmyk}{0.07, 0.90, 0   , 0.34}
\definecolor{fuchsia}       {cmyk}{0.47, 0.91, 0   , 0.08}
\definecolor{lavender}      {cmyk}{0   , 0.48, 0   , 0   }
\definecolor{thistle}       {cmyk}{0.12, 0.59, 0   , 0   }
\definecolor{orchid}        {cmyk}{0.32, 0.64, 0   , 0   }
\definecolor{darkorchid}    {cmyk}{0.40, 0.80, 0.20, 0   }
\definecolor{purple}        {cmyk}{0.45, 0.86, 0   , 0   }
\definecolor{plum}          {cmyk}{0.50, 1   , 0   , 0   }
\definecolor{violet}        {cmyk}{0.79, 0.88, 0   , 0   }
\definecolor{royalpurple}   {cmyk}{0.75, 0.90, 0   , 0   }
\definecolor{blueviolet}    {cmyk}{0.86, 0.91, 0   , 0.04}
\definecolor{periwinkle}    {cmyk}{0.57, 0.55, 0   , 0   }
\definecolor{cadetblue}     {cmyk}{0.62, 0.57, 0.23, 0   }
\definecolor{cornflowerblue}{cmyk}{0.65, 0.13, 0   , 0   }
\definecolor{midnightblue}  {cmyk}{0.98, 0.13, 0   , 0.43}
\definecolor{navyblue}      {cmyk}{0.94, 0.54, 0   , 0   }
\definecolor{royalblue}     {cmyk}{1   , 0.50, 0   , 0   }
\definecolor{blue}          {cmyk}{1   , 1   , 0   , 0   }
\definecolor{cerulean}      {cmyk}{0.94, 0.11, 0   , 0   }
\definecolor{cyan}          {cmyk}{1   , 0   , 0   , 0   }
\definecolor{processblue}   {cmyk}{0.96, 0   , 0   , 0   }
\definecolor{skyblue}       {cmyk}{0.62, 0   , 0.12, 0   }
\definecolor{turquoise}     {cmyk}{0.85, 0   , 0.20, 0   }
\definecolor{tealblue}      {cmyk}{0.86, 0   , 0.34, 0.02}
\definecolor{aquamarine}    {cmyk}{0.82, 0   , 0.30, 0   }
\definecolor{bluegreen}     {cmyk}{0.85, 0   , 0.33, 0   }
\definecolor{emerald}       {cmyk}{1   , 0   , 0.50, 0   }
\definecolor{junglegreen}   {cmyk}{0.99, 0   , 0.52, 0   }
\definecolor{seagreen}      {cmyk}{0.69, 0   , 0.50, 0   }
\definecolor{green}         {cmyk}{1   , 0   , 1   , 0   }
\definecolor{forestgreen}   {cmyk}{0.91, 0   , 0.88, 0.12}
\definecolor{pinegreen}     {cmyk}{0.92, 0   , 0.59, 0.25}
\definecolor{limegreen}     {cmyk}{0.50, 0   , 1   , 0   }
\definecolor{yellowgreen}   {cmyk}{0.44, 0   , 0.74, 0   }
\definecolor{springgreen}   {cmyk}{0.26, 0   , 0.76, 0   }
\definecolor{olivegreen}    {cmyk}{0.64, 0   , 0.95, 0.40}
\definecolor{rawsienna}     {cmyk}{0   , 0.72, 1   , 0.45}
\definecolor{sepia}         {cmyk}{0   , 0.83, 1   , 0.70}
\definecolor{brown}         {cmyk}{0   , 0.81, 1   , 0.60}
\definecolor{tan}           {cmyk}{0.14, 0.42, 0.56, 0   }
\definecolor{gray}          {cmyk}{0   , 0   , 0   , 0.50}
\definecolor{black}         {cmyk}{0   , 0   , 0   , 1   }
\definecolor{white}         {cmyk}{0   , 0   , 0   , 0   }

\pgfplotscreateplotcyclelist{color parula}{%
  royalblue!70,every mark/.append style={solid,line width = 0pt,fill=royalblue!60!black},mark=ball\\%
  goldenrod!50!yelloworange,every mark/.append style={solid,fill=goldenrod!30!black},mark=*\\%
  green,every mark/.append style={black,fill=yellowgreen},mark=diamond*\\%
  bluegreen,every mark/.append style={fill=black},mark=triangle*\\%
  royalblue!70,densely dashed,every mark/.append style={solid,line width = 0pt,fill=royalblue!60!black},mark=ball\\%
  goldenrod!50!yelloworange,densely dashed,every mark/.append style={solid,black,fill=goldenrod},mark=diamond*\\%
  yellowgreen,densely dashed,every mark/.append style={solid,fill=yellowgreen!60!royalblue},mark=square*, mark size= 1pt\\%
  bluegreen,densely dashed,mark size = 1.5pt,every mark/.append style={solid,fill=white},mark=otimes*\\%
  royalblue,densely dashed,mark size = 1.5pt,every mark/.append style={solid,fill=royalblue!30!black},mark=pentagon*\\%
  goldenrod!40!yelloworange,every mark/.append style={solid,fill=goldenrod!30!black},mark=|\\%
}

\pgfplotscreateplotcyclelist{color parula pairwise}{%
royalblue!70,solid,every mark/.append style={solid,line width = 0pt,fill=royalblue!60!black},mark=ball\\%
royalblue!70,densely dashed,every mark/.append style={solid,line width = 0pt,fill=royalblue!60!black},mark=*\\%
goldenrod!50!yelloworange,solid,every mark/.append
style={solid,fill=goldenrod!30!black},mark=diamond*\\%
goldenrod!50!yelloworange,densely dashed, every mark/.append
style={solid,fill=goldenrod!30!black},mark=triangle*\\%
green,solid,every mark/.append style={solid, black,fill=yellowgreen},mark=square*\\%
green,densely dashed,every mark/.append style={solid, black, fill=yellowgreen},mark=pentagon*\\%
bluegreen,every mark/.append style={fill=black},mark=triangle*\\%
blueviolet,densely dashed,every mark/.append style={solid,fill=blueviolet},mark=square*, mark size= 1pt\\%
bluegreen,densely dashed,mark size = 1.5pt,every mark/.append style={solid,fill=white},mark=otimes*\\%
royalblue,densely dashed,mark size = 1.5pt,every mark/.append style={solid,fill=royalblue!30!black},mark=pentagon*\\%
goldenrod!40!yelloworange,every mark/.append style={solid,fill=goldenrod!30!black},mark=|\\%
}

\pgfplotsset{
  discard if not/.style 2 args={
      x filter/.code={
          \edef\tempa{\thisrow{#1}}
          \edef\tempb{#2}
          \ifx\tempa\tempb
          \else
            
          \fi
        }
    }
}

\usepgfplotslibrary{groupplots}
\newlength{\figureheight}
\newlength{\figurewidth}
\pgfplotsset{%
  ,compat=1.11
  ,colormap={parula}{%
      rgb(0pt)=(0.2081,0.1663,0.5292);
      rgb(1pt)=(0.208355,0.16778,0.532238);
      rgb(2pt)=(0.208611,0.169261,0.535275);
      rgb(3pt)=(0.208866,0.170741,0.538313);
      rgb(4pt)=(0.209121,0.172222,0.54135);
      rgb(5pt)=(0.209376,0.173702,0.544388);
      rgb(6pt)=(0.209632,0.175183,0.547425);
      rgb(7pt)=(0.209887,0.176663,0.550463);
      rgb(8pt)=(0.210134,0.178144,0.553505);
      rgb(9pt)=(0.210338,0.179624,0.556568);
      rgb(10pt)=(0.210542,0.181105,0.559631);
      rgb(11pt)=(0.210746,0.182585,0.562694);
      rgb(12pt)=(0.210944,0.184066,0.565763);
      rgb(13pt)=(0.211123,0.185546,0.568852);
      rgb(14pt)=(0.211302,0.187027,0.57194);
      rgb(15pt)=(0.21148,0.188507,0.575029);
      rgb(16pt)=(0.211642,0.189996,0.578117);
      rgb(17pt)=(0.21177,0.191502,0.581206);
      rgb(18pt)=(0.211897,0.193008,0.584295);
      rgb(19pt)=(0.212025,0.194514,0.587383);
      rgb(20pt)=(0.212132,0.19602,0.590472);
      rgb(21pt)=(0.212208,0.197526,0.59356);
      rgb(22pt)=(0.212285,0.199032,0.596649);
      rgb(23pt)=(0.212361,0.200538,0.599738);
      rgb(24pt)=(0.212413,0.202044,0.602839);
      rgb(25pt)=(0.212438,0.20355,0.605953);
      rgb(26pt)=(0.212464,0.205056,0.609067);
      rgb(27pt)=(0.212489,0.206562,0.612181);
      rgb(28pt)=(0.212471,0.208083,0.61531);
      rgb(29pt)=(0.21242,0.209614,0.61845);
      rgb(30pt)=(0.212368,0.211146,0.621589);
      rgb(31pt)=(0.212317,0.212677,0.624729);
      rgb(32pt)=(0.212216,0.214209,0.627868);
      rgb(33pt)=(0.212088,0.215741,0.631008);
      rgb(34pt)=(0.211961,0.217272,0.634148);
      rgb(35pt)=(0.211833,0.218804,0.637287);
      rgb(36pt)=(0.211668,0.220354,0.640446);
      rgb(37pt)=(0.211489,0.221911,0.643611);
      rgb(38pt)=(0.21131,0.223468,0.646776);
      rgb(39pt)=(0.211132,0.225025,0.649941);
      rgb(40pt)=(0.210848,0.226603,0.653107);
      rgb(41pt)=(0.210541,0.228186,0.656272);
      rgb(42pt)=(0.210235,0.229768,0.659437);
      rgb(43pt)=(0.209929,0.231351,0.662602);
      rgb(44pt)=(0.209553,0.232934,0.665767);
      rgb(45pt)=(0.20917,0.234516,0.668932);
      rgb(46pt)=(0.208787,0.236099,0.672098);
      rgb(47pt)=(0.208405,0.237681,0.675263);
      rgb(48pt)=(0.20787,0.239289,0.678453);
      rgb(49pt)=(0.207334,0.240897,0.681644);
      rgb(50pt)=(0.206798,0.242505,0.684835);
      rgb(51pt)=(0.206255,0.244114,0.688025);
      rgb(52pt)=(0.205617,0.245722,0.691216);
      rgb(53pt)=(0.204979,0.24733,0.694407);
      rgb(54pt)=(0.204341,0.248938,0.697597);
      rgb(55pt)=(0.203675,0.250554,0.700792);
      rgb(56pt)=(0.202858,0.252213,0.704008);
      rgb(57pt)=(0.202041,0.253872,0.707224);
      rgb(58pt)=(0.201225,0.255531,0.710441);
      rgb(59pt)=(0.200372,0.257184,0.713657);
      rgb(60pt)=(0.199402,0.258818,0.716873);
      rgb(61pt)=(0.198432,0.260452,0.720089);
      rgb(62pt)=(0.197462,0.262085,0.723305);
      rgb(63pt)=(0.196419,0.263735,0.726522);
      rgb(64pt)=(0.195219,0.26542,0.729738);
      rgb(65pt)=(0.19402,0.267105,0.732954);
      rgb(66pt)=(0.19282,0.268789,0.73617);
      rgb(67pt)=(0.191549,0.270474,0.739386);
      rgb(68pt)=(0.19017,0.272159,0.742603);
      rgb(69pt)=(0.188792,0.273843,0.745819);
      rgb(70pt)=(0.187414,0.275528,0.749035);
      rgb(71pt)=(0.1859,0.277237,0.752264);
      rgb(72pt)=(0.184241,0.278973,0.755505);
      rgb(73pt)=(0.182581,0.280709,0.758747);
      rgb(74pt)=(0.180922,0.282444,0.761989);
      rgb(75pt)=(0.179133,0.284209,0.765245);
      rgb(76pt)=(0.177244,0.285996,0.768512);
      rgb(77pt)=(0.175356,0.287783,0.77178);
      rgb(78pt)=(0.173467,0.289569,0.775047);
      rgb(79pt)=(0.171363,0.291406,0.778314);
      rgb(80pt)=(0.169142,0.293269,0.781581);
      rgb(81pt)=(0.166922,0.295132,0.784849);
      rgb(82pt)=(0.164701,0.296996,0.788116);
      rgb(83pt)=(0.162238,0.298934,0.791365);
      rgb(84pt)=(0.159686,0.300899,0.794606);
      rgb(85pt)=(0.157133,0.302865,0.797848);
      rgb(86pt)=(0.15458,0.30483,0.80109);
      rgb(87pt)=(0.151738,0.306858,0.804352);
      rgb(88pt)=(0.148828,0.3089,0.80762);
      rgb(89pt)=(0.145918,0.310942,0.810887);
      rgb(90pt)=(0.143008,0.312984,0.814154);
      rgb(91pt)=(0.139687,0.31514,0.81733);
      rgb(92pt)=(0.136318,0.31731,0.820495);
      rgb(93pt)=(0.132949,0.319479,0.82366);
      rgb(94pt)=(0.129579,0.321649,0.826826);
      rgb(95pt)=(0.125811,0.323918,0.829841);
      rgb(96pt)=(0.122033,0.32619,0.832853);
      rgb(97pt)=(0.118256,0.328462,0.835865);
      rgb(98pt)=(0.114458,0.330737,0.838862);
      rgb(99pt)=(0.110349,0.333059,0.841619);
      rgb(100pt)=(0.106239,0.335382,0.844376);
      rgb(101pt)=(0.102129,0.337705,0.847132);
      rgb(102pt)=(0.0979874,0.340021,0.849835);
      rgb(103pt)=(0.093648,0.342292,0.852209);
      rgb(104pt)=(0.0893087,0.344564,0.854583);
      rgb(105pt)=(0.0849694,0.346836,0.856957);
      rgb(106pt)=(0.08063,0.349091,0.859234);
      rgb(107pt)=(0.0762907,0.351286,0.861174);
      rgb(108pt)=(0.0719514,0.353481,0.863114);
      rgb(109pt)=(0.067612,0.355676,0.865053);
      rgb(110pt)=(0.0633195,0.357817,0.866853);
      rgb(111pt)=(0.0591333,0.359833,0.868333);
      rgb(112pt)=(0.0549471,0.36185,0.869814);
      rgb(113pt)=(0.050761,0.363866,0.871294);
      rgb(114pt)=(0.0466838,0.365823,0.872626);
      rgb(115pt)=(0.0427784,0.367687,0.873724);
      rgb(116pt)=(0.038873,0.36955,0.874821);
      rgb(117pt)=(0.0349676,0.371414,0.875919);
      rgb(118pt)=(0.0315066,0.373217,0.876872);
      rgb(119pt)=(0.0285456,0.374953,0.877664);
      rgb(120pt)=(0.0255847,0.376688,0.878455);
      rgb(121pt)=(0.0226237,0.378424,0.879246);
      rgb(122pt)=(0.0202132,0.380061,0.879868);
      rgb(123pt)=(0.0182477,0.381618,0.880353);
      rgb(124pt)=(0.0162823,0.383175,0.880838);
      rgb(125pt)=(0.0143168,0.384732,0.881323);
      rgb(126pt)=(0.0127892,0.386241,0.881695);
      rgb(127pt)=(0.0115129,0.387721,0.882001);
      rgb(128pt)=(0.0102366,0.389202,0.882307);
      rgb(129pt)=(0.00896036,0.390682,0.882614);
      rgb(130pt)=(0.00812372,0.392089,0.88281);
      rgb(131pt)=(0.00746006,0.393468,0.882963);
      rgb(132pt)=(0.0067964,0.394846,0.883116);
      rgb(133pt)=(0.00613273,0.396224,0.883269);
      rgb(134pt)=(0.00581622,0.397562,0.88332);
      rgb(135pt)=(0.00558649,0.398889,0.883346);
      rgb(136pt)=(0.00535676,0.400217,0.883371);
      rgb(137pt)=(0.00512703,0.401544,0.883397);
      rgb(138pt)=(0.00516757,0.402804,0.883332);
      rgb(139pt)=(0.00524414,0.404054,0.883256);
      rgb(140pt)=(0.00532072,0.405305,0.883179);
      rgb(141pt)=(0.0053973,0.406556,0.883103);
      rgb(142pt)=(0.00572012,0.407757,0.882952);
      rgb(143pt)=(0.00605195,0.408957,0.882799);
      rgb(144pt)=(0.00638378,0.410157,0.882646);
      rgb(145pt)=(0.00672643,0.411355,0.882489);
      rgb(146pt)=(0.00728799,0.412529,0.882259);
      rgb(147pt)=(0.00784955,0.413704,0.88203);
      rgb(148pt)=(0.00841111,0.414878,0.8818);
      rgb(149pt)=(0.00898919,0.416045,0.881564);
      rgb(150pt)=(0.00967838,0.417168,0.881283);
      rgb(151pt)=(0.0103676,0.418292,0.881002);
      rgb(152pt)=(0.0110568,0.419415,0.880721);
      rgb(153pt)=(0.011773,0.420532,0.880435);
      rgb(154pt)=(0.0125898,0.42163,0.880129);
      rgb(155pt)=(0.0134066,0.422728,0.879823);
      rgb(156pt)=(0.0142234,0.423825,0.879516);
      rgb(157pt)=(0.0150703,0.424915,0.879195);
      rgb(158pt)=(0.0159892,0.425987,0.878838);
      rgb(159pt)=(0.0169081,0.427059,0.87848);
      rgb(160pt)=(0.017827,0.428132,0.878123);
      rgb(161pt)=(0.0187748,0.429194,0.877746);
      rgb(162pt)=(0.0197703,0.430241,0.877338);
      rgb(163pt)=(0.0207658,0.431287,0.876929);
      rgb(164pt)=(0.0217613,0.432334,0.876521);
      rgb(165pt)=(0.0227802,0.43338,0.876113);
      rgb(166pt)=(0.0238267,0.434427,0.875704);
      rgb(167pt)=(0.0248733,0.435473,0.875296);
      rgb(168pt)=(0.0259198,0.43652,0.874887);
      rgb(169pt)=(0.0269802,0.437553,0.874451);
      rgb(170pt)=(0.0280523,0.438574,0.873992);
      rgb(171pt)=(0.0291243,0.439595,0.873532);
      rgb(172pt)=(0.0301964,0.440616,0.873073);
      rgb(173pt)=(0.0312844,0.441621,0.872614);
      rgb(174pt)=(0.032382,0.442616,0.872154);
      rgb(175pt)=(0.0334796,0.443612,0.871695);
      rgb(176pt)=(0.0345772,0.444607,0.871235);
      rgb(177pt)=(0.0357108,0.445603,0.870758);
      rgb(178pt)=(0.0368595,0.446598,0.870273);
      rgb(179pt)=(0.0380081,0.447594,0.869788);
      rgb(180pt)=(0.0391568,0.448589,0.869303);
      rgb(181pt)=(0.0402652,0.449565,0.868798);
      rgb(182pt)=(0.0413628,0.450535,0.868287);
      rgb(183pt)=(0.0424604,0.451505,0.867777);
      rgb(184pt)=(0.043558,0.452474,0.867266);
      rgb(185pt)=(0.0445889,0.453444,0.866756);
      rgb(186pt)=(0.0456099,0.454414,0.866245);
      rgb(187pt)=(0.0466309,0.455384,0.865735);
      rgb(188pt)=(0.047652,0.456354,0.865224);
      rgb(189pt)=(0.0486,0.457324,0.864714);
      rgb(190pt)=(0.0495444,0.458294,0.864203);
      rgb(191pt)=(0.0504889,0.459264,0.863692);
      rgb(192pt)=(0.0514315,0.460234,0.863181);
      rgb(193pt)=(0.0523249,0.461204,0.862645);
      rgb(194pt)=(0.0532183,0.462174,0.862109);
      rgb(195pt)=(0.0541117,0.463144,0.861573);
      rgb(196pt)=(0.0549991,0.464111,0.861034);
      rgb(197pt)=(0.0558414,0.465056,0.860472);
      rgb(198pt)=(0.0566838,0.466,0.859911);
      rgb(199pt)=(0.0575261,0.466944,0.859349);
      rgb(200pt)=(0.0583532,0.467889,0.858793);
      rgb(201pt)=(0.0591189,0.468833,0.858257);
      rgb(202pt)=(0.0598847,0.469778,0.857721);
      rgb(203pt)=(0.0606505,0.470722,0.857185);
      rgb(204pt)=(0.0614018,0.471667,0.856641);
      rgb(205pt)=(0.0621165,0.472611,0.85608);
      rgb(206pt)=(0.0628312,0.473556,0.855518);
      rgb(207pt)=(0.0635459,0.4745,0.854957);
      rgb(208pt)=(0.064242,0.475444,0.854405);
      rgb(209pt)=(0.0649057,0.476389,0.853868);
      rgb(210pt)=(0.0655694,0.477333,0.853332);
      rgb(211pt)=(0.066233,0.478278,0.852796);
      rgb(212pt)=(0.0668625,0.479222,0.852249);
      rgb(213pt)=(0.0674495,0.480167,0.851687);
      rgb(214pt)=(0.0680366,0.481111,0.851126);
      rgb(215pt)=(0.0686237,0.482056,0.850564);
      rgb(216pt)=(0.0691838,0.483,0.850003);
      rgb(217pt)=(0.0697198,0.483944,0.849441);
      rgb(218pt)=(0.0702559,0.484889,0.84888);
      rgb(219pt)=(0.0707919,0.485833,0.848318);
      rgb(220pt)=(0.0712967,0.486778,0.847772);
      rgb(221pt)=(0.0717817,0.487722,0.847236);
      rgb(222pt)=(0.0722667,0.488667,0.8467);
      rgb(223pt)=(0.0727517,0.489611,0.846164);
      rgb(224pt)=(0.0732012,0.490573,0.845628);
      rgb(225pt)=(0.0736351,0.491543,0.845092);
      rgb(226pt)=(0.0740691,0.492513,0.844556);
      rgb(227pt)=(0.074503,0.493483,0.84402);
      rgb(228pt)=(0.0748973,0.494433,0.843484);
      rgb(229pt)=(0.0752802,0.495378,0.842948);
      rgb(230pt)=(0.0756631,0.496322,0.842412);
      rgb(231pt)=(0.0760459,0.497267,0.841876);
      rgb(232pt)=(0.0763631,0.498233,0.841362);
      rgb(233pt)=(0.0766694,0.499203,0.840851);
      rgb(234pt)=(0.0769757,0.500173,0.840341);
      rgb(235pt)=(0.077282,0.501143,0.83983);
      rgb(236pt)=(0.0775162,0.502137,0.83932);
      rgb(237pt)=(0.0777459,0.503132,0.838809);
      rgb(238pt)=(0.0779757,0.504128,0.838298);
      rgb(239pt)=(0.0782042,0.505123,0.837789);
      rgb(240pt)=(0.0783829,0.506093,0.837304);
      rgb(241pt)=(0.0785616,0.507063,0.836819);
      rgb(242pt)=(0.0787402,0.508033,0.836334);
      rgb(243pt)=(0.0789135,0.509008,0.835851);
      rgb(244pt)=(0.0790411,0.510029,0.835392);
      rgb(245pt)=(0.0791688,0.51105,0.834932);
      rgb(246pt)=(0.0792964,0.512071,0.834473);
      rgb(247pt)=(0.0794048,0.513092,0.834018);
      rgb(248pt)=(0.0794303,0.514113,0.833584);
      rgb(249pt)=(0.0794559,0.515134,0.83315);
      rgb(250pt)=(0.0794814,0.516155,0.832717);
      rgb(251pt)=(0.0794862,0.517183,0.832289);
      rgb(252pt)=(0.0794351,0.51823,0.831881);
      rgb(253pt)=(0.0793841,0.519276,0.831473);
      rgb(254pt)=(0.079333,0.520323,0.831064);
      rgb(255pt)=(0.079255,0.521369,0.830665);
      rgb(256pt)=(0.0791273,0.522416,0.830282);
      rgb(257pt)=(0.0789997,0.523462,0.829899);
      rgb(258pt)=(0.0788721,0.524509,0.829516);
      rgb(259pt)=(0.0786889,0.525589,0.829156);
      rgb(260pt)=(0.0784336,0.526712,0.828824);
      rgb(261pt)=(0.0781784,0.527835,0.828492);
      rgb(262pt)=(0.0779231,0.528958,0.82816);
      rgb(263pt)=(0.077615,0.530081,0.827868);
      rgb(264pt)=(0.0772577,0.531205,0.827613);
      rgb(265pt)=(0.0769003,0.532328,0.827357);
      rgb(266pt)=(0.0765429,0.533451,0.827102);
      rgb(267pt)=(0.0761243,0.534589,0.826862);
      rgb(268pt)=(0.0756649,0.535738,0.826632);
      rgb(269pt)=(0.0752054,0.536886,0.826403);
      rgb(270pt)=(0.0747459,0.538035,0.826173);
      rgb(271pt)=(0.0742168,0.539219,0.825961);
      rgb(272pt)=(0.0736553,0.540418,0.825756);
      rgb(273pt)=(0.0730937,0.541618,0.825552);
      rgb(274pt)=(0.0725321,0.542818,0.825348);
      rgb(275pt)=(0.0718925,0.544037,0.825183);
      rgb(276pt)=(0.0712288,0.545262,0.82503);
      rgb(277pt)=(0.0705652,0.546487,0.824877);
      rgb(278pt)=(0.0699015,0.547713,0.824723);
      rgb(279pt)=(0.0691514,0.548938,0.824614);
      rgb(280pt)=(0.0683856,0.550163,0.824511);
      rgb(281pt)=(0.0676198,0.551388,0.824409);
      rgb(282pt)=(0.0668541,0.552614,0.824307);
      rgb(283pt)=(0.0660408,0.553886,0.824205);
      rgb(284pt)=(0.065224,0.555162,0.824103);
      rgb(285pt)=(0.0644072,0.556439,0.824001);
      rgb(286pt)=(0.0635892,0.557715,0.823899);
      rgb(287pt)=(0.0626703,0.558991,0.823848);
      rgb(288pt)=(0.0617514,0.560268,0.823797);
      rgb(289pt)=(0.0608324,0.561544,0.823746);
      rgb(290pt)=(0.0599087,0.56282,0.823693);
      rgb(291pt)=(0.0589387,0.564096,0.823616);
      rgb(292pt)=(0.0579688,0.565373,0.82354);
      rgb(293pt)=(0.0569988,0.566649,0.823463);
      rgb(294pt)=(0.0560243,0.567925,0.823386);
      rgb(295pt)=(0.0550288,0.569202,0.82331);
      rgb(296pt)=(0.0540333,0.570478,0.823233);
      rgb(297pt)=(0.0530378,0.571754,0.823157);
      rgb(298pt)=(0.0520423,0.57303,0.82308);
      rgb(299pt)=(0.0510468,0.574307,0.823004);
      rgb(300pt)=(0.0500514,0.575583,0.822927);
      rgb(301pt)=(0.0490559,0.576859,0.82285);
      rgb(302pt)=(0.0480604,0.578127,0.822756);
      rgb(303pt)=(0.0470649,0.579377,0.822629);
      rgb(304pt)=(0.0460694,0.580628,0.822501);
      rgb(305pt)=(0.0450739,0.581879,0.822374);
      rgb(306pt)=(0.0441,0.583119,0.822235);
      rgb(307pt)=(0.0431556,0.584344,0.822082);
      rgb(308pt)=(0.0422111,0.585569,0.821929);
      rgb(309pt)=(0.0412667,0.586795,0.821776);
      rgb(310pt)=(0.0403351,0.58802,0.821597);
      rgb(311pt)=(0.0394162,0.589245,0.821392);
      rgb(312pt)=(0.0384973,0.59047,0.821188);
      rgb(313pt)=(0.0375784,0.591695,0.820984);
      rgb(314pt)=(0.0367495,0.592891,0.820735);
      rgb(315pt)=(0.0359838,0.594065,0.820454);
      rgb(316pt)=(0.035218,0.595239,0.820173);
      rgb(317pt)=(0.0344523,0.596413,0.819892);
      rgb(318pt)=(0.0337721,0.597553,0.819595);
      rgb(319pt)=(0.0331339,0.598676,0.819288);
      rgb(320pt)=(0.0324958,0.599799,0.818982);
      rgb(321pt)=(0.0318577,0.600923,0.818676);
      rgb(322pt)=(0.0312964,0.602026,0.818312);
      rgb(323pt)=(0.0307604,0.603124,0.817929);
      rgb(324pt)=(0.0302243,0.604222,0.817546);
      rgb(325pt)=(0.0296883,0.605319,0.817163);
      rgb(326pt)=(0.0292375,0.606395,0.816738);
      rgb(327pt)=(0.0288036,0.607468,0.816304);
      rgb(328pt)=(0.0283697,0.60854,0.81587);
      rgb(329pt)=(0.0279357,0.609612,0.815436);
      rgb(330pt)=(0.0275721,0.610637,0.814955);
      rgb(331pt)=(0.0272147,0.611658,0.81447);
      rgb(332pt)=(0.0268574,0.612679,0.813985);
      rgb(333pt)=(0.0265,0.6137,0.8135);
      rgb(334pt)=(0.0262447,0.614695,0.812964);
      rgb(335pt)=(0.0259895,0.615691,0.812428);
      rgb(336pt)=(0.0257342,0.616686,0.811892);
      rgb(337pt)=(0.0254853,0.61768,0.811352);
      rgb(338pt)=(0.0253066,0.61865,0.810765);
      rgb(339pt)=(0.0251279,0.61962,0.810177);
      rgb(340pt)=(0.0249492,0.62059,0.80959);
      rgb(341pt)=(0.024779,0.621551,0.808995);
      rgb(342pt)=(0.0246514,0.62247,0.808357);
      rgb(343pt)=(0.0245237,0.623389,0.807719);
      rgb(344pt)=(0.0243961,0.624308,0.80708);
      rgb(345pt)=(0.0242748,0.625221,0.80643);
      rgb(346pt)=(0.0241727,0.626114,0.805741);
      rgb(347pt)=(0.0240706,0.627008,0.805051);
      rgb(348pt)=(0.0239685,0.627901,0.804362);
      rgb(349pt)=(0.0238832,0.628786,0.803656);
      rgb(350pt)=(0.0238321,0.629654,0.802916);
      rgb(351pt)=(0.0237811,0.630522,0.802176);
      rgb(352pt)=(0.02373,0.631389,0.801435);
      rgb(353pt)=(0.023679,0.632247,0.800685);
      rgb(354pt)=(0.0236279,0.633089,0.799919);
      rgb(355pt)=(0.0235769,0.633932,0.799153);
      rgb(356pt)=(0.0235258,0.634774,0.798387);
      rgb(357pt)=(0.0234748,0.635604,0.797596);
      rgb(358pt)=(0.0234237,0.63642,0.79678);
      rgb(359pt)=(0.0233727,0.637237,0.795963);
      rgb(360pt)=(0.0233216,0.638054,0.795146);
      rgb(361pt)=(0.0232706,0.638856,0.794329);
      rgb(362pt)=(0.0232195,0.639647,0.793512);
      rgb(363pt)=(0.0231685,0.640439,0.792695);
      rgb(364pt)=(0.0231174,0.64123,0.791879);
      rgb(365pt)=(0.0230832,0.642005,0.791011);
      rgb(366pt)=(0.0230577,0.64277,0.790118);
      rgb(367pt)=(0.0230321,0.643536,0.789225);
      rgb(368pt)=(0.0230066,0.644302,0.788331);
      rgb(369pt)=(0.0229811,0.645049,0.787438);
      rgb(370pt)=(0.0229556,0.645789,0.786544);
      rgb(371pt)=(0.02293,0.646529,0.785651);
      rgb(372pt)=(0.0229045,0.647269,0.784758);
      rgb(373pt)=(0.022858,0.64801,0.783843);
      rgb(374pt)=(0.0228069,0.64875,0.782924);
      rgb(375pt)=(0.0227559,0.64949,0.782005);
      rgb(376pt)=(0.0227048,0.65023,0.781086);
      rgb(377pt)=(0.0227,0.650947,0.780144);
      rgb(378pt)=(0.0227,0.651662,0.7792);
      rgb(379pt)=(0.0227,0.652377,0.778256);
      rgb(380pt)=(0.0227,0.653092,0.777311);
      rgb(381pt)=(0.0228261,0.653781,0.776341);
      rgb(382pt)=(0.0229538,0.65447,0.775371);
      rgb(383pt)=(0.0230814,0.655159,0.774402);
      rgb(384pt)=(0.0232108,0.655849,0.77343);
      rgb(385pt)=(0.023364,0.656538,0.772434);
      rgb(386pt)=(0.0235171,0.657227,0.771439);
      rgb(387pt)=(0.0236703,0.657916,0.770443);
      rgb(388pt)=(0.0238312,0.658602,0.769444);
      rgb(389pt)=(0.0240354,0.659265,0.768423);
      rgb(390pt)=(0.0242396,0.659929,0.767402);
      rgb(391pt)=(0.0244438,0.660592,0.766381);
      rgb(392pt)=(0.0247021,0.661256,0.765354);
      rgb(393pt)=(0.025136,0.66192,0.764307);
      rgb(394pt)=(0.02557,0.662583,0.763261);
      rgb(395pt)=(0.0260039,0.663247,0.762214);
      rgb(396pt)=(0.0264541,0.663911,0.761168);
      rgb(397pt)=(0.026939,0.664574,0.760121);
      rgb(398pt)=(0.027424,0.665238,0.759074);
      rgb(399pt)=(0.027909,0.665902,0.758028);
      rgb(400pt)=(0.028445,0.666555,0.756971);
      rgb(401pt)=(0.0290577,0.667193,0.755899);
      rgb(402pt)=(0.0296703,0.667832,0.754827);
      rgb(403pt)=(0.0302829,0.66847,0.753755);
      rgb(404pt)=(0.030994,0.669095,0.752683);
      rgb(405pt)=(0.0318108,0.669708,0.751611);
      rgb(406pt)=(0.0326276,0.670321,0.750539);
      rgb(407pt)=(0.0334444,0.670933,0.749467);
      rgb(408pt)=(0.0343045,0.67156,0.748366);
      rgb(409pt)=(0.0351979,0.672198,0.747243);
      rgb(410pt)=(0.0360913,0.672837,0.74612);
      rgb(411pt)=(0.0369847,0.673475,0.744996);
      rgb(412pt)=(0.0380432,0.674096,0.743873);
      rgb(413pt)=(0.0391919,0.674709,0.74275);
      rgb(414pt)=(0.0403405,0.675322,0.741627);
      rgb(415pt)=(0.0414892,0.675934,0.740504);
      rgb(416pt)=(0.0427123,0.676528,0.739381);
      rgb(417pt)=(0.0439631,0.677115,0.738258);
      rgb(418pt)=(0.0452138,0.677702,0.737135);
      rgb(419pt)=(0.0464646,0.678289,0.736011);
      rgb(420pt)=(0.0477153,0.678897,0.734868);
      rgb(421pt)=(0.0489661,0.67951,0.733719);
      rgb(422pt)=(0.0502168,0.680123,0.73257);
      rgb(423pt)=(0.0514676,0.680735,0.731422);
      rgb(424pt)=(0.0529237,0.681325,0.73025);
      rgb(425pt)=(0.0544042,0.681912,0.729076);
      rgb(426pt)=(0.0558847,0.682499,0.727902);
      rgb(427pt)=(0.0573652,0.683086,0.726728);
      rgb(428pt)=(0.0587709,0.683673,0.725553);
      rgb(429pt)=(0.0601748,0.68426,0.724379);
      rgb(430pt)=(0.0615787,0.684847,0.723205);
      rgb(431pt)=(0.0629946,0.685435,0.722028);
      rgb(432pt)=(0.0646027,0.686022,0.720803);
      rgb(433pt)=(0.0662108,0.686609,0.719577);
      rgb(434pt)=(0.0678189,0.687196,0.718352);
      rgb(435pt)=(0.069427,0.687779,0.717131);
      rgb(436pt)=(0.0710351,0.688341,0.715931);
      rgb(437pt)=(0.0726432,0.688902,0.714731);
      rgb(438pt)=(0.0742514,0.689464,0.713532);
      rgb(439pt)=(0.0758709,0.690026,0.712326);
      rgb(440pt)=(0.07753,0.690587,0.711101);
      rgb(441pt)=(0.0791892,0.691149,0.709876);
      rgb(442pt)=(0.0808483,0.69171,0.70865);
      rgb(443pt)=(0.0825387,0.692272,0.707417);
      rgb(444pt)=(0.0843,0.692833,0.706167);
      rgb(445pt)=(0.0860613,0.693395,0.704916);
      rgb(446pt)=(0.0878225,0.693956,0.703665);
      rgb(447pt)=(0.089564,0.694518,0.702405);
      rgb(448pt)=(0.0912742,0.69508,0.701128);
      rgb(449pt)=(0.0929844,0.695641,0.699852);
      rgb(450pt)=(0.0946946,0.696203,0.698576);
      rgb(451pt)=(0.0965009,0.696752,0.697299);
      rgb(452pt)=(0.0984153,0.697288,0.696023);
      rgb(453pt)=(0.10033,0.697824,0.694747);
      rgb(454pt)=(0.102244,0.69836,0.693471);
      rgb(455pt)=(0.10413,0.698896,0.69218);
      rgb(456pt)=(0.105994,0.699432,0.690878);
      rgb(457pt)=(0.107857,0.699968,0.689577);
      rgb(458pt)=(0.10972,0.700505,0.688275);
      rgb(459pt)=(0.111632,0.701041,0.686973);
      rgb(460pt)=(0.113572,0.701577,0.685671);
      rgb(461pt)=(0.115512,0.702113,0.684369);
      rgb(462pt)=(0.117452,0.702649,0.683068);
      rgb(463pt)=(0.119429,0.703185,0.681747);
      rgb(464pt)=(0.12142,0.703721,0.68042);
      rgb(465pt)=(0.123411,0.704257,0.679093);
      rgb(466pt)=(0.125402,0.704793,0.677765);
      rgb(467pt)=(0.127372,0.705308,0.676438);
      rgb(468pt)=(0.129338,0.705819,0.675111);
      rgb(469pt)=(0.131303,0.706329,0.673783);
      rgb(470pt)=(0.133269,0.70684,0.672456);
      rgb(471pt)=(0.135369,0.70735,0.671084);
      rgb(472pt)=(0.137488,0.707861,0.669705);
      rgb(473pt)=(0.139607,0.708371,0.668327);
      rgb(474pt)=(0.141725,0.708882,0.666949);
      rgb(475pt)=(0.143795,0.709392,0.665595);
      rgb(476pt)=(0.145862,0.709903,0.664242);
      rgb(477pt)=(0.14793,0.710414,0.662889);
      rgb(478pt)=(0.150003,0.710924,0.661534);
      rgb(479pt)=(0.152198,0.711435,0.66013);
      rgb(480pt)=(0.154394,0.711945,0.658726);
      rgb(481pt)=(0.156589,0.712456,0.657322);
      rgb(482pt)=(0.158784,0.712963,0.655922);
      rgb(483pt)=(0.160979,0.713448,0.654543);
      rgb(484pt)=(0.163174,0.713933,0.653165);
      rgb(485pt)=(0.16537,0.714418,0.651786);
      rgb(486pt)=(0.16757,0.714908,0.650397);
      rgb(487pt)=(0.169791,0.715419,0.648968);
      rgb(488pt)=(0.172012,0.715929,0.647538);
      rgb(489pt)=(0.174232,0.71644,0.646109);
      rgb(490pt)=(0.176483,0.716935,0.64468);
      rgb(491pt)=(0.178806,0.717395,0.64325);
      rgb(492pt)=(0.181129,0.717854,0.641821);
      rgb(493pt)=(0.183452,0.718314,0.640391);
      rgb(494pt)=(0.185755,0.718783,0.638952);
      rgb(495pt)=(0.188027,0.719268,0.637497);
      rgb(496pt)=(0.190299,0.719753,0.636042);
      rgb(497pt)=(0.192571,0.720238,0.634587);
      rgb(498pt)=(0.194913,0.720711,0.633132);
      rgb(499pt)=(0.197338,0.72117,0.631677);
      rgb(500pt)=(0.199762,0.72163,0.630223);
      rgb(501pt)=(0.202187,0.722089,0.628768);
      rgb(502pt)=(0.204612,0.722549,0.627299);
      rgb(503pt)=(0.207037,0.723008,0.625818);
      rgb(504pt)=(0.209462,0.723468,0.624338);
      rgb(505pt)=(0.211887,0.723927,0.622857);
      rgb(506pt)=(0.214328,0.724386,0.621377);
      rgb(507pt)=(0.216778,0.724846,0.619896);
      rgb(508pt)=(0.219229,0.725305,0.618416);
      rgb(509pt)=(0.221679,0.725765,0.616935);
      rgb(510pt)=(0.224202,0.726188,0.615455);
      rgb(511pt)=(0.226754,0.726597,0.613974);
      rgb(512pt)=(0.229307,0.727005,0.612494);
      rgb(513pt)=(0.231859,0.727414,0.611014);
      rgb(514pt)=(0.234392,0.727842,0.609513);
      rgb(515pt)=(0.236919,0.728276,0.608007);
      rgb(516pt)=(0.239446,0.72871,0.606501);
      rgb(517pt)=(0.241973,0.729144,0.604995);
      rgb(518pt)=(0.244611,0.729556,0.603467);
      rgb(519pt)=(0.247266,0.729964,0.601935);
      rgb(520pt)=(0.24992,0.730372,0.600404);
      rgb(521pt)=(0.252575,0.730781,0.598872);
      rgb(522pt)=(0.25523,0.731189,0.597365);
      rgb(523pt)=(0.257884,0.731598,0.595859);
      rgb(524pt)=(0.260539,0.732006,0.594353);
      rgb(525pt)=(0.263194,0.732414,0.592846);
      rgb(526pt)=(0.265848,0.732796,0.591314);
      rgb(527pt)=(0.268503,0.733179,0.589783);
      rgb(528pt)=(0.271158,0.733562,0.588251);
      rgb(529pt)=(0.27383,0.733945,0.58672);
      rgb(530pt)=(0.276638,0.734328,0.585188);
      rgb(531pt)=(0.279446,0.734711,0.583657);
      rgb(532pt)=(0.282254,0.735094,0.582125);
      rgb(533pt)=(0.285051,0.735471,0.580594);
      rgb(534pt)=(0.287808,0.735829,0.579062);
      rgb(535pt)=(0.290565,0.736186,0.577531);
      rgb(536pt)=(0.293322,0.736544,0.575999);
      rgb(537pt)=(0.2961,0.736894,0.574468);
      rgb(538pt)=(0.298933,0.737226,0.572936);
      rgb(539pt)=(0.301767,0.737557,0.571405);
      rgb(540pt)=(0.3046,0.737889,0.569873);
      rgb(541pt)=(0.307452,0.738221,0.568351);
      rgb(542pt)=(0.310336,0.738553,0.566845);
      rgb(543pt)=(0.313221,0.738885,0.565339);
      rgb(544pt)=(0.316105,0.739217,0.563833);
      rgb(545pt)=(0.318978,0.739537,0.562315);
      rgb(546pt)=(0.321837,0.739843,0.560784);
      rgb(547pt)=(0.324696,0.74015,0.559252);
      rgb(548pt)=(0.327555,0.740456,0.557721);
      rgb(549pt)=(0.330468,0.740749,0.556216);
      rgb(550pt)=(0.333429,0.741029,0.554736);
      rgb(551pt)=(0.336389,0.74131,0.553255);
      rgb(552pt)=(0.33935,0.741591,0.551775);
      rgb(553pt)=(0.342296,0.741872,0.550279);
      rgb(554pt)=(0.345231,0.742153,0.548773);
      rgb(555pt)=(0.348167,0.742433,0.547267);
      rgb(556pt)=(0.351102,0.742714,0.545761);
      rgb(557pt)=(0.354038,0.742977,0.54429);
      rgb(558pt)=(0.356973,0.743232,0.542835);
      rgb(559pt)=(0.359908,0.743488,0.54138);
      rgb(560pt)=(0.362844,0.743743,0.539925);
      rgb(561pt)=(0.365839,0.743959,0.53847);
      rgb(562pt)=(0.368851,0.744163,0.537015);
      rgb(563pt)=(0.371863,0.744367,0.53556);
      rgb(564pt)=(0.374875,0.744571,0.534105);
      rgb(565pt)=(0.377843,0.744775,0.532672);
      rgb(566pt)=(0.380804,0.74498,0.531243);
      rgb(567pt)=(0.383765,0.745184,0.529814);
      rgb(568pt)=(0.386726,0.745388,0.528384);
      rgb(569pt)=(0.389711,0.745568,0.527003);
      rgb(570pt)=(0.392697,0.745747,0.525624);
      rgb(571pt)=(0.395684,0.745926,0.524246);
      rgb(572pt)=(0.39867,0.746104,0.522868);
      rgb(573pt)=(0.401657,0.746257,0.521489);
      rgb(574pt)=(0.404643,0.74641,0.520111);
      rgb(575pt)=(0.40763,0.746563,0.518732);
      rgb(576pt)=(0.410611,0.746716,0.517359);
      rgb(577pt)=(0.413546,0.746869,0.516032);
      rgb(578pt)=(0.416482,0.747023,0.514705);
      rgb(579pt)=(0.419417,0.747176,0.513377);
      rgb(580pt)=(0.422357,0.747319,0.512055);
      rgb(581pt)=(0.425318,0.747421,0.510753);
      rgb(582pt)=(0.428279,0.747523,0.509451);
      rgb(583pt)=(0.43124,0.747626,0.50815);
      rgb(584pt)=(0.43418,0.747735,0.506848);
      rgb(585pt)=(0.437065,0.747862,0.505546);
      rgb(586pt)=(0.439949,0.74799,0.504244);
      rgb(587pt)=(0.442834,0.748117,0.502942);
      rgb(588pt)=(0.445727,0.748227,0.501659);
      rgb(589pt)=(0.448637,0.748304,0.500408);
      rgb(590pt)=(0.451547,0.74838,0.499157);
      rgb(591pt)=(0.454457,0.748457,0.497906);
      rgb(592pt)=(0.457333,0.748522,0.496667);
      rgb(593pt)=(0.460167,0.748573,0.495441);
      rgb(594pt)=(0.463,0.748624,0.494216);
      rgb(595pt)=(0.465833,0.748675,0.492991);
      rgb(596pt)=(0.468667,0.748726,0.491779);
      rgb(597pt)=(0.4715,0.748777,0.490579);
      rgb(598pt)=(0.474333,0.748829,0.48938);
      rgb(599pt)=(0.477167,0.74888,0.48818);
      rgb(600pt)=(0.479969,0.748931,0.486995);
      rgb(601pt)=(0.482752,0.748982,0.485821);
      rgb(602pt)=(0.485534,0.749033,0.484647);
      rgb(603pt)=(0.488316,0.749084,0.483473);
      rgb(604pt)=(0.491081,0.7491,0.482316);
      rgb(605pt)=(0.493838,0.7491,0.481168);
      rgb(606pt)=(0.496595,0.7491,0.480019);
      rgb(607pt)=(0.499351,0.7491,0.47887);
      rgb(608pt)=(0.502069,0.74912,0.477722);
      rgb(609pt)=(0.504775,0.749145,0.476573);
      rgb(610pt)=(0.50748,0.749171,0.475424);
      rgb(611pt)=(0.510186,0.749196,0.474276);
      rgb(612pt)=(0.512892,0.7492,0.47317);
      rgb(613pt)=(0.515598,0.7492,0.472073);
      rgb(614pt)=(0.518303,0.7492,0.470975);
      rgb(615pt)=(0.521009,0.7492,0.469877);
      rgb(616pt)=(0.523644,0.749176,0.46878);
      rgb(617pt)=(0.526273,0.749151,0.467682);
      rgb(618pt)=(0.528902,0.749125,0.466585);
      rgb(619pt)=(0.531531,0.7491,0.465487);
      rgb(620pt)=(0.53416,0.749074,0.464415);
      rgb(621pt)=(0.536789,0.749049,0.463343);
      rgb(622pt)=(0.539418,0.749023,0.462271);
      rgb(623pt)=(0.542043,0.748998,0.461199);
      rgb(624pt)=(0.544621,0.748972,0.460127);
      rgb(625pt)=(0.547199,0.748947,0.459055);
      rgb(626pt)=(0.549777,0.748921,0.457983);
      rgb(627pt)=(0.55235,0.748891,0.45692);
      rgb(628pt)=(0.554903,0.74884,0.455899);
      rgb(629pt)=(0.557456,0.748789,0.454878);
      rgb(630pt)=(0.560008,0.748738,0.453857);
      rgb(631pt)=(0.562554,0.748687,0.452829);
      rgb(632pt)=(0.565081,0.748636,0.451783);
      rgb(633pt)=(0.567608,0.748585,0.450736);
      rgb(634pt)=(0.570135,0.748534,0.449689);
      rgb(635pt)=(0.572653,0.748474,0.44866);
      rgb(636pt)=(0.575155,0.748397,0.447665);
      rgb(637pt)=(0.577656,0.748321,0.446669);
      rgb(638pt)=(0.580158,0.748244,0.445674);
      rgb(639pt)=(0.582649,0.748168,0.444678);
      rgb(640pt)=(0.585125,0.748091,0.443683);
      rgb(641pt)=(0.587601,0.748014,0.442687);
      rgb(642pt)=(0.590077,0.747938,0.441692);
      rgb(643pt)=(0.59254,0.747861,0.440709);
      rgb(644pt)=(0.59499,0.747785,0.439739);
      rgb(645pt)=(0.597441,0.747708,0.438769);
      rgb(646pt)=(0.599891,0.747632,0.437799);
      rgb(647pt)=(0.602311,0.747555,0.436814);
      rgb(648pt)=(0.604711,0.747478,0.435819);
      rgb(649pt)=(0.60711,0.747402,0.434823);
      rgb(650pt)=(0.60951,0.747325,0.433828);
      rgb(651pt)=(0.611909,0.747232,0.432867);
      rgb(652pt)=(0.614308,0.747129,0.431922);
      rgb(653pt)=(0.616708,0.747027,0.430978);
      rgb(654pt)=(0.619107,0.746925,0.430033);
      rgb(655pt)=(0.621487,0.746823,0.429089);
      rgb(656pt)=(0.623861,0.746721,0.428144);
      rgb(657pt)=(0.626235,0.746619,0.4272);
      rgb(658pt)=(0.628609,0.746517,0.426256);
      rgb(659pt)=(0.630962,0.746393,0.425311);
      rgb(660pt)=(0.63331,0.746266,0.424367);
      rgb(661pt)=(0.635658,0.746138,0.423422);
      rgb(662pt)=(0.638007,0.746011,0.422478);
      rgb(663pt)=(0.640332,0.745906,0.421557);
      rgb(664pt)=(0.642654,0.745804,0.420638);
      rgb(665pt)=(0.644977,0.745702,0.419719);
      rgb(666pt)=(0.6473,0.7456,0.4188);
      rgb(667pt)=(0.649623,0.745472,0.417881);
      rgb(668pt)=(0.651946,0.745345,0.416962);
      rgb(669pt)=(0.654268,0.745217,0.416043);
      rgb(670pt)=(0.656587,0.745089,0.415124);
      rgb(671pt)=(0.658859,0.744962,0.414205);
      rgb(672pt)=(0.661131,0.744834,0.413286);
      rgb(673pt)=(0.663402,0.744707,0.412368);
      rgb(674pt)=(0.665674,0.744579,0.411453);
      rgb(675pt)=(0.667946,0.744451,0.410559);
      rgb(676pt)=(0.670218,0.744324,0.409666);
      rgb(677pt)=(0.672489,0.744196,0.408773);
      rgb(678pt)=(0.674755,0.744062,0.407879);
      rgb(679pt)=(0.677001,0.743909,0.406986);
      rgb(680pt)=(0.679247,0.743756,0.406092);
      rgb(681pt)=(0.681494,0.743603,0.405199);
      rgb(682pt)=(0.68374,0.743458,0.404306);
      rgb(683pt)=(0.685986,0.74333,0.403412);
      rgb(684pt)=(0.688232,0.743203,0.402519);
      rgb(685pt)=(0.690479,0.743075,0.401626);
      rgb(686pt)=(0.692704,0.742937,0.400732);
      rgb(687pt)=(0.694899,0.742784,0.399839);
      rgb(688pt)=(0.697094,0.742631,0.398945);
      rgb(689pt)=(0.699289,0.742477,0.398052);
      rgb(690pt)=(0.701497,0.742324,0.397171);
      rgb(691pt)=(0.703718,0.742171,0.396303);
      rgb(692pt)=(0.705939,0.742018,0.395435);
      rgb(693pt)=(0.708159,0.741865,0.394568);
      rgb(694pt)=(0.710351,0.741712,0.3937);
      rgb(695pt)=(0.71252,0.741559,0.392832);
      rgb(696pt)=(0.71469,0.741405,0.391964);
      rgb(697pt)=(0.71686,0.741252,0.391096);
      rgb(698pt)=(0.719029,0.741082,0.390228);
      rgb(699pt)=(0.721199,0.740904,0.38936);
      rgb(700pt)=(0.723369,0.740725,0.388492);
      rgb(701pt)=(0.725538,0.740546,0.387625);
      rgb(702pt)=(0.727708,0.740386,0.386757);
      rgb(703pt)=(0.729878,0.740233,0.385889);
      rgb(704pt)=(0.732047,0.74008,0.385021);
      rgb(705pt)=(0.734217,0.739927,0.384153);
      rgb(706pt)=(0.736366,0.739753,0.383285);
      rgb(707pt)=(0.73851,0.739574,0.382417);
      rgb(708pt)=(0.740654,0.739395,0.38155);
      rgb(709pt)=(0.742798,0.739217,0.380682);
      rgb(710pt)=(0.744919,0.739038,0.379837);
      rgb(711pt)=(0.747038,0.738859,0.378995);
      rgb(712pt)=(0.749156,0.738681,0.378152);
      rgb(713pt)=(0.751275,0.738502,0.37731);
      rgb(714pt)=(0.753394,0.738323,0.376442);
      rgb(715pt)=(0.755512,0.738145,0.375574);
      rgb(716pt)=(0.757631,0.737966,0.374707);
      rgb(717pt)=(0.75975,0.737789,0.373841);
      rgb(718pt)=(0.761868,0.737636,0.372998);
      rgb(719pt)=(0.763987,0.737483,0.372156);
      rgb(720pt)=(0.766105,0.73733,0.371314);
      rgb(721pt)=(0.76822,0.737169,0.370471);
      rgb(722pt)=(0.770313,0.736965,0.369629);
      rgb(723pt)=(0.772406,0.73676,0.368786);
      rgb(724pt)=(0.774499,0.736556,0.367944);
      rgb(725pt)=(0.776592,0.736358,0.367096);
      rgb(726pt)=(0.778686,0.736179,0.366228);
      rgb(727pt)=(0.780779,0.736001,0.36536);
      rgb(728pt)=(0.782872,0.735822,0.364492);
      rgb(729pt)=(0.784957,0.735643,0.363632);
      rgb(730pt)=(0.787024,0.735465,0.36279);
      rgb(731pt)=(0.789092,0.735286,0.361948);
      rgb(732pt)=(0.791159,0.735107,0.361105);
      rgb(733pt)=(0.793227,0.734929,0.360263);
      rgb(734pt)=(0.795295,0.73475,0.359421);
      rgb(735pt)=(0.797362,0.734571,0.358578);
      rgb(736pt)=(0.79943,0.734392,0.357736);
      rgb(737pt)=(0.801485,0.734214,0.356881);
      rgb(738pt)=(0.803527,0.734035,0.356014);
      rgb(739pt)=(0.805569,0.733856,0.355146);
      rgb(740pt)=(0.807611,0.733678,0.354278);
      rgb(741pt)=(0.809668,0.733499,0.353424);
      rgb(742pt)=(0.811735,0.73332,0.352582);
      rgb(743pt)=(0.813803,0.733142,0.35174);
      rgb(744pt)=(0.81587,0.732963,0.350897);
      rgb(745pt)=(0.817921,0.732784,0.350038);
      rgb(746pt)=(0.819963,0.732606,0.349171);
      rgb(747pt)=(0.822005,0.732427,0.348303);
      rgb(748pt)=(0.824047,0.732248,0.347435);
      rgb(749pt)=(0.826071,0.73207,0.346567);
      rgb(750pt)=(0.828087,0.731891,0.345699);
      rgb(751pt)=(0.830104,0.731712,0.344831);
      rgb(752pt)=(0.83212,0.731534,0.343963);
      rgb(753pt)=(0.834158,0.731355,0.343095);
      rgb(754pt)=(0.8362,0.731176,0.342228);
      rgb(755pt)=(0.838242,0.730998,0.34136);
      rgb(756pt)=(0.840284,0.730819,0.340492);
      rgb(757pt)=(0.842303,0.73064,0.339624);
      rgb(758pt)=(0.84432,0.730462,0.338756);
      rgb(759pt)=(0.846336,0.730283,0.337888);
      rgb(760pt)=(0.848353,0.730104,0.33702);
      rgb(761pt)=(0.850369,0.729926,0.336153);
      rgb(762pt)=(0.852386,0.729747,0.335285);
      rgb(763pt)=(0.854402,0.729568,0.334417);
      rgb(764pt)=(0.856419,0.729391,0.333546);
      rgb(765pt)=(0.858435,0.729238,0.332627);
      rgb(766pt)=(0.860452,0.729085,0.331708);
      rgb(767pt)=(0.862468,0.728932,0.330789);
      rgb(768pt)=(0.864481,0.728778,0.329874);
      rgb(769pt)=(0.866472,0.728625,0.32898);
      rgb(770pt)=(0.868463,0.728472,0.328087);
      rgb(771pt)=(0.870454,0.728319,0.327194);
      rgb(772pt)=(0.872445,0.728166,0.326295);
      rgb(773pt)=(0.874436,0.728013,0.325376);
      rgb(774pt)=(0.876427,0.727859,0.324457);
      rgb(775pt)=(0.878418,0.727706,0.323538);
      rgb(776pt)=(0.880417,0.727561,0.322619);
      rgb(777pt)=(0.882433,0.727433,0.3217);
      rgb(778pt)=(0.88445,0.727306,0.320781);
      rgb(779pt)=(0.886466,0.727178,0.319862);
      rgb(780pt)=(0.888463,0.72705,0.318933);
      rgb(781pt)=(0.890429,0.726923,0.317989);
      rgb(782pt)=(0.892394,0.726795,0.317044);
      rgb(783pt)=(0.894359,0.726668,0.3161);
      rgb(784pt)=(0.896337,0.726552,0.315132);
      rgb(785pt)=(0.898328,0.72645,0.314136);
      rgb(786pt)=(0.900319,0.726348,0.313141);
      rgb(787pt)=(0.90231,0.726246,0.312145);
      rgb(788pt)=(0.904301,0.726158,0.31115);
      rgb(789pt)=(0.906292,0.726081,0.310154);
      rgb(790pt)=(0.908283,0.726005,0.309159);
      rgb(791pt)=(0.910274,0.725928,0.308163);
      rgb(792pt)=(0.912249,0.725851,0.307151);
      rgb(793pt)=(0.914214,0.725775,0.30613);
      rgb(794pt)=(0.91618,0.725698,0.305109);
      rgb(795pt)=(0.918145,0.725622,0.304088);
      rgb(796pt)=(0.920111,0.7256,0.303031);
      rgb(797pt)=(0.922076,0.7256,0.301959);
      rgb(798pt)=(0.924041,0.7256,0.300886);
      rgb(799pt)=(0.926007,0.7256,0.299814);
      rgb(800pt)=(0.927972,0.7256,0.298722);
      rgb(801pt)=(0.929938,0.7256,0.297624);
      rgb(802pt)=(0.931903,0.7256,0.296527);
      rgb(803pt)=(0.933869,0.7256,0.295429);
      rgb(804pt)=(0.935812,0.725668,0.294264);
      rgb(805pt)=(0.937752,0.725744,0.29309);
      rgb(806pt)=(0.939692,0.725821,0.291916);
      rgb(807pt)=(0.941632,0.725897,0.290741);
      rgb(808pt)=(0.943571,0.726023,0.289518);
      rgb(809pt)=(0.945511,0.726151,0.288293);
      rgb(810pt)=(0.947451,0.726278,0.287068);
      rgb(811pt)=(0.949389,0.726411,0.285839);
      rgb(812pt)=(0.951278,0.726641,0.284537);
      rgb(813pt)=(0.953167,0.72687,0.283235);
      rgb(814pt)=(0.955056,0.7271,0.281933);
      rgb(815pt)=(0.956938,0.72734,0.280622);
      rgb(816pt)=(0.958776,0.727646,0.279243);
      rgb(817pt)=(0.960614,0.727952,0.277865);
      rgb(818pt)=(0.962451,0.728259,0.276486);
      rgb(819pt)=(0.964273,0.728597,0.275086);
      rgb(820pt)=(0.966034,0.729057,0.273606);
      rgb(821pt)=(0.967795,0.729516,0.272126);
      rgb(822pt)=(0.969557,0.729976,0.270645);
      rgb(823pt)=(0.971288,0.730473,0.269135);
      rgb(824pt)=(0.972947,0.73106,0.267552);
      rgb(825pt)=(0.974606,0.731647,0.265969);
      rgb(826pt)=(0.976265,0.732234,0.264387);
      rgb(827pt)=(0.977857,0.732879,0.262785);
      rgb(828pt)=(0.979338,0.733619,0.261151);
      rgb(829pt)=(0.980818,0.734359,0.259518);
      rgb(830pt)=(0.982299,0.735099,0.257884);
      rgb(831pt)=(0.983697,0.73591,0.256227);
      rgb(832pt)=(0.984999,0.736803,0.254542);
      rgb(833pt)=(0.986301,0.737697,0.252858);
      rgb(834pt)=(0.987603,0.73859,0.251173);
      rgb(835pt)=(0.988753,0.739566,0.249474);
      rgb(836pt)=(0.989774,0.740613,0.247764);
      rgb(837pt)=(0.990795,0.741659,0.246054);
      rgb(838pt)=(0.991816,0.742706,0.244344);
      rgb(839pt)=(0.992677,0.743816,0.242681);
      rgb(840pt)=(0.993443,0.744965,0.241048);
      rgb(841pt)=(0.994209,0.746114,0.239414);
      rgb(842pt)=(0.994975,0.747262,0.23778);
      rgb(843pt)=(0.995578,0.748465,0.236165);
      rgb(844pt)=(0.996114,0.74969,0.234557);
      rgb(845pt)=(0.99665,0.750915,0.232949);
      rgb(846pt)=(0.997186,0.752141,0.231341);
      rgb(847pt)=(0.997562,0.753386,0.229813);
      rgb(848pt)=(0.997893,0.754637,0.228307);
      rgb(849pt)=(0.998225,0.755887,0.226801);
      rgb(850pt)=(0.998557,0.757138,0.225295);
      rgb(851pt)=(0.998711,0.758433,0.223856);
      rgb(852pt)=(0.998839,0.759735,0.222426);
      rgb(853pt)=(0.998966,0.761037,0.220997);
      rgb(854pt)=(0.999094,0.762339,0.219567);
      rgb(855pt)=(0.999076,0.763641,0.218186);
      rgb(856pt)=(0.99905,0.764942,0.216808);
      rgb(857pt)=(0.999025,0.766244,0.21543);
      rgb(858pt)=(0.998995,0.767546,0.214054);
      rgb(859pt)=(0.998868,0.768848,0.212752);
      rgb(860pt)=(0.99874,0.77015,0.21145);
      rgb(861pt)=(0.998613,0.771451,0.210149);
      rgb(862pt)=(0.998473,0.772756,0.208856);
      rgb(863pt)=(0.998243,0.774083,0.207631);
      rgb(864pt)=(0.998014,0.775411,0.206405);
      rgb(865pt)=(0.997784,0.776738,0.20518);
      rgb(866pt)=(0.997539,0.77806,0.20396);
      rgb(867pt)=(0.997232,0.779362,0.20276);
      rgb(868pt)=(0.996926,0.780664,0.201561);
      rgb(869pt)=(0.99662,0.781966,0.200361);
      rgb(870pt)=(0.996299,0.783268,0.199168);
      rgb(871pt)=(0.995942,0.784569,0.197994);
      rgb(872pt)=(0.995584,0.785871,0.19682);
      rgb(873pt)=(0.995227,0.787173,0.195646);
      rgb(874pt)=(0.994842,0.788475,0.19449);
      rgb(875pt)=(0.994408,0.789777,0.193367);
      rgb(876pt)=(0.993974,0.791078,0.192244);
      rgb(877pt)=(0.99354,0.79238,0.191121);
      rgb(878pt)=(0.993083,0.793671,0.190021);
      rgb(879pt)=(0.992598,0.794947,0.188949);
      rgb(880pt)=(0.992113,0.796223,0.187877);
      rgb(881pt)=(0.991628,0.797499,0.186805);
      rgb(882pt)=(0.99113,0.798789,0.185732);
      rgb(883pt)=(0.990619,0.800091,0.18466);
      rgb(884pt)=(0.990109,0.801393,0.183588);
      rgb(885pt)=(0.989598,0.802695,0.182516);
      rgb(886pt)=(0.989072,0.803996,0.18146);
      rgb(887pt)=(0.988536,0.805298,0.180413);
      rgb(888pt)=(0.988,0.8066,0.179367);
      rgb(889pt)=(0.987464,0.807902,0.17832);
      rgb(890pt)=(0.98691,0.809186,0.177291);
      rgb(891pt)=(0.986349,0.810462,0.17627);
      rgb(892pt)=(0.985787,0.811738,0.175249);
      rgb(893pt)=(0.985226,0.813015,0.174228);
      rgb(894pt)=(0.984644,0.814311,0.173207);
      rgb(895pt)=(0.984057,0.815613,0.172186);
      rgb(896pt)=(0.98347,0.816914,0.171165);
      rgb(897pt)=(0.982883,0.818216,0.170144);
      rgb(898pt)=(0.982296,0.819518,0.169145);
      rgb(899pt)=(0.981709,0.82082,0.16815);
      rgb(900pt)=(0.981122,0.822122,0.167154);
      rgb(901pt)=(0.980535,0.823423,0.166159);
      rgb(902pt)=(0.979947,0.824725,0.165163);
      rgb(903pt)=(0.97936,0.826027,0.164168);
      rgb(904pt)=(0.978773,0.827329,0.163172);
      rgb(905pt)=(0.978186,0.828631,0.162177);
      rgb(906pt)=(0.977599,0.829932,0.161181);
      rgb(907pt)=(0.977012,0.831234,0.160186);
      rgb(908pt)=(0.976425,0.832536,0.15919);
      rgb(909pt)=(0.975838,0.833841,0.158195);
      rgb(910pt)=(0.975251,0.835168,0.157199);
      rgb(911pt)=(0.974664,0.836495,0.156204);
      rgb(912pt)=(0.974077,0.837823,0.155208);
      rgb(913pt)=(0.973489,0.83915,0.154213);
      rgb(914pt)=(0.972902,0.840477,0.153217);
      rgb(915pt)=(0.972315,0.841805,0.152222);
      rgb(916pt)=(0.971728,0.843132,0.151226);
      rgb(917pt)=(0.971155,0.844466,0.150224);
      rgb(918pt)=(0.970619,0.845819,0.149203);
      rgb(919pt)=(0.970083,0.847172,0.148182);
      rgb(920pt)=(0.969547,0.848525,0.147161);
      rgb(921pt)=(0.96902,0.849886,0.14614);
      rgb(922pt)=(0.968509,0.851265,0.145119);
      rgb(923pt)=(0.967999,0.852643,0.144098);
      rgb(924pt)=(0.967488,0.854022,0.143077);
      rgb(925pt)=(0.967,0.855411,0.142056);
      rgb(926pt)=(0.966541,0.856815,0.141035);
      rgb(927pt)=(0.966081,0.858219,0.140014);
      rgb(928pt)=(0.965622,0.859623,0.138992);
      rgb(929pt)=(0.965189,0.86104,0.137945);
      rgb(930pt)=(0.96478,0.862469,0.136873);
      rgb(931pt)=(0.964372,0.863899,0.135801);
      rgb(932pt)=(0.963963,0.865328,0.134729);
      rgb(933pt)=(0.96357,0.866773,0.133657);
      rgb(934pt)=(0.963187,0.868228,0.132585);
      rgb(935pt)=(0.962805,0.869683,0.131513);
      rgb(936pt)=(0.962422,0.871138,0.130441);
      rgb(937pt)=(0.962091,0.87261,0.129351);
      rgb(938pt)=(0.961785,0.874091,0.128253);
      rgb(939pt)=(0.961478,0.875571,0.127156);
      rgb(940pt)=(0.961172,0.877052,0.126058);
      rgb(941pt)=(0.960885,0.878571,0.124961);
      rgb(942pt)=(0.960605,0.880103,0.123863);
      rgb(943pt)=(0.960324,0.881634,0.122765);
      rgb(944pt)=(0.960043,0.883166,0.121668);
      rgb(945pt)=(0.959849,0.884719,0.120549);
      rgb(946pt)=(0.95967,0.886276,0.119426);
      rgb(947pt)=(0.959491,0.887833,0.118302);
      rgb(948pt)=(0.959313,0.88939,0.117179);
      rgb(949pt)=(0.959181,0.890995,0.116032);
      rgb(950pt)=(0.959054,0.892603,0.114884);
      rgb(951pt)=(0.958926,0.894211,0.113735);
      rgb(952pt)=(0.958799,0.895819,0.112587);
      rgb(953pt)=(0.958748,0.897453,0.111464);
      rgb(954pt)=(0.958697,0.899086,0.110341);
      rgb(955pt)=(0.958646,0.90072,0.109217);
      rgb(956pt)=(0.958602,0.902359,0.108089);
      rgb(957pt)=(0.958628,0.904043,0.106915);
      rgb(958pt)=(0.958653,0.905728,0.105741);
      rgb(959pt)=(0.958679,0.907413,0.104567);
      rgb(960pt)=(0.958718,0.909102,0.103393);
      rgb(961pt)=(0.95882,0.910812,0.102219);
      rgb(962pt)=(0.958922,0.912522,0.101044);
      rgb(963pt)=(0.959024,0.914232,0.0998703);
      rgb(964pt)=(0.959153,0.915962,0.0986961);
      rgb(965pt)=(0.959357,0.917749,0.0975219);
      rgb(966pt)=(0.959561,0.919536,0.0963477);
      rgb(967pt)=(0.959765,0.921323,0.0951736);
      rgb(968pt)=(0.959996,0.923118,0.0939907);
      rgb(969pt)=(0.960277,0.924931,0.092791);
      rgb(970pt)=(0.960557,0.926743,0.0915913);
      rgb(971pt)=(0.960838,0.928555,0.0903916);
      rgb(972pt)=(0.961151,0.930378,0.0891919);
      rgb(973pt)=(0.961509,0.932216,0.0879922);
      rgb(974pt)=(0.961866,0.934054,0.0867925);
      rgb(975pt)=(0.962223,0.935892,0.0855928);
      rgb(976pt)=(0.96262,0.937768,0.0843802);
      rgb(977pt)=(0.963053,0.939683,0.083155);
      rgb(978pt)=(0.963487,0.941597,0.0819297);
      rgb(979pt)=(0.963921,0.943512,0.0807045);
      rgb(980pt)=(0.964415,0.945426,0.0794643);
      rgb(981pt)=(0.964951,0.947341,0.0782135);
      rgb(982pt)=(0.965487,0.949255,0.0769628);
      rgb(983pt)=(0.966023,0.951169,0.075712);
      rgb(984pt)=(0.966594,0.953118,0.0744441);
      rgb(985pt)=(0.967181,0.955083,0.0731679);
      rgb(986pt)=(0.967768,0.957049,0.0718916);
      rgb(987pt)=(0.968355,0.959014,0.0706153);
      rgb(988pt)=(0.96898,0.960999,0.0693006);
      rgb(989pt)=(0.969619,0.96299,0.0679733);
      rgb(990pt)=(0.970257,0.964981,0.0666459);
      rgb(991pt)=(0.970895,0.966972,0.0653186);
      rgb(992pt)=(0.971554,0.968984,0.0639486);
      rgb(993pt)=(0.972218,0.971001,0.0625703);
      rgb(994pt)=(0.972882,0.973017,0.0611919);
      rgb(995pt)=(0.973545,0.975034,0.0598135);
      rgb(996pt)=(0.974232,0.97705,0.058318);
      rgb(997pt)=(0.974922,0.979067,0.056812);
      rgb(998pt)=(0.975611,0.981083,0.055306);
      rgb(999pt)=(0.9763,0.9831,0.0538)
    }
}

\usetikzlibrary{arrows,shapes,positioning}
\usetikzlibrary{decorations.markings}
\tikzstyle arrowstyle=[scale=1]
\tikzstyle directed=[postaction={decorate,decoration={markings,
        mark=at position .65 with {\arrow[arrowstyle]{stealth}}}}]
\tikzstyle reverse directed=[postaction={decorate,decoration={markings,
        mark=at position .65 with {\arrowreversed[arrowstyle]{stealth};}}}]

\newcounter{tikzsubfigcounter}[figure]
\renewcommand{\thetikzsubfigcounter}{\thesection.\the\numexpr\value{figure}+1\relax\alph{tikzsubfigcounter}}

\newcounter{tikzsubfigcounterinvisible}[figure]
\renewcommand{\thetikzsubfigcounterinvisible}{\the\numexpr\value{figure}+1\relax\alph{tikzsubfigcounterinvisible}}

\begin{document}

\begin{abstract}
  We derive a second-order realizability-preserving scheme for moment models for linear kinetic
  equations. We apply this scheme to the first-order continuous ($\HFMN$) and discontinuous ($\PMMN$)
  models in slab and three-dimensional geometry derived in \cite{SchneiderLeibner2020} as well as
  the classical full-moment $\Mn$ models. We provide extensive numerical analysis as well as our code
  to show that the new class of models can compete or even outperform the full-moment models in
  reasonable test cases.
\end{abstract}
\maketitle

\noindent


\def\tikzpath{Images/}

\section{Introduction}
We consider moment closures, which are a type of (non-linear) Galerkin projection, in the context
of kinetic transport equations. Here, moments are defined by taking velocity- or phase-space
averages with respect to some (truncated) basis of the velocity space. Unfortunately, the
truncation inevitably comes at the cost that information is required from the basis elements
which were removed.

The specification of this information, the so-called moment closure problem, distinguishes
different moment
methods.
In the context of linear radiative transport, the standard spectral method
is commonly referred to as the $\PN$ closure \cite{Lewis1984},
where $\momentorder$ is the degree of the highest-order moments in the model.
The $\PN$ method is powerful and simple to implement, but does not
take into account the fact that the original function to be
approximated, the kinetic density, must be non-negative.
Thus, $\PN$ solutions can contain negative values for the local
densities of particles, rendering the solution physically meaningless.
Entropy-based moment closures, typically denoted by $\Mn$ models in the
context of radiative transport \cite{Min78,DubFeu99},
have (for physically relevant entropies) all the properties one would desire in a moment method,
namely
positivity of the underlying kinetic density, hyperbolicity of the closed system of equations,
and entropy dissipation \cite{Levermore1996}.
These models are usually comparatively expensive as they require the numerical solution of an
optimization problem at every point on the space-time grid. Practical interest in such models
increased recently due to their inherent parallelizability \cite{Hauck2010}.
While the cost of solving the local nonlinear problems in the $\Mn$ model scales
strongly with the
number of moments $\momentnumber$ (since one has to solve square problems of size
$\momentnumber$),
the desired spectral convergence with respect to the moment order $\momentorder$ is only
achieved
for smooth test cases, which rarely occur in reality. This means that the gain in efficiency by
increasing the order of approximation will become rather insignificant.

To increase the accuracy of the $\Mn$ models while maintaining the lower cost for
small moment
order $\momentorder$, a partition of the velocity space while keeping the moment order
fixed is
useful, similar to some h-refinement for, e.g., finite element approximations
\cite{Babuska1992}. We
focus on the continuous and discontinuous piece-wise linear bases derived in
\cite{SchneiderLeibner2020}, which aim to be a generalization of the special cases provided in
\cite{Frank2006,DubKla02,Schneider2014,Ritter2016,Schneider2017} in slab geometry and
the fully three-dimensional case.

Besides their inherent parallelizability, in order to make these methods truly competitive with
more
basic discretizations, the gains in efficiency that come from higher-order methods (in space and
time) are necessary. Here the issue of realizability becomes a stumbling block.
The property of positivity implies that the system of moment
equations only evolves on the set of so-called realizable moments.
Realizable moments are simply those moments associated with positive
densities, and the set of these moments forms a convex cone which is a
strict subset of all moment vectors.
This property, even though desirable due to its consistency with the
original kinetic distribution, can cause problems in numerical simulations.
Standard high-order numerical solutions (in space and time) to the Euler equations, which indeed
are an entropy-based moment closure, have been observed to have negative
local densities and pressures \cite{Zhang2010}. Similar effects have been reported in the
context
of elastic flow \cite{Schar1996}. This is exactly loss of realizability.

We propose a second-order realizability-preserving
scheme, that is based on a splitting technique and analytic solutions of the stiff part, combined
with a realizability-preserving reconstruction scheme. It turns out that this scheme is very
effective for (medium) smooth and non-smooth test cases, which can also occur in practice. The
realizability-preserving property is achieved using the realizability limiter proposed in
\cite{Alldredge2015,Schneider2015b,Schneider2016a,Chidyagwai2017}. This limiter requires
information about the set of realizable moments, which turns out to be very simple in the context
of
our first-order models \cite{SchneiderLeibner2020}. Again, this additionally makes the
implementation of such models faster (and easier) compared to standard $\Mn$
models.

This paper is organized as follows. First, we shortly recall the transport equation, its moment
approximations and the relevant results from \cite{SchneiderLeibner2020} (\cref{sec:modelling,sec:angularbases}).
Then, we propose our
second-order realizability-preserving scheme and investigate all the required properties that it
should fulfill (\cref{sec:scheme}). In \cref{section:ImplementationDetails}, we discuss some
implementation details of our scheme.
Finally, in \cref{sec:results}, we give a comprehensive numerical investigation of our models
and the
$\Mn$
models in slab geometry and three dimension, to show that our models can indeed compete with or
even
outperform the full-moment models.

\section{Modeling}\label{sec:modelling}
This section closely follows the corresponding part in \cite{SchneiderLeibner2020}. We consider
the linear transport equation
\begin{subequations}\label{eq:FokkerPlanckEasy}
  \begin{equation}\label{eq:TransportEquation}
    \dt\distribution+\SC\cdot\spatialGradient\distribution + \absorption\distribution =
    \scattering\collision{\distribution}+\source,
  \end{equation}
  which describes the density of particles with speed $\SC\in\sphere$ at position
  $\spatialvar = (\x,\y,\z)^T\in\domain\subseteq\R^3$ and time $\timevar$ under the events of
  scattering (proportional to $\scattering\left(\timevar,\spatialvar\right)$), absorption
  (proportional to $\absorption\left(\spatialvar\right)$) and emission (proportional to
  $\source\left(\spatialvar,\SC\right)$). Collisions are modeled using the BGK-type collision
  operator
  \begin{equation}\label{eq:collisionOperatorR}
    \collision{\distribution} =  \int\limits_{\sphere} \collisionkernel(\SC, \SC^\prime)
    \distribution(\timevar, \spatialvar, \SC^\prime)~d\SC^\prime
    - \int\limits_{\sphere} \collisionkernel(\SC^\prime, \SC) \distribution(\timevar, \spatialvar, \SC)~d\SC^\prime.
  \end{equation}
  The collision kernel $\collisionkernel$ is assumed to be strictly positive, symmetric (i.e.
  $\collisionkernel(\SC,\SC')=\collisionkernel(\SC',\SC)$) and normalized to
  $\int\limits_{\sphere} \collisionkernel(\SC^\prime, \SC) d\SC^\prime~\equiv~1$.	In this paper, we restrict ourselves to
  \emph{isotropic scattering}, where $\collisionkernel(\SC, \SC^\prime) \equiv \frac{1}{\abs{\sphere}} = \frac{1}{4\pi}$.

  The equation is supplemented with initial condition and Dirichlet boundary conditions:
  \begin{align}
    \distribution(0,\spatialvar,\SC)
     & = \distributiontzero(\spatialvar,\SC)
     & \text{for } \spatialvar\in\domain, \SC\in\sphere     \label{eq:kineticequationinitial}                                    \\
    \distribution(\timevar,\spatialvar,\SC)
     & = \distributionboundary(\timevar,\spatialvar,\SC)
     & \text{for } \timevar\in\timeint, \spatialvar\in\partial\domain, \outernormal\cdot\SC<0 \label{eq:kineticequationboundary}
  \end{align}
  where $\outernormal$ is the outward unit normal vector in $\spatialvar\in\partial\domain$.
\end{subequations}
Parameterizing $\SC$ in spherical coordinates we obtain
\begin{equation}\label{eq:SphericalCoordinates}
  \SC = \left(\sqrt{1-\SCheight^2}\cos(\SCangle),\sqrt{1-\SCheight^2}\sin(\SCangle),
  \SCheight\right)^T \eqqcolon \left(\SCx,\SCy,\SCz\right)^T
\end{equation}
where $\SCangle\in[0,2\pi]$ is the azimuthal and $\SCheight\in[-1,1]$ the cosine of the polar
angle.
\begin{definition}
  The vector of functions $\basis:\sphere\to\R^{\momentnumber}$ consisting of $\momentnumber$ basis functions
  $\basiscomp[i]$, $\basisindex=0,\ldots\momentnumber-1$ of maximal \emph{order}
  $\momentorder$ (in $\SC$) is called an \emph{angular basis}.

  The so-called \emph{moments} $\moments=\left(\momentcomp{0},\ldots,\momentcomp{\momentnumber-1}\right)^T$ of a given distribution function
  $\distribution$ are then defined by
  \begin{equation}\label{eq:moments}
    \moments = \int\limits_{\sphere} {\basis}\distribution~d\SC \eqqcolon \ints{\basis\distribution}
  \end{equation}
  where the integration is performed component-wise.

  Furthermore, the quantity $\density = \density(\moments) \coloneqq \ints{\distribution}$ is called the
  \emph{local particle density}. Additionally, $\isotropicmoment = \ints{\basis}$ is called the
  \emph{isotropic moment}.
\end{definition}
Equations for $\moments$ can then be obtained by multiplying \eqref{eq:FokkerPlanckEasy}
with $\basis$ and integration over $\sphere$, resulting in
\begin{equation}\label{eq:MomentSystemUnclosed}
  \dt\moments+\spatialGradient\cdot\ints{\SC \basis\distribution} + \absorption\moments =
  \scattering\ints{\basis\collision{\distribution}}+\ints{\basis\source}.
\end{equation}
Depending on the choice of $\basis$ the terms $\ints{\SCx \basis\distribution}$,
$\ints{\SCy \basis\distribution}$, $\ints{\SCz \basis\distribution}$, and in some cases even $\ints{\basis\collision{\distribution}}$,
cannot be given explicitly in terms of $\moments$. Therefore an ansatz
$\ansatz$ has to be made for $\distribution$ closing the unknown terms. This is
called the \emph{moment-closure problem}.

In this paper the ansatz density $\ansatz$ is reconstructed from the moments
$\moments$ by minimizing the entropy-functional
\begin{equation}\label{eq:OptProblem}
  \entropyFunctional(\distribution) = \ints{\entropy(\distribution)}\text{ under the moment constraints }  \ints{\basis\distribution} =
  \moments.
\end{equation}
The kinetic entropy density $\entropy\colon\R\to\R$ is strictly convex and twice continuously
differentiable and the
minimum is simply taken over all functions $\distribution = \distribution(\SC)$ such that
$\entropyFunctional(\distribution)$ is well defined.
This problem, which must be solved over the space-time mesh, is typically solved through its
strictly convex finite-dimensional dual,
\begin{equation} \label{eq:dual}
  \multipliers(\moments) \coloneqq \underset{\tilde{\multipliers} \in \R^{\momentnumber}}{\argmin}
  \ints{\ld{\entropy}(\basis \cdot \tilde{\multipliers})} - \moments \cdot \tilde{\multipliers},
\end{equation}
where \(\ld{\entropy}\) is the Legendre dual of \(\entropy\). The first-order necessary
conditions for the
multipliers \(\multipliers(\moments)\) show that
the solution to~\eqref{eq:OptProblem}, if it exists, has the form
\begin{equation}\label{eq:psiME}
  \ansatz[\moments] = \ld{\entropy}' \left(\basis \cdot \multipliers(\moments) \right)
\end{equation}
This approach is called the \emph{minimum-entropy closure} \cite{Levermore1996}. The resulting model
has many desirable properties:
symmetric hyperbolicity, bounded eigenvalues of the directional flux Jacobian and the direct
existence of an entropy-entropy flux pair (compare \cite{Levermore1996,Schneider2016}).

The kinetic entropy density $\entropy$ can be chosen according to the
physics being modelled.
As in \cite{Levermore1996,Hauck2010}, Maxwell-Boltzmann entropy%
\begin{equation}\label{eq:EntropyM}
  \entropy(\distribution) = \distribution \log(\distribution) - \distribution
\end{equation}
is used, thus $\ld{\entropy}(p) = \ld{\entropy}'(p) = \exp(p)$. This entropy is used for non-interacting particles as in an
ideal gas.

Substituting $\distribution$ in \eqref{eq:MomentSystemUnclosed} with $\ansatz[\moments]$ yields a
closed system of equations for $\moments$:
\begin{equation}\label{eq:MomentSystemClosed}
  \dt\moments+\dx\ints{\SCx \basis\ansatz[\moments]}+\dy\ints{\SCy \basis\ansatz[\moments]}+\dz\ints{\SCz \basis\ansatz[\moments]} +
  \absorption\moments = \scattering\ints{\basis\collision{\ansatz[\moments]}}+\ints{\basis\source}.
\end{equation}
\begin{remark}
  Note that using the entropy $\entropy(\distribution) = \frac12\distribution^2$ the linear ansatz
  \begin{equation}\label{eq:PnAnsatz}
    \ansatz[\moments] = \basis \cdot \multipliers(\moments)
  \end{equation}
  is obtained, leading to standard continuous/discontinuous-Galerkin approaches.
  If the angular basis is chosen as spherical harmonics of order $\momentorder$,
  \eqref{eq:MomentSystemClosed} turns into the classical $\PN$
  model \cite{Blanco1997,Brunner2005b,Seibold2014}.
\end{remark}
For convenience, we write \eqref{eq:MomentSystemClosed} in the standard form of a non-linear hyperbolic
system of partial differential equations:
\begin{equation}\label{eq:GeneralHyperbolicSystem2D}
  \dt\moments+\dx\Flux_1\left(\moments\right)+\dy\Flux_2\left(\moments\right)+\dz\Flux_3\left(\moments\right) = \Source\left(\moments\right),
\end{equation}
where
\begin{subequations}\label{eq:FluxDefinitions}
  \begin{align}
    \Flux_1\left(\moments\right)             & =  \ints{\SC_x\basis\ansatz[\moments]},\quad \Flux_2\left(\moments\right) =  \ints{\SC_y\basis\ansatz[\moments]},\quad \Flux_3\left(\moments\right) =  \ints{\SC_z\basis\ansatz[\moments]}\in\R^{\momentnumber}, \\
    \Source\left(\spatialvar,\moments\right) & = {\scattering}(\spatialvar)\ints{\basis \collision{\ansatz[\moments]}}+\ints{\basis\source(\spatialvar,\cdot)}-\absorption(\spatialvar)\moments.
  \end{align}
\end{subequations}
For ease of visibility, we also consider our models in slab geometry, which is a projection of the
sphere onto the $\z$-axis \cite{Seibold2014}. The transport equation under
consideration then has the form
\begin{equation}\label{eq:TransportEquation1D}
  \dt\distribution+\SCheight\dz\distribution + \absorption\distribution =
  \scattering\collision{\distribution}+\source, \qquad
  \timevar\in\timeint,\z\in\domain,\SCheight\in[-1,1].
\end{equation}
The shorthand notation $\ints{\cdot} = \int\limits_{-1}^1\cdot~d\SCheight$ then denotes integration over $[-1,1]$
instead of $\sphere{}$. Finally, the moment system is given by
\begin{equation}\label{eq:MomentSystemUnclosed1D}
  \dt\moments+\dz\ints{\SCheight \basis\ansatz[\moments]} + \absorption\moments =
  \scattering\ints{\basis\collision{\ansatz[\moments]}}+\ints{\basis\source}.
\end{equation}

\section{Angular bases}\label{sec:angularbases}
We shortly recall the angular bases under consideration. For a detailed derivation and further
information, we refer the reader to \cite{SchneiderLeibner2020}\,.
\subsection{Slab geometry}
\begin{itemize}
  \item Full-moment basis
        \begin{subequations}
          \begin{align}
            \label{eq:monomialbasis}
            \fmbasis[\momentorder] & = \left(1,\SCheight,\ldots,\SCheight^\momentorder\right)^T \qquad\qquad
            \text{ or }                                                                                      \\
            \label{eq:legendrebasis}
            \fmbasis[\momentorder] & = \left(P_0^0,P_1^0,P_2^0\ldots,P_\momentorder^0\right)^T
          \end{align}
        \end{subequations}
        with the monomials or the \emph{Legendre polynomials} $P_{\momentindex}^0$, $\momentindex=0,\ldots,\momentorder$.
  \item Piecewise-linear angular basis (hat functions, continuous-Galerkin ansatz) $\hfbasis =
          \left(\hfbasiscomp[0],\ldots,\hfbasiscomp[\momentnumber-1]\right)^T$
        \begin{equation}\label{eq:hfbasis}
          \hfbasiscomp[\basisindex](\SCheight) =
          \indicator[\ccell{\basisindex-1}]\cfrac{\SCheight-\SCheight_{\basisindex-1}}{\SCheight_\basisindex-\SCheight_{\basisindex-1}}
          +\indicator[\ccell{\basisindex}]\cfrac{\SCheight-\SCheight_{\basisindex+1}}{\SCheight_\basisindex-\SCheight_{\basisindex+1}},
        \end{equation}
        where $-1 = \SCheight_0 < \SCheight_1 < \ldots < \SCheight_{\momentnumber-2} < \SCheight_{\momentnumber-1}=1$ are some angular ``grid'' points
        and
        \(\indicator[\ccell{\cellindex}](\SCheight)\) is the indicator function on the interval
        \(\ccell{\cellindex} = [\SCheight_{\cellindex}, \SCheight_{\cellindex+1}]\)
        (with \(\indicator[\ccell{-1}] \equiv \indicator[\ccell{\momentnumber-1}] \equiv 0\)).
  \item Partial moments (discontinuous-Galerkin ansatz) $\pmbasis[] = \left(\pmbasis[\cell{0}],\ldots\pmbasis[\cell{\hankelhalfind-1}]\right)$
        \begin{equation*}
          \pmbasis[\cell{\cellindex}] = \indicator[\cell{\cellindex}]\left(1,\SCheight\right)^T,
        \end{equation*}
        where \(\hankelhalfind\) is the number of intervals.
\end{itemize}

\begin{definition}
  \label{def:SlabGeometryNames}
  The resulting linear (compare \eqref{eq:PnAnsatz}) and nonlinear models (compare
  \eqref{eq:EntropyM}) will be called $\PN$/$\Mn$ (full moment
  basis), $\HFPN$/$\HFMN$ (hat functions basis) and
  $\PMPN$/$\PMMN$ (partial moment basis), respectively.
\end{definition}

\subsection{Angular bases in three dimensions}
Albeit both approaches are not limited to this, we consider moments on spherical triangles.
To that end, let $\TriangulationSphere$ be a spherical triangulation of $\sphere$ and
$\sphericaltriangle\in\TriangulationSphere$ be a spherical triangle.
In this paper, the triangulation \(\generalpartition\) will be obtained by dyadic refinement
of the octants of the sphere \(\angularDomain = \sphere\),
i.e.\ the coarsest triangulation contains the eight spherical triangles obtained by projecting the
octahedron with vertices
\(\{{(\pm1, 0, 0)}^T\), \({(0, \pm1, 0)}^T\),
\({(0, 0, \pm1)}^T\}\) to the sphere
and finer partitions are obtained by iteratively subdividing each spherical triangle into four new
ones,
adding vertices at the midpoints of the triangle edges. After
\(\refinementnumber\) refinements, we thus obtain \(\nvertex(\refinementnumber) =
4^{\refinementnumber+1}+2\) vertices and \(\nentity(\refinementnumber) = 2 \cdot
4^{\refinementnumber+1}\) spherical triangles.

The bases that we use are the following.
\begin{itemize}
  \item Full-moment basis
        \begin{equation*}
          \fmbasis[\momentorder] = \left(\Slm(\SCheight,\SCangle); \SHl = 0,\ldots,\momentorder,~ \SHm = -\SHl,\ldots,\SHl \right)^T,
        \end{equation*}
        where $\Slm$ are the real-valued spherical harmonics on the unit sphere
        \cite{Blanco1997,Seibold2014}.
  \item Barycentric-coordinate basis functions
        \begin{equation*}
          \hfbasis[\nvertex] = \left(\hfbasiscomp[0],\ldots,\hfbasiscomp[\nvertex-1]\right),
        \end{equation*}
        where \(\nvertex\) is the number of vertices of the triangulation and $\hfbasiscomp[\basisindex]$ is the
        basis function
        defined using spherical barycentric coordinates
        on the $\basisindex$-th vertex as
        in \cite{Buss2001,Langer2006,Rustamov2010}.
  \item Partial moments on the unit sphere
        \begin{equation*}
          \pmbasis[\momentnumber] = \left(\pmbasis[\sphericaltriangle]; \sphericaltriangle\in\TriangulationSphere\right) = \left(\left(\indicator[\sphericaltriangle],\indicator[\sphericaltriangle]\SC\right); \sphericaltriangle\in\TriangulationSphere\right),
        \end{equation*}
        where $\momentnumber = 4\cdot\abs{\TriangulationSphere}$ is the number of moments.
\end{itemize}
Naming of the models will be analogous to the slab-geometry case, compare \defnref{def:SlabGeometryNames}.

\subsection{Realizability}\label{sec:realizability}
The minimum-entropy moment problem \eqref{eq:OptProblem} has a
solution if and only if the moment vector is realizable.
\begin{definition}
  \label{def:RealizableSet}
  The \emph{realizable set} $\RD{\basis}{}$ is
  \begin{equation*}
    \RD{\basis}{} = \left\{\moments \in \R^{\momentnumber}~:~\exists
    \distribution(\SC)\ge 0,\, \density =
    \ints{\distribution} > 0,
    \text{ such that } \moments =\ints{\basis\distribution} \right\}.
  \end{equation*}
  If $\moments\in\RD{\basis}{}$, then $\moments$ is called \emph{realizable}.
  Any $\distribution$ such that $\moments =\ints{\basis \distribution}$ is called a \emph{representing
    density}.
\end{definition}
Unfortunately, checking whether a moment vector is realizable is not trivial for general bases.
However, for the piecewise linear moment models, the realizability conditions are particularly
simple (see~\cite{SchneiderLeibner2020}).
\begin{lemma}\label{cor:hf1drealizabilitypos}
  For the hat function basis in one or three dimensions, \(\moments \in \closure{\RD{\hfbasis}{}}\)
  if and only if $\momentcomp{\basisindex} \geq 0$ for all $\basisindex=0,\ldots,\momentnumber-1$.
\end{lemma}
\begin{lemma}\label{cor:pm1drealizabilitypos}
  For the partial moment basis in one dimension (slab geometry), \(\moments \in
  \closure{\RD{\pmbasis}{}}\)
  if and only if
  \begin{equation}\label{eq:pmrealizabilityconds}
    \momentcomp{2\cellindex} \geq 0 \quad \text{ and, for } \momentcomp{2\cellindex} > 0, \quad
    \frac{\momentcomp{2\cellindex+1}}{\momentcomp{2\cellindex}} \in \ccell{\cellindex} = [\SCheight_{\cellindex},
    \SCheight_{\cellindex+1}]
  \end{equation}
  for all $\cellindex=0,\ldots,\frac{\momentnumber}{2}-1$.
\end{lemma}
For more details on the realizability of the regarded models, see \cite{SchneiderLeibner2020}.

\section{Second-order realizability-preserving splitting scheme}\label{sec:scheme}

As already mentioned before, the minimum-entropy moment problem \eqref{eq:OptProblem} has a
solution if and only if the moment vector is realizable. This implies that it is mandatory to
maintain realizability during the numerical simulation (since otherwise the flux function cannot be
evaluated).
Explicit high-order schemes have been developed in \cite{Alldredge2015,Schneider2015b}.
Unfortunately, the physical parameters $\scattering$ and $\absorption$ directly
influence the CFL
condition, resulting in very small time steps for large scattering/absorption.

This can be overcome by using a first-order implicit-explicit time stepping scheme
\cite{Schneider2016a,Schneider2016aCodeIMEX,schneider2016implicit}, treating the transport part
explicit while implicitly solving the (time-)critical source term. Unfortunately, using
higher-order
IMEX schemes again results in a CFL condition of the same magnitude as for the fully explicit
schemes.

We are interested in a second-order scheme for \eqref{eq:GeneralHyperbolicSystem2D}. This can be
achieved by doing a Strang splitting for
\begin{subequations}
  \begin{align}
    \label{eq:SplittedSystema}
    \dt\moments & +\dx\Flux_1\left(\moments\right)+\dy\Flux_2\left(\moments\right)+\dz\Flux_3\left(\moments\right) = 0,                                                    \\
    \label{eq:SplittedSystemb}
    \dt\moments & \phantom{+\dx\Flux_1\left(\moments\right)+\dy\Flux_2\left(\moments\right)+\dz\Flux_3\left(\moments\right)}\,~= \Source\left(\spatialvar,\moments\right). 
  \end{align}
\end{subequations}
A \emph{second-order realizability preserving scheme} will be obtained if both subsystems are
solved with a (at least) second-order accurate and realizability-preserving scheme. For notational
simplicity, we show the full scheme for \emph{one spatial dimension} only. A generalization to
structured meshes in higher dimensions is straightforward.

\subsection{Source system}\label{sec:sourcesystem}
Let us start with the stiff part \eqref{eq:SplittedSystemb} whose finite-volume form is given by
\begin{equation}
  \label{eq:ODEexact}
  \dt \momentscellmean{\cellindex} =
  \frac{1}{\gridwidthz}\int\limits_{\z_{\cellindex-\frac12}}^{\z_{\cellindex+\frac12}}
  \Source\left(\z,\moments\right)~d\z.
\end{equation}
Fortunately, using the midpoint rule, it holds that
\begin{equation}
  \cellmean{\Source\left(\z,\moments\right)} =
  \frac{1}{\gridwidthz}\int\limits_{\z_{\cellindex-\frac12}}^{\z_{\cellindex+\frac12}}
  \Source\left(z, \moments\right)~d\z =
  \Source\left(\z_\cellindex,\momentscellmean{\cellindex}\right)+\mathcal{O}(\gridwidthz^2).
\end{equation}
To obtain a second-order accurate solution of \eqref{eq:ODEexact}, it is thus sufficient to
solve the system
\begin{equation}
  \label{eq:ODE}
  \dt \momentscellmean{\cellindex} = \Source\left(\z_\cellindex,\momentscellmean{\cellindex}\right),
\end{equation}
which is purely an ODE (in every cell).  As mentioned above, we restrict ourselves to isotropic
scattering,
where we have
$\collisionkernel(\SCheight,\SCheight') = \frac{1}{\ints{1}} = \frac12$, i.e.
\begin{equation}
  \collision{\distribution} = \frac{\ints{\distribution}}{\ints{1}} - \distribution.
\end{equation}
The source term now becomes
\begin{align}
  \begin{split}
    \Source\left(\z,\moments\right)
    & = {\scattering}\ints{\basis \collision{\ansatz[\moments]}}+\ints{\basis\source}
    -\absorption\moments
    = {\scattering}\frac{\density(\moments)}{\ints{1}}\ints{\basis}
    +\ints{\basis\source} -\crosssection\moments					     \\
    & = {\scattering} \isotropicmoment(\moments) + \ints{\basis\source}
    -\crosssection\moments
    = \left({\scattering} \isomatrix - \crosssection \eyematrix\right) \moments +
    \ints{\basis\source},
  \end{split}
\end{align}
where $\isomatrix = \frac{\isotropicmoment{(\multipliersone)}^T}{\ints{1}}$ is the matrix mapping the
moment
vector $\moments$
to the isotropic moment vector with the same density
$\isotropicmoment(\moments) = \isotropicmoment \cdot \frac{\density(\moments)}{\ints{1}}$. Here we
assumed that there exists a vector $\multipliersone$ such that $\multipliersone \cdot \basis \equiv 1$
(true for all regarded bases: $\multipliersone = (1, 0, \ldots, 0)^T$ for Legendre Polynomials,
$\multipliersone =
  (\sqrt{4 \pi}, 0, \ldots, 0)^T$ for real spherical harmonics,  $\multipliersone = (1, \ldots, 1)^T$ for the hat functions
basis,
$\multipliersone = (1, 0, 1, 0, \ldots)^T$ for the partial moments in slab geometry and $(1, 0, 0, 0, 1, 0, 0, 0, \ldots)^T$ for the
partial moment basis in three dimensions).

Since in this case, \eqref{eq:ODE} is
linear and the parameters $\scattering$, $\absorption$, $\source$
are
time-independent, we solve it explicitly using matrix exponentials, trivially
obtaining a realizable second-order accurate solution of \eqref{eq:ODEexact}.
\begin{remark}
  Note that in this specific situation, the solution of this sub-step does not depend on the moment
  closure used in the flux system.
\end{remark}
Using the matrix exponential and the variation of constants formula, the solution to
\eqref{eq:ODE} is
\begin{equation}
  \label{eq:rhssolutionbymatexp}
  \moments(t,\z) = \matexp\left(\left({\scattering} \isomatrix - \crosssection
    \eyematrix\right) t\right) \moments(0,\z) +
  \left(\int\limits_0^t \matexp
  \left(\left({\scattering} \isomatrix - \crosssection
      \eyematrix\right) (t-s) \right) \mathrm{d}s \right) \ints{\basis\source}
\end{equation}
As $\isomatrix$ and $\eyematrix$ commute, we have
\begin{equation}
  \label{eq:rhsmatexp}
  \matexp\left(\left({\scattering} \isomatrix - \crosssection \eyematrix\right)
  t\right) =
  \matexp\left({\scattering} t \isomatrix\right) \matexp\left(-\crosssection t
  \eyematrix\right) = \matexp\left({\scattering} t \isomatrix\right)
  \left(\exp(-\crosssection t \right)\eyematrix )
\end{equation}
It remains to compute the matrix exponential of ${\scattering} t
  \isomatrix$. As $\isomatrix \isotropicmoment(\moments) =
  \isotropicmoment(\moments)$,
we have that $\isomatrix^k = \isomatrix$ for all
$k \geq 1$. It follows
\begin{equation}
  \label{eq:isomatexp}
  \matexp\left({\scattering} t \isomatrix\right)
  = \sum \limits_{k=0}^\infty \frac{\left({\scattering} t
    \isomatrix\right)^k}{k!}
  = \eyematrix + \sum \limits_{k=1}^\infty \frac{\left({\scattering} t
    \right)^k}{k!} \isomatrix
  = \eyematrix + (\exp(\scattering t) - 1) \isomatrix
\end{equation}
Inserting \eqref{eq:isomatexp} in \eqref{eq:rhsmatexp}, we get
\begin{equation}
  \label{eq:rhsmatexp2}
  \matexp\left(\left({\scattering} \isomatrix - \crosssection \eyematrix\right)
  t\right) = \exp(-\crosssection t) \left(\eyematrix + (\exp(\scattering t) -
  1)
  \isomatrix\right)
\end{equation}
Plugging \eqref{eq:rhsmatexp2} into \eqref{eq:rhssolutionbymatexp}, we finally get
\begin{align}
  \begin{split}
    \moments(t)
    &= \exp(-\crosssection t) \left(\eyematrix + (\exp(\scattering t) -
    1) \isomatrix\right) \moments(0,\z) \\
    &\hspace{2.4cm}+
    \left(\int\limits_0^t \exp(-\crosssection (t-s)) \left(\eyematrix +
    (\exp(\scattering (t-s)) -
    1) \isomatrix\right) \mathrm{d}s \right) \ints{\basis\source} \\
    &= \exp(-\crosssection t) \left(\eyematrix + (\exp(\scattering t) -
    1) \isomatrix\right) \moments(0,\z) \\
    &\hspace{3.5cm}+
    \left(\frac{1-\exp(-\crosssection t)}{\crosssection} \left(\eyematrix -
      \isomatrix\right) +
    \frac{1-\exp(-\absorption t)}{\absorption} \isomatrix\right)
    \ints{\basis\source} \\
    &= e^{-\absorption t} \left(e^{-\scattering t} \moments(0,\z) +
    \left(1 - e^{-\scattering t}\right) \isotropicmoment(\moments(0,\z)) \right) \\
    &\hspace{5cm}+ \left(\frac{1-e^{-\crosssection t}}{\crosssection} \left(\eyematrix -
    \isomatrix\right) +
    \frac{1-e^{-\absorption t}}{\absorption} \isomatrix\right)
    \ints{\basis\source}
    \label{eq:matexpsolutiongeneral}
  \end{split}
\end{align}

If the source is also isotropic then $\isomatrix \ints{\basis\source} =
  \ints{\basis\source} = \ints{\basis}\source$ and \eqref{eq:matexpsolutiongeneral} simplifies
to
\begin{equation}
  \moments(t,\z)
  = e^{-\absorption t} \left(e^{-\scattering t} \moments(0,\z) +
  \left(1 - e^{-\scattering t}\right) \isotropicmoment(\moments(0,\z)) \right)
  +
  \frac{1-e^{-\absorption t}}{\absorption} \ints{\basis} \source
  \label{eq:matexpsolutionisotropicsource}
\end{equation}
which can easily be calculated without explicit calculation of $\isomatrix$ or any matrix
operations.

\subsection{Flux system}
Let us now consider the non-stiff part \eqref{eq:SplittedSystema}. This can be solved using standard
realizability-preserving methods \cite{Schneider2015b,Alldredge2015,Schneider2016,Chidyagwai2017}, which will be summarized in the
following.

The standard finite-volume scheme in semi-discrete form for \eqref{eq:SplittedSystema} looks like
\begin{equation}
  \label{eq:FV1}
  \dt \momentscellmean{\cellindex} = \numericalFlux(\momentspv{\cellindex+\frac12}^-,\momentspv{\cellindex+\frac12}^+)-\numericalFlux(\momentspv{\cellindex-\frac12}^-,\momentspv{\cellindex-\frac12}^+),
\end{equation}
where $\numericalFlux$ is a numerical flux function. The simplest example is the global
Lax-Friedrichs flux
\begin{equation}
  \label{eq:globalLF}
  \numericalFlux(\moments[1], \moments[2]) = \dfrac{1}{2} \left( \Flux_3(\moments[1]) + \Flux_3(\moments[2]) - \viscosityconstant ( \moments[2] -
    \moments[1]) \right).
\end{equation}
The numerical viscosity constant $\viscosityconstant$ is taken as the global estimate of the
absolute value of the largest eigenvalue of the Jacobian
$\Flux'$. In our case, the viscosity constant can be set to $\viscosityconstant = 1$,
because for the moment systems used here the largest eigenvalue is bounded in absolute value by one
\cite{Alldredge2015,Schneider2016,Olbrant2012}.

Another possible choice is the kinetic flux \cite{Schneider2015b,Hauck2010,Garrett2014,Frank2006}
\begin{equation}
  \label{eq:kineticFlux}
  \numericalFlux(\moments[1], \moments[2]) = \intp{\SCheight\basis\ansatz[1]}+ \intm{\SCheight \basis \ansatz[2]},\qquad
  \moments[\basisindex] = \ints{\basis \ansatz[\basisindex]},\quad \basisindex\in\{1,2\},
\end{equation}
where \(\intp{\cdot}\) and \(\intm{\cdot}\) denote integration over the
positive and negative half intervals \([-1, 0]\) and
\([0, 1]\), respectively.
The kinetic flux is less diffusive than the (global) Lax-Friedrichs flux and admits a more
consistent
implementation of kinetic boundary conditions (see \cite{Schneider2016} and
\cref{sec:initialandboundaryfvscheme}). For this reason, we will use
\eqref{eq:kineticFlux} in all our computations.

\subsubsection{Polynomial reconstruction}
\label{subsubsec:polynomial_reconstruction}
The value $\momentspv{\cellindex+\frac12}$ is the evaluation of a suitable linear
reconstruction of $\moments$ at the cell interface $\z_{\cellindex+\frac12}$. In
one dimension, it can be obtained from a minmod reconstruction\footnote{Other
  second-order accurate reconstructions like WENO \cite{Jiang1995,cravero2018cool}
  are also possible.}
\begin{align*}
  \moments[\cellindex](\z) & = \momentscellmean{\cellindex}+\moments[\cellindex]'\left(\z-\z_\cellindex\right)                                                                                                                                                   \\
  \moments[\cellindex]'    & = \frac{1}{\gridwidthz}\minmodfunc{\momentscellmean{\cellindex+1}-\momentscellmean{\cellindex},\momentscellmean{\cellindex}-\momentscellmean{\cellindex-1},\frac12(\momentscellmean{\cellindex+1}-\momentscellmean{\cellindex-1})},
\end{align*}
where $\minmodfunc{\cdot}$ is the minmod function
\begin{align*}
  \minmodfunc{a_1,a_2,a_3} & = \begin{cases}
    \operatorname{sign}(a_1) \min\{|a_1|,|a_2|,|a_3|\} & \text{if }
    \operatorname{sign}(a_1) = \operatorname{sign}(a_2) =
    \operatorname{sign}(a_3),                                         \\
    0                                                  & \text{else}.
  \end{cases}
\end{align*}
applied componentwise. We then set $\momentspv{\cellindex+\frac12}^- = \moments[\cellindex](\z_{\cellindex+\frac12})$ and $\momentspv{\cellindex+\frac12}^+ = \moments[\cellindex+1](\z_{\cellindex+\frac12})$.

To avoid spurious oscillations, the reconstruction has to be performed in characteristic variables.
They are found by transforming the moment vector $\moments$
using the matrix $\eigenvectormatrix_{\cellindex}$, whose columns hold the eigenvectors of the Jacobian
$\Flux'(\momentscellmean{\cellindex})$ evaluated at the cell mean $\momentscellmean{\cellindex}$. This leads to
\begin{equation}
  \moments[\cellindex]' = \frac{1}{\gridwidthz}\eigenvectormatrix_{\cellindex}\minmodfunc{\eigenvectormatrix_{\cellindex}^{-1}\left(\momentscellmean{\cellindex+1}-\momentscellmean{\cellindex}\right),\eigenvectormatrix_{\cellindex}^{-1}\left(\momentscellmean{\cellindex}-\momentscellmean{\cellindex-1}\right),\frac12\eigenvectormatrix_{\cellindex}^{-1}\left(\momentscellmean{\cellindex+1}-\momentscellmean{\cellindex-1}\right)}.
\end{equation}
For details on the eigenvalue computation see \cref{sec:eigenvalueproblems}.
In several dimension, we perform a dimension-by-dimension reconstruction as in
\cite{Titarev2004} using the minmod reconstruction in characteristic variables
in each one-dimensional reconstruction step.

\subsubsection{Realizability-preservation}
While this already gives us a second-order scheme, we do not have the
realizability-preserving property yet. To achieve this, we need to apply a
realizability limiter, ensuring that $\moments[\cellindex](\z)$ is point-wise
realizable at the interface nodes $\z
  \in\{\z_{\cellindex-\frac12},\z_{\cellindex+\frac12} \}$. We follow the construction
from \cite{Alldredge2015}.

We replace $\moments[\cellindex]$ with the limited version
\begin{equation}\label{eq:realizabilitylimiting}
  \momentslimited{\cellindex} = \limitervariable\momentscellmean{\cellindex}
  +(1-\limitervariable)\moments[\cellindex] =
  \momentscellmean{\cellindex}+(1-\limitervariable)\moments[\cellindex]'\left(\z-\z_\cellindex\right).
\end{equation}
The limiter variable $\limitervariable\in[0,1]$ dampens the reconstruction from
unlimited ($\limitervariable=0$) to first-order ($\limitervariable=1$). Assuming
that $\momentscellmean{\cellindex}\in\RD{\basis}{}$ is realizable, there exists at
least one $\limitervariable$ (namely $\limitervariable = 1$) such that
$\momentslimited{\cellindex}(\z)\in\closure{\RD{\basis}{}}$ for every $\z$ in the set
of quadrature nodes (where \(\closure{\RD{\basis}{}}\) is the closure of
\(\RD{\basis}{}\)).
Since the realizable set is a convex cone, and by
continuity, it is guaranteed that there exists a minimal $\limitervariable$
satisfying this assumption. We are thus searching for the solution of the
minimization problem
\begin{equation}
  \max\limits_{\z \in\{\z_{\cellindex-\frac12},\z_{\cellindex+\frac12}
    \}}\min\limits_{\limitervariable\in[0,1]} \limitervariable \\
  \quad \text{ s. t. }\quad \momentslimited{\cellindex}(\z) \in \closure{\RD{\basis}{}}
\end{equation}
In practice, given some interface node $\z$, we search for the intersection of
the line $\momentslimited{\cellindex}(\z)$ (wrt.\ $\limitervariable$) with
the boundary of the realizable set, check if the value is in
$[0,1]$ and store it in the case that it is.

For the presented first-order moment models, the solution of the above limiter
problem can often be computed explicitly (see Section
\ref{sec:Limiting} for more details).

If we discretize \eqref{eq:FV1} with a second-order SSP Runge-Kutta (RK) scheme,
e.g.\ Heun's
method or the general $\rkstages$ stage SSP ERK${}_2$
\cite{Gottlieb2005,Ketcheson2008}, a realizability-preserving scheme is obtained
under a CFL-like condition if reconstruction and limiting is performed in every
stage of the RK method, see Lemma \ref{lem:CFLcond}.

\subsubsection{Solving the optimization problem}\label{subsection:OptimizationProblem}
For the minimum-entropy models, in each stage of the time stepping scheme
for~\eqref{eq:SplittedSystema}, we
have to solve the optimization problem~\eqref{eq:OptProblem} once in each cell (to compute the
Jacobians) and twice at each interface of the computational mesh (one optimization problem for the
left and right reconstructed value at the interface, respectively). This usually accounts for
the majority of computation time which makes it mandatory to pay special
attention to the implementation of the optimization algorithm. In this section,
we will focus on the stopping criteria for the optimization algorithm.
For details on the implementation,
see \cref{subsection:OptimizationProblemDetails}.

Recall that the objective function in the dual problem~\eqref{eq:dual} is
\begin{equation}\label{eq:optObjective}
  \optObjective_{\moments}(\multipliers) = \ints{\ld{\entropy}(\basis \cdot \multipliers)} - \moments \cdot \multipliers.
\end{equation}
The gradient and Hessian of \(\optObjective\) are given by
\begin{equation}\label{eq:optGradient}
  \optGradient_{\moments}(\multipliers) = \multipliersGradient \optObjective(\multipliers) = \ints{\basis \ld{\entropy}'(\basis \cdot \multipliers)} -
  \moments
\end{equation}
and
\begin{equation}\label{eq:optHessian}
  \optHessian(\multipliers) = \jacobianop_{\multipliers} \optGradient(\multipliers) =
  \ints{\basis \basis^T \ld{\entropy}''\left(\basis \cdot \multipliers\right)},
\end{equation}
respectively. Note that \(\ld{\entropy}'' > 0\) since we assumed that \(\entropy\) (and
thus also \(\ld{\entropy}\)) is
strictly convex, and remember that
the basis functions contained in \(\basis\) are linearly independent. As a consequence, the Hessian
\(\optHessian\)
is symmetric positive definite.

To find a minimizer of \(\optObjective\), we are searching for a root of the gradient
\(\optGradient\) using Newton's method.
For simplicity, we will restrict ourselves
to Maxwell-Boltzmann entropy~\eqref{eq:EntropyM} such that
\(\ld{\entropy}'(\multipliers \cdot \basis) = \exp(\multipliers \cdot \basis)\).
We use the Newton algorithm from \cite{Schneider2016,Alldredge2015,Schneider2016c}
with some adaptions. Before entering the algorithm
for the moment vector \(\moments\) we rescale it to
\begin{equation}
  \normalizedmoments \coloneqq \frac{\moments}{\density(\moments)}
\end{equation}
such that
\(\density(\normalizedmoments) = 1\). Let
\(\densityMult(\multipliers) = \ints{\exp(\basis \cdot \multipliers)}\) be the mapping \(\multipliers\mapsto
\density(\moments(\multipliers[]))\) which maps a set
of multipliers \(\multipliers\) to the density of its associated moment.
If the optimization algorithm for \(\normalizedmoments\) stops at an iterate \(\multiplierstilde\),
we return
\begin{equation}\label{eq:beta_to_alpha}
  \numericalmultipliers = \multiplierstilde + \multipliersone
  \log\left(\frac{\density(\moments)}{\densityMult(\multiplierstilde)}\right)
\end{equation}
where \(\multipliersone\) is the multiplier with the
property that \(\multipliersone \cdot \basis \equiv 1\) (see~\cref{sec:sourcesystem}).
This ensures that the local particle density is preserved exactly:
\begin{equation}\label{eq:densityadjustmultipliers}
  \densityMult\left(\numericalmultipliers\right) =
  \ints{\exp(\basis \cdot \numericalmultipliers)} = \ints{\exp(\basis \cdot \multiplierstilde)}
  \frac{\density(\moments)}{\densityMult(\multiplierstilde)} =
  \density(\moments).
\end{equation}
Given
\(\opttol\in\Rpos\), \(\opttoleps\in(0,1)\), we will stop the Newton iteration
at iterate \(\multiplierstilde\) if
\begin{subequations}\label{eq:stoppingcriteria}
  \begin{align}\label{eq:firststoppingcriterion}
    (1) & \ \ \norm{\optGradient_{\normalizedmoments}(\multiplierstilde)}{2} < \opttolmod \coloneqq
    \begin{cases}
      \frac{\opttol}{\left(1+\norm{\normalizedmoments}{2}
      \right)\density(\moments) + \opttol}                     & \text{ if }\quad  \basis = \fmbasis                \\
      \frac{\opttol}{\left(1+\sqrt{\momentnumber}\norm{\normalizedmoments}{2}
      \right)\density(\moments) + \sqrt{\momentnumber}\opttol} & \text{ if }\quad  \basis\in\{\hfbasis, \pmbasis\},
    \end{cases} \quad \text{and }                                                         \\
    (2) & \ \ \moments - (1-\opttoleps) \moments(\numericalmultipliers)               \in\RDpos{\basis},
    \label{eq:thirdstoppingcriterion}
  \end{align}
\end{subequations}
where \(\numericalmultipliers\) is obtained from \(\multiplierstilde\) by~\eqref{eq:beta_to_alpha}
and, as always, \(\momentnumber\) is the number of moments.

In the following, we will explain the rationale behind these stopping criteria.

The first criterion guarantees that the gradient of the objective function is sufficiently small.
\begin{lemma}\label{lem:StoppingCrit1}
  Let \(\opttol\in\Rpos\). If~\eqref{eq:firststoppingcriterion} is fulfilled,
  we have that \(\norm{\optGradient_{\moments}(\numericalmultipliers)}{2} \leq \opttol\).
\end{lemma}
\begin{proof}
  First note that, by its definition~\eqref{eq:optGradient}, the gradient can be written as
  \begin{equation}
    \optGradient_{\moments}(\numericalmultipliers) = \ints{\basis \exp(\basis \cdot \numericalmultipliers)} - \moments = \moments(\numericalmultipliers) - \moments
  \end{equation}
  where \(\moments(\numericalmultipliers)\) is the moment vector corresponding to the multipliers
  \(\numericalmultipliers\).

  Let
  \(\multipliersmod \coloneqq \multiplierstilde-\multipliersone \log(\densityMult(\multiplierstilde))\).
  Then it follows that
  \begin{align*}
    \density(\moments)\norm{\optGradient_{\normalizedmoments}(\multipliersmod)}{2}
     & = \density(\moments)\norm{\left(\ints{\basis\exp(\basis \cdot \multipliersmod)}-\normalizedmoments\right)}{2}      \\
     & = \norm{\ints{\basis\exp\Big(\basis \cdot \big(\multipliersmod
    + \multipliersone \log(\density(\moments))\big)\Big)} - \density(\moments)\normalizedmoments}{2}                      \\
     & = \norm{\ints{\basis\exp(\basis \cdot \numericalmultipliers)} - \density(\moments)\normalizedmoments}{2}           \\
     & = \norm{\moments(\numericalmultipliers) - \moments}{2} = \norm{\optGradient_{\moments}(\numericalmultipliers)}{2}.
  \end{align*}
  Consequently, we have
  \begin{align}
    \begin{split} \label{eq:prooffirststopping}
      \norm{\optGradient_{\moments}(\numericalmultipliers)}{2} =
      \density(\moments) \norm{\optGradient_{\normalizedmoments}(\multipliersmod)}{2}
      &= \density(\moments) \norm{\moments(\multipliersmod) - \normalizedmoments}{2} \\
      &= \density(\moments) \norm{\frac{1}{\densityMult(\multiplierstilde)} \moments(\multiplierstilde)
        - \normalizedmoments}{2} \\
      &= \density(\moments) \norm{\frac{1}{\densityMult(\multiplierstilde)} \moments(\multiplierstilde)
        - \frac{1}{\densityMult(\multiplierstilde)} \normalizedmoments
        + \frac{1}{\densityMult(\multiplierstilde)} \normalizedmoments
        - \normalizedmoments}{2} \\
      &\leq \density(\moments) \left(\frac{1}{\densityMult(\multiplierstilde)} \norm{\moments(\multiplierstilde)
        - \normalizedmoments}{2}
      + \norm{\frac{1}{\densityMult(\multiplierstilde)} \normalizedmoments
        - \normalizedmoments}{2} \right) \\
      &= \density(\moments) \left(\frac{1}{\densityMult(\multiplierstilde)}
      \norm{\optGradient_{\normalizedmoments}(\multiplierstilde)}{2}
      + \frac{\abs{1-\densityMult(\multiplierstilde)}}{\densityMult(\multiplierstilde)}
      \norm{\normalizedmoments}{2}\right).
    \end{split}
  \end{align}
  Moreover,
  \begin{align}
    \begin{split}\label{eq:densityminusoneformula}
      \abs{\densityMult(\multiplierstilde) - 1}
      &= \abs{\densityMult(\multiplierstilde) - \density(\normalizedmoments)} \\
      &= \abs{\density(\moments(\multiplierstilde)) - \density(\normalizedmoments)} \\
      &= \abs{\multipliersone \cdot \left(\moments(\multiplierstilde) - \normalizedmoments\right)} \\
      &= \begin{cases}
        \abs{\sum_{\basisindex=0}^{\momentnumber-1}
          \big(
          \momentcomp{\basisindex}(\multiplierstilde) -
          \normalizedmomentcomp{\basisindex}
          \big)}
         & \text{ for hat functions,}                       \\[4pt]
        \abs{\sum_{\basisindexvar=0}^{\hankelhalfind-1}
          \big(\momentcomp{\basisindexvar,0}(\multiplierstilde) -
          \normalizedmomentcomp{\basisindexvar,0}
          \big)}
         & \text{ for partial moments,} \phantom{aaaaaaaaa} \\[4pt]
        \abs{
          \momentcomp{0}(\multiplierstilde) -
          \normalizedmomentcomp{0}}
         & \text{ for full moments,}                        \\
      \end{cases}
    \end{split}
  \end{align}
  where in the last step we used the explicit
  forms of \(\multipliersone\) for the different bases (see \cref{sec:sourcesystem}).
  Since
  \begin{equation}
    \moments(\multiplierstilde) - \normalizedmoments = \optGradient_{\normalizedmoments}(\multiplierstilde),
  \end{equation}
  it follows from~\eqref{eq:densityminusoneformula} and~\eqref{eq:firststoppingcriterion} that
  \begin{equation*}
    \abs{\densityMult(\multiplierstilde) - 1} \leq
    \norm{\optGradient_{\normalizedmoments}(\multiplierstilde)}{1} \leq
    \sqrt{\momentnumber} \norm{\optGradient_{\normalizedmoments}(\multiplierstilde)}{2}
    \leq \sqrt{\momentnumber} \opttolmod
  \end{equation*}
  for partial moments and hat functions, and
  \begin{equation*}
    \abs{\densityMult(\multiplierstilde) - 1} \leq
    \norm{\optGradient_{\normalizedmoments}(\multiplierstilde)}{\infty} \leq
    \norm{\optGradient_{\normalizedmoments}(\multiplierstilde)}{2}
    \leq \opttolmod
  \end{equation*}
  for full moments,
  which directly gives
  \begin{equation*}
    \frac{1}{\densityMult(\multiplierstilde)} \leq \frac{1}{1- \sqrt{\momentnumber} \opttolmod}
    \qquad \text{ and }\qquad
    \frac{1}{\densityMult(\multiplierstilde)} \leq \frac{1}{1-\opttolmod},
  \end{equation*}
  respectively.
  Inserting these bounds in~\eqref{eq:prooffirststopping}, we finally obtain
  \begin{equation*}
    \norm{\optGradient_{\moments}(\numericalmultipliers)}{2}
    \leq \density(\moments) \left( \frac{\opttolmod}{1- \sqrt{\momentnumber} \opttolmod}
    + \frac{\sqrt{\momentnumber} \opttolmod}{1- \sqrt{\momentnumber} \opttolmod}
    \norm{\normalizedmoments}{2} \right)
    = \density(\moments) \left(
    \frac{\opttolmod\left(1+\sqrt{\momentnumber}\norm{\normalizedmoments}{2}\right)}{1-\sqrt{
        \momentnumber}
      \opttolmod} \right)
    = \opttol
  \end{equation*}
  for partial moments and hat functions, and similarly for full moments,
  removing \(\sqrt{\momentnumber}\) accordingly.
\end{proof}
The second criterion~\eqref{eq:thirdstoppingcriterion} ensures
that the ansatz density~\eqref{eq:psiME} corresponding to the multiplier
\(\numericalmultipliers\) obtained
from the Newton iteration is close enough to a representing density for the moments \(\moments\).
\begin{lemma}\label{lem:stoppingcriterion2}
  Let \(\moments\in\RDpos{\basis}{}\) and let \(\opttoleps\in(0,1)\),
  \(\numericalmultipliers\in\R^\momentnumber\)
  such that the second stopping criterion~\eqref{eq:thirdstoppingcriterion} holds.
  Then there exists a representing distribution \(\distribution\) for \(\moments\),
  i.e.\ \(\moments = \ints{\basis \distribution}\), such that
  \begin{equation} \label{eq:distributionrequirement}
    \frac{\distribution}{\ansatz[\moments(\numericalmultipliers)]} =
    \frac{\distribution}{\ld{\entropy}'(\numericalmultipliers \cdot \basis)} \geq 1 - \opttoleps.
  \end{equation}
\end{lemma}
\begin{proof}
  If~\eqref{eq:thirdstoppingcriterion} is satisfied, there exists a
  positive distribution \(\distribution^\opttoleps\) such that
  \begin{equation}
    \ints{\distribution^\opttoleps \basis} =
    \moments - (1-\opttoleps) \moments(\numericalmultipliers).
  \end{equation}
  Then
  \begin{equation}\label{eq:representingdistproof}
    \distribution \coloneqq \distribution^\opttoleps + (1-\opttoleps) \ansatz[\moments(\numericalmultipliers)]
  \end{equation}
  is a positive distribution representing 
  \(\moments\) and satisfying~\eqref{eq:distributionrequirement}.
\end{proof}
In \cref{sec:timesteprestrictionsecondorder} we will use \cref{lem:stoppingcriterion2}
to show that the scheme is realizability-preserving although the optimization problems
are only solved approximately.
\begin{remark}
  Note that
  \begin{equation*}
    \moments - (1-\opttoleps) \moments = \opttoleps \moments
  \end{equation*}
  is realizable for all \(\opttoleps > 0\). Due to the openness of \(\RDpos{\basis}{}\), there
  exists a
  \(\delta > 0\) s.t.\ \(\altvariable{\moments} \in \RDpos{\basis}\) for all
  \(\altvariable{\moments}\) with
  \(\norm{\altvariable{\moments} - \opttoleps \moments}{2} < \delta\). Note further that
  \begin{align*}
    \norm{\moments - (1-\opttoleps) \moments(\numericalmultipliers) - \opttoleps \moments}{2}
     & = \norm{(1-\opttoleps)(\moments - \moments(\numericalmultipliers))}{2}     \\
     & = (1-\opttoleps) \norm{\moments - \moments(\numericalmultipliers)}{2}      \\
     & = (1-\opttoleps) \norm{\optGradient_{\moments}(\numericalmultipliers)}{2},
  \end{align*}
  so~\eqref{eq:thirdstoppingcriterion} is fulfilled if \(\norm{\optGradient_{\moments}(\numericalmultipliers)}{2} \leq
  \frac{\delta}{1-\opttoleps}\), i.e.\ if our numerical solution to the approximation problem is close
  enough to the exact solution. For moments \(\moments\) that are very close to the realizable
  boundary
  (so \(\delta\) is very small and in addition \(\optHessian\) may be very badly conditioned), we
  might
  not be able to achieve such an accuracy. In that case, we either use a regularized version of
  \(\moments\) (see~\eqref{eq:RegularizedMoments}) or disable linear reconstruction
  (see \cref{linear_reconstruction_step} in \cref{section:ImplementationDetails}). Choosing \(\opttoleps\)
  closer to \(1\) makes it easier to fulfil~\eqref{eq:thirdstoppingcriterion} at the expense of
  smaller time steps (see \cref{lem:CFLcond}). In our computations, we used the value
  \(\opttoleps = 0.01\) which worked well in practice.
\end{remark}

\subsubsection{Time-step restriction}\label{sec:timesteprestrictionsecondorder}
Now we are able to put all the things together to show that one forward-Euler step of our scheme
\eqref{eq:FV1} is indeed realizability-preserving.
\begin{lemma}
  \label{lem:CFLcond}
  The finite volume scheme \eqref{eq:FV1}, using the kinetic flux
  \eqref{eq:kineticFlux} and the stopping criteria from Section
  \ref{subsection:OptimizationProblem}, on a rectangular grid in $\dimension$ dimensions
  preserves realizability under the CFL-like condition
  \begin{equation}
    \timestep < \frac{1-\opttoleps}{2\sqrt{\dimension}} \gridwidthx.
  \end{equation}
\end{lemma}

\begin{proof}
  Adapted from \cite[Theorem~3.19]{Schneider2016}. As we are using time stepping schemes that consist of a
  convex combination of Euler forward steps, it
  is enough to show realizability preservation in a single Euler forward step.  Consider the
  one-dimensional ($\dimension=1$) case first. The update formula in one
  step is
  \begin{align*}
    \momentstime{j}{\timeindex+1}
     & = \momentstime{j}{\timeindex} - \frac{\timestep}{\gridwidthx}
    \left( \kineticFlux(\momentspv{j+\frac12}^{-,\opttol},\momentspv{j+\frac12}
      ^{+,\opttol})-\kineticFlux(\momentspv{j-\frac12}^{-,\opttol},\momentspv{
    j-\frac12 } ^{+,\opttol}) \right)                                                        \\
     & = \momentstime{j}{\timeindex} - \frac{\timestep}{\gridwidthx}
    \left( \intp{\SCheight\basis\ansatz[j+\frac12]^{-,\opttol}} +
    \intm{\SCheight \basis \ansatz[j+\frac12]^{+,\opttol}} -
    \intp{\SCheight\basis\ansatz[j-\frac12]^{-,\opttol}} - \intm{\SCheight \basis
    \ansatz[j-\frac12]^{+,\opttol}} \right) \nonumber                                        \\
     & = \ints{\basis\distribution[j]^{(\timeindex)}} -
    \frac{\timestep}{\gridwidthx} \bigg( \ints{ \max(\SCheight, 0) \basis
    \left(\ansatz[j+\frac12]^{-,\opttol} - \ansatz[j-\frac12]^{-,\opttol} \right)}           \\
     & \hspace{7cm}+ \ints{ \min(\SCheight, 0) \basis \left(\ansatz[j+\frac12]^{+,\opttol} -
    \ansatz[j-\frac12]^+ \right)} \bigg)  \nonumber                                          \\
     & = \ints{\basis \left(\distribution[j]^{(\timeindex)} -
      \frac{\timestep}{\gridwidthx} \left(\max(\SCheight, 0)
      \left(\ansatz[j+\frac12]^{-,\opttol} - \ansatz[j-\frac12]^{-,\opttol} \right) +
      \min(\SCheight, 0) \left(\ansatz[j+\frac12]^{+,\opttol} -
    \ansatz[j-\frac12]^{+,\opttol} \right) \right) \right)}  \nonumber                       \\
     & \eqqcolon \ints{\basis
      \distribution[j]^{(\timeindex+1)}} \nonumber
  \end{align*}
  where $\distribution[j]^{(\timeindex)}$ is an arbitrary representing density for
  $\momentstime{j}{\timeindex}$ and $\ansatz[j-\frac12]^{+,\opttol}$ is the
  ansatz distribution obtained from the approximate solution of the optimization
  problem. To preserve realizability, we have to ensure that
  $\distribution[j]^{(\timeindex+1)} \geq 0$ for all $\SCheight \in [-1, 1]$ and all
  cells $j$.

  For $\SCheight > 0$, after stripping away positive terms and using
  $\SCheight \leq 1$,
  we have
  \begin{equation}\label{eq:schemepreservingproof}
    \distribution[j]^{(\timeindex+1)} \geq \ansatz[j]^{(\timeindex)} -
    \frac{\timestep}{\gridwidthx} \ansatz[j+\frac12]^{-,\opttol}
    \geq \ansatz[j]^{(\timeindex)} -
    \frac{\timestep}{\gridwidthx} \frac{\distribution[j+\frac12]^-}{1-\opttoleps},
  \end{equation}
  where $\distribution[j+\frac12]^-$ is the distribution from
  \eqref{eq:distributionrequirement}.

  We have that
  \begin{equation}
    \momentspv{j\pm\frac12}^{\mp} = \momentstime{j}{\timeindex} \pm \frac12
    \moments[\cellindex]'
  \end{equation}
  where $ \moments[\cellindex]'$ is the (limited) slope on cell $\cellindex$.
  Thus we have
  \begin{equation}
    \momentstime{j}{\timeindex} = \frac{\momentspv{j+\frac12}^{-} +
      \momentspv{j-\frac12}^{+}}{2}
  \end{equation}
  and therefore a representing density for $\momentstime{j}{\timeindex}$ is
  $\frac{\distribution[j+\frac12]^- + \distribution[j-\frac12]^+}{2}$.
  Inserting this in
  \eqref{eq:schemepreservingproof} gives
  \begin{equation}\label{eq:schemepreservingproof2}
    \distribution[j]^{(\timeindex+1)}
    \geq \frac{\distribution[j+\frac12]^- +
      \distribution[j-\frac12]^+}{2} - \frac{\timestep}{\gridwidthx}
    \frac{\distribution[j+\frac12]^-}{1-\opttoleps}
    = \left(\frac12 - \frac{\timestep}{\gridwidthx(1-\opttoleps)}\right)
    \distribution[j+\frac12]^- +
    \frac{\distribution[j-\frac12]^+}{2}
  \end{equation}
  This is positive under the time step restriction
  \begin{equation}
    \timestep < \frac{(1-\opttoleps)}{2} \gridwidthx.
  \end{equation}
  The case $\SCheight \leq 0$ follows in a similar way.

  In $\dimension$ dimensions, the update formula changes to
  \begin{equation}
    \momentstime{\mathbf{j}}{\timeindex+1}
    = \momentstime{\mathbf{j}}{\timeindex} - \sum_{l=1}^{\dimension}
    \frac{\timestep}{\gridwidthx_l} \left(
    \kineticFlux_l(\momentspv{j_l+\frac12}^-,\momentspv{j_l+\frac12}
      ^+)-\kineticFlux_l(\momentspv{j_l-\frac12}^-,\momentspv{j_l-\frac12}^+) \right)
  \end{equation}
  where $\mathbf{j} = (1, \ldots, \dimension)^T$ is an index tuple. As in one
  dimension, we define the representing density
  $\distribution[\mathbf{j}]^{(\timeindex+1)}$ and only regard the case $\SC_l >
    0\, \forall l$, the other cases follow
  similarly. After stripping away positive
  terms we are left with
  \begin{equation}\label{eq:schemepreservingproof3d}
    \distribution[\mathbf{j}]^{(\timeindex+1)} \geq
    \distribution[\mathbf{j}]^{(\timeindex)} -
    \frac{\timestep}{\gridwidthx} \sum_{l=1}^{\dimension} \SC_l
    \frac{\distribution[j_l+\frac12]^-}{1-\opttoleps},
  \end{equation}
  where $\gridwidthx = \min_l \gridwidthx_l$.
  We proceed as in one dimension and note that
  \begin{equation}
    \distribution[\mathbf{j}]^{(\timeindex)} = \sum_{l=1}^\dimension w_l
    \frac{\distribution[j_l+\frac12]^- + \distribution[j_l-\frac12]^+}{2}
  \end{equation}
  is a representing density for $\momentstime{\mathbf{j}}{\timeindex}$ for
  any
  partition of unity $\sum_{l=1}^\dimension w_l = 1$.
  Inserting this ansatz in \eqref{eq:schemepreservingproof3d} gives
  \begin{equation}
    \distribution[\mathbf{j}]^{(\timeindex+1)}
    \geq \sum_{l=1}^{\dimension} \left(\frac{w_l}{2} - \SC_l
    \frac{\timestep}{\gridwidthx(1-\opttoleps)}\right)
    \distribution[j_l+\frac12]^- + \sum_{l=1}^{\dimension} w_l
    \frac{\distribution[j_l-\frac12]^+}{2}
  \end{equation}
  This is positive if
  \begin{equation}\label{eq:timesteprestrictionrealizability3d}
    \frac{\timestep}{\gridwidthx} < \min_l \frac{1-\opttoleps}{2} \frac{w_l}{\SC_l}
    \hspace{0.5 cm} \forall \, \SC \text{ with } \SC_l > 0 \ \forall\,
    l=1,\ldots,\dimension.
  \end{equation}
  So for given $\SC$ we have to find a partition of unity $\mathbf{w}$
  such
  that the right-hand side of \eqref{eq:timesteprestrictionrealizability3d} is
  maximal, i.e., we want to find
  \begin{equation}\label{eq:partitionofunityoptproblem}
    \min_{\norm{\boldsymbol{\SC}}{2} \leq 1} \
    \max_{\norm{\mathbf{w}}{1} = 1} \ \min_{l \in \{1,\ldots,\dimension\}} \
    \frac{w_l}{\SC_l}
  \end{equation}
  Obviously, the maximum is attained if $\frac{w_{l_1}}{\SC_{l_1}}
    = \frac{w_{l_2}}{\SC_{l_2}}$ for all $l_1, l_2$
  (otherwise we could
  increase the $w_l$ which belongs to the minimum and decrease the other ones a
  little). Taking the partition of unity property into account, we thus have to
  choose $w_l = \frac{\SC_l}{\norm{\SC}{1}}$. Inserting this in
  \eqref{eq:partitionofunityoptproblem} gives
  \begin{equation*}
    \min_{\norm{\boldsymbol{\SC}}{2} \leq 1} \
    \max_{\norm{\mathbf{w}}{1} = 1} \ \min_{l \in \{1,\ldots,\dimension\}} \
    \frac{w_l}{\SC_l}
    =
    \min_{\norm{\boldsymbol{\SC}}{2}\leq 1}
    \frac{1}{\norm{\boldsymbol{\SC}}{1}}
    \leq
    \min_{\boldsymbol{\SC}}
    \frac{\norm{\boldsymbol{\SC}}{2}}{\norm{\boldsymbol{\SC}}{1}}
    = \frac{1}{\sqrt{\dimension}}.
  \end{equation*}

  Using this in \eqref{eq:timesteprestrictionrealizability3d}, we end up with the
  time-step restriction
  \begin{equation*}
    \timestep < \frac{1-\opttoleps}{2\sqrt{\dimension}} \gridwidthx.
  \end{equation*}
\end{proof}

\section{Implementation details}\label{section:ImplementationDetails}
We implemented the whole scheme in the generic \Cpp framework 
\texttt{DUNE}~\cite{Bastian2008,Bastian2008a}, more specifically in the \texttt{DUNE}
generic
discretization toolbox \texttt{dune-gdt}~\cite{Schindler2017} and the
\texttt{dune-xt}-modules~\cite{Milk2017, Milk2017a}.

As mentioned above, we advance the flux system in time using Heun's method, which is a second-order
strong-stability preserving
Runge-Kutta scheme~\cite{Gottlieb2005}. In each stage of the Runge-Kutta scheme, we perform the
following steps:
\begin{enumerate}
  \item Solve the optimization problem for the cell means \(\momentscellmean{\cellindex}\) in each
        grid cell (see \cref{subsection:OptimizationProblemDetails}). If regularization is needed, replace
        \(\momentscellmean{\cellindex}\) by its regularized
        version\footnote{This formally destroys the consistency of the scheme.
          However, since regularization rarely occurs (and only near the realizability boundary), this effect can usually be neglected in practice.}
        (see \cref{sec:regularization}).
  \item Reconstruct the values at the cell interfaces using linear reconstruction in
        characteristic variables (see \cref{subsubsec:polynomial_reconstruction}), using the
        solution of the optimization problems from step 1 to calculate the Jacobians.
  \item Perform the realizability limiting (see \cref{sec:Limiting}).
  \item Solve the optimization problem for all reconstructed values \(\moments[\cellindex\pm\frac{1}{2}]\).
        If the solver fails for a reconstructed value, disable the
        linear reconstruction in that cell.\label{linear_reconstruction_step}
  \item Evaluate the kinetic flux~\eqref{eq:kineticFlux} and update the stage values according
        to~\eqref{eq:FV1}.
\end{enumerate}

In the following, we will give some details on the implementation of these steps.

\subsection{Implementation of the minimum-entropy solver}\label{subsection:OptimizationProblemDetails}
Our solver for the optimization problem is based on the algorithm
from~\cite{Alldredge2014}. It uses a Newton-type algorithm
with Armijo line search, i.e.\
to find a minimizer of the objective function \(\optObjective\) (see \eqref{eq:optObjective}),
we are searching for a root of the
gradient \(\optGradient\) (see \eqref{eq:optGradient})
in the Newton direction \(\optNewtonDirection(\multipliers)\) which solves
\begin{equation}
  \optHessian(\multipliers)\optNewtonDirection(\multipliers) = -\optGradient(\multipliers).
\end{equation}
and then update the multipliers as
\begin{equation}
  \multipliers[k+1] = \multipliers[k] + \zeta_k \optNewtonDirection(\multipliers[k])
\end{equation}
where \(\zeta_k\) is determined by a backtracking line search such that
\begin{equation}
  \optObjective(\multipliers[k+1]) < \optObjective(\multipliers[k])
  + \xi \zeta_k {\optGradient(\multipliers[k])} \cdot \optNewtonDirection(\multipliers[k])
\end{equation}
with \(\xi \in (0,1)\).

We stop the optimization if the new iterate \(\multipliers[k+1]\) satisfies the
stopping criteria \eqref{eq:stoppingcriteria}, except that we use
\begin{equation}
  \norm{\optGradient_{\normalizedmoments}(\multiplierstilde)}{2} < \min(\opttolmod, \opttol)
\end{equation}
as the first stopping criterion instead of simply using~\eqref{eq:firststoppingcriterion}.
This avoids numerical difficulties for moments
with small density, where \(\opttolmod\) is in the order of \(1\) and thus some
iterates \(\multiplierstilde\) with very large (in absolute values) entries might
fulfil the stopping criterion by chance. Moreover, checking the second stopping
criterion~\eqref{eq:thirdstoppingcriterion} might be
quite expensive (depending on the basis \(\basis\)). We therefore check this criterion only if
additionally
\begin{equation}\label{eq:secondstoppingcriterion}
  1-\opttoleps < \exp(-\left(\norm{\newtonDirection(\multiplierstilde)}{1} +
    \abs{\log{\densityMult(\multiplierstilde)}} \right))
\end{equation}
holds. This criterion approximately ensures~\eqref{eq:distributionrequirement}
(see~\cite{Schneider2016, Alldredge2014}) but, in general, is much easier
to
evaluate than~\eqref{eq:thirdstoppingcriterion}. For the \(\HFMN\) models, however,
checking realizability is just checking positivity, so in that case we do not
need to check~\eqref{eq:secondstoppingcriterion} first.

To improve the performance and stability of the algorithm, we use
several additional techniques which we will detail in the following.
The values of the algorithms' parameters that we use in all computations are given in
\tabref{tab:parameters}.
\begin{table}[htb]
  \renewcommand{\arraystretch}{1.3}
  \begin{center}
    \addtolength{\tabcolsep}{-1.5pt}
    \begin{tabular}{cccccccc}
      \toprule
      \multicolumn{8}{c}{Newton algorithm}                                                                                                                    \\
      \cmidrule{1-8}
      $k_0$                 & $k_{max}$ & $\opttoleps$   & $\epsilon$               & $\chi$ & $\xi$            &
      $\tau$                & $\{r_l\}$                                                                                                                       \\
      $500$                 & $1000$    & $10^{-2}$      & $2^{-52}$                & $1/2$  & $10^{-3}$        & $10^{-9} $ &
      $\{0,10^{-8},10^{-6},10^{-4}, 10^{-3}$                                                                                                                  \\
                            &           &                &                          &        &                  &            & $10^{-2}, 0.05, 0.1, 0.5, 1\}$ \\
      \toprule
                            &           & \multicolumn{2}{c}{Realizability limiter} &        & \multicolumn{2}{c}{Minima}                                     \\
      \cmidrule{3-4} \cmidrule{6-7}
                            &           & $\limitereps $ & $\limiterepstilde$       &        & $\densityvacuum$ &
      $\distributionvacuum$ &                                                                                                                                 \\
                            &           & $10^{-11}$     & $10^{-11}$               &        & $10^{-8}$        &
      $\densityvacuum/\ints{1}$                                                                                                                               \\
      \bottomrule
    \end{tabular}
    \caption{Parameter choice for the different aspects of the simulation. Notation for the
      Newton algorithm as in \cite{Alldredge2014}.}
    \label{tab:parameters}
  \end{center}
\end{table}

\subsubsection{Adaptive change of basis}\label{sec:adaptivechangeofbasis}
Though the Hessian \(\optHessian(\multipliers)\) is positive definite and thus invertible, it may be
very badly conditioned,
especially for multipliers \(\multipliers\) corresponding to moments \(\moments(\multipliers)\)
close to the boundary
of the realizable set. Moreover, in general, the integral in the definition~\eqref{eq:optHessian}
of \(\optHessian\)
can only be calculated approximately using a numerical quadrature (see \cref{sec:QuadratureRules}).
If the quadrature is not sufficiently accurate,
the approximate Hessian may have a significantly worse condition or may even be numerically
singular.

To improve this situation, a change of
basis can be performed after each Newton iteration such that
the Hessian at the current iterate becomes the unit matrix
in the new basis \cite{Alldredge2014}. We use this procedure in our
implementation for all bases except for the
hat function bases \(\hfbasis\).

For the hat function basis, all matrices and vectors required in the optimization algorithm
are sparse and exploiting this fact in the implementation greatly speeds up the
computations. Including the change of basis destroys the sparsity and thus harms
performance. In theory, this could be compensated by faster convergence and thus less iterations of
the algorithm due to the condition improvements. Further, the algorithm with change of basis
might use regularization less frequently and thus introduce less errors in the solution, as shown
for the full moments in~\cite{Alldredge2014}. We thus compared the algorithm with and without
change of basis in several test problems. The differences in the results were negligible in all
tests cases and the version without change of basis was significantly faster. We thus do not use
the adaptive change of basis for the hat functions.

The first-order partial moments have a similarly simple structure as the hat functions, so the
adaptive change of
basis might also not be needed for these models.
However, the change of basis does not have a significant performance impact in this case
as the support of each basis function is
restricted to a single interval or spherical triangle, and thus all matrix
operations can be performed on the \(2\times2\) or \(4\times4\)
submatrices corresponding to an
interval in 1d and a spherical triangle in 3d, respectively. Similar, quadrature
evaluations can be performed for each interval or spherical triangle separately.
For this reason, we include the adaptive change of basis in our optimization
algorithm for the partial moments.

For details on the change
of basis algorithm see~\cite{Alldredge2014}.
Note that the stopping criteria \eqref{eq:stoppingcriteria} are
computed in the original basis such that
knowledge of the change of basis algorithm is not required
to understand the presentation in this paper.

\subsubsection{Regularization}\label{sec:regularization}
For tests with strong absorption,
the local particle density may become very small in parts of the domain.
As a consequence, also the entries
of the Hessian \(\optHessian\) become very small
which may cause numerical problems.
We thus choose a ``vacuum'' density
\(\distributionvacuum\) with corresponding local particle density
\(\densityvacuum = \ints{\distributionvacuum}\). We then enforce
a minimum local particle density of \(\densityvacuum\)
by replacing moments \(\moments\) with local particle density
\(\density(\moments) < \densityvacuum\) by the
isotropic moment with vacuum density \(\densityvacuum\).
Obviously, this approach leads to a violation of the conservation properties of the scheme.
However, since we only replace moments with
very small local particle densities by moments with slightly larger
but still very small densities, the effect should be negligible in practice.

Additionally,
if the optimizer fails for a moment vector \(\moments\)
(for example, by reaching a maximum number of iterations or being
unable to solve for the Newton direction)
we use the isotropic-regularization
technique from~\cite{Alldredge2014}, i.e.\
we replace \(\moments\) by the regularized moment vector
\begin{equation}\label{eq:RegularizedMoments}
  \regularizedmoments \coloneqq (1 - \regularizationParameter) \moments + \regularizationParameter \isomatrix
  \moments.
\end{equation}
and retry the optimization.
If the optimizer still fails,
we increase \(\regularizationParameter\) until the optimizer succeeds,
which is guaranteed at least for \(\regularizationParameter = 1\) where
\(\regularizedmoments\) is isotropic.
As the regularized moment vector \(\regularizedmoments\) always has the same local particle
density as the original moment vector \(\moments\),
this technique does not violate the mass conservation of the scheme
but it may potentially completely alter the solution.
In practice, regularization is only used rarely and if it is used, a
small regularization parameters is usually sufficient.

\subsubsection{Caching}\label{subsec:caching}
We use two types of caching.
First, for each grid cell we
store the moment vector 
from the last time step and
the corresponding multiplier 
obtained by entropy minimization. In this way we do not have to
solve the optimization problem again if the moment vector in that grid cell
did not change during the last time step. In addition, we store the
last few solutions of the minimization problem with corresponding input moment vectors
per thread of execution, so if several grid cells contain the
same values,
we only have to perform the optimization once and then use the cached values. If we encounter a
moment
vector that can not
be found in the caches, we take the moment vector that is closest to the input vector (in one-norm)
and use the corresponding multiplier
as an initial guess.

\subsubsection{Linear solvers}\label{sec:linearsolvers}
In each iteration of the Newton scheme,
we have to apply the inverse
of a positive definite Hessian matrix. We assemble the matrices
using the quadratures described in \cref{sec:QuadratureRules}.
Inversion is then done by computing a Cholesky factorization
of the assembled matrix. For the full moment models,
the Hessian matrices are dense, so we use
the \texttt{LAPACK}~\cite{Anderson1999} routine \texttt{dpotrf}
to compute the factorization and then use
\texttt{dtrsv} to actually invert the linear systems.
For the \(\PMMN\) models,
the Hessian is block-diagonal (each
block corresponds to one interval/triangle of the partition)
such that we can perform the Cholesky decomposition
independently for each block. For the \(\HFMN\)
models in one dimension, the Hessian matrices are
tridiagonal, so we can use the specialized \texttt{LAPACK}
algorithms \texttt{dpttrf} and \texttt{dpttrs}.
In three dimensions, the \(\HFMN\)
Hessians are not tridiagonal anymore
but still sparse, so we use the sparse
\texttt{SimplicialLDLT} solver from
the \texttt{Eigen} library~\cite{Guennebaud2010}.

\subsection{Solving the eigenvalue problems}\label{sec:eigenvalueproblems}
To avoid spurious oscillations, the reconstruction has to be performed in characteristic
coordinates (see \cref{subsubsec:polynomial_reconstruction}). For that reason, we have to
compute the eigenvectors of the flux Jacobians
\begin{equation}\label{eq:fluxjacobian}
  \Flux'(\momentscellmean{\cellindex})
  = \velocityHessian(\momentscellmean{\cellindex}) \optHessian^{-1}(\momentscellmean{\cellindex})
\end{equation}
where
\begin{equation}
  \velocityHessian(\moments) \coloneqq
  \ints{\SCheight\basis\basis^T\ld{\entropy}''\left(\basis \cdot \multipliers(\moments)\right)}
\end{equation}
and
\begin{equation}
  \optHessian(\moments) \coloneqq \optHessian(\multipliers(\moments)) = \ints{\basis \basis^T \ld{\entropy}''\left(\basis \cdot \multipliers(\moments)\right)}
\end{equation}
(compare \cref{subsection:OptimizationProblem}).
Note that, in general,
the Jacobian~\eqref{eq:fluxjacobian} is not symmetric. However,
since \(\velocityHessian\) is symmetric and \(\optHessian\) symmetric positive definite
(see \cref{subsection:OptimizationProblem}),
we can see that \(\Flux'(\momentscellmean{\cellindex})\) is similar to
a symmetric matrix, i.e.\
\begin{equation*}
  \optHessian^{-\frac{1}{2}}\Flux'\optHessian^{\frac{1}{2}}
  = \optHessian^{-\frac{1}{2}}\velocityHessian\optHessian^{-1}\optHessian^{\frac{1}{2}}
  = \optHessian^{-\frac{1}{2}}\velocityHessian\optHessian^{-\frac{1}{2}},
\end{equation*}
and thus has real eigenvalues. In our implementation, we explicitly compute the matrix
representation
and then use an eigensolver for non-symmetric matrices (\texttt{LAPACK}'s
\texttt{dgeevx}) to
obtain the eigen decomposition.
Unfortunately, though the Jacobian is a real matrix with real eigen values
and thus also admits a set of real eigenvectors, the
standard solvers for non-symmetric eigen problems (apart from \texttt{dgeevx},
we also tested the \texttt{EigenSolver} of the \texttt{Eigen}
library~\cite{Guennebaud2010})
often return complex eigenvectors. We thus add a step to compute
real eigenvectors from the complex ones. Note that if
\begin{equation}
  \Set{\Vz_{\basisindex} = \Vy_{2\basisindex} + \imaginaryunit \Vy_{2\basisindex+1} \given
    \basisindex = 0, \ldots, k-1}
\end{equation}
is a set of linearly independent complex eigenvectors to the same
eigenvalue \(\eigenvalue\) for the Jacobian
\(\Flux'\), where \(\imaginaryunit\) is the imaginary unit
and \(\Vy_{\basisindexvar} \in \R^{\momentnumber}\) are real vectors,
then
\begin{equation}\label{eq:eigenvectorset}
  \Set{\Vy_{\basisindexvar} \given \basisindexvar = 0,  \ldots, 2k-1}
\end{equation}
is a set of \(2k\) real eigenvectors for \(\Flux'\).
Moreover, there are at least \(k\) linearly independent vectors in this set. To see that,
assume the opposite, i.e.\ that any \(k\) vectors from the set~\eqref{eq:eigenvectorset}
are linearly dependent. Without loss of generality, we assume that
every vector in~\eqref{eq:eigenvectorset} can be written as a linear combination of
the first \(k-1\) vectors, i.e.\
\begin{equation}
  \Vy_{\basisindexvar} = \sum_{r=0}^{k-2} a_{\basisindexvar,r}
  \Vy_{r},
  \quad \basisindexvar=0,\ldots,2k-1,
\end{equation}
with coefficients \(a_{\basisindexvar,r} \in \R\).
Then, the \(k\) vectors \(\Vz_{\basisindexvar}\) can also be written
as (complex) linear combinations of these \(k-1\) real vectors
\begin{equation}
  \Vz_{\basisindex}
  =  \Vy_{2\basisindex} + \imaginaryunit \Vy_{2\basisindex+1}
  = \sum_{r=0}^{k-2} \left(a_{2\basisindex,r} + \imaginaryunit\, a_{2\basisindex + 1,r}\right) \Vy_{r},
  \quad \basisindex = 0, \ldots, k-1,
\end{equation}
and thus cannot be linearly independent.

Consequently, to get real eigenvectors for \(\Flux'\) from the complex ones
computed by the eigensolver, we first sort the eigenvectors into sets belonging
to the same eigenvalue and then perform a Gram-Schmidt process with the real and imaginary parts
for each of these sets.
\begin{remark}
  While this procedure works reasonably well, a better approach would probably be to use the
  structure
  of the Jacobian and, instead of solving the non-symmetric eigenvalue problem
  \begin{equation}
    \left(\Flux'\right) \Vz = \velocityHessian \optHessian^{-1} \Vz = \eigenvalue \Vz,
  \end{equation}
  solve the symmetric generalized eigenvalue problem~\cite{Martin1971,BunseGerstner1984}
  \begin{equation}
    \velocityHessian \tilde{\Vz} = \eigenvalue \optHessian \tilde{\Vz}
  \end{equation}
  and then get the eigenvectors as \(\Vz = \optHessian\tilde{\Vz}\).
  Since the matrices \(\velocityHessian\) and \(\optHessian\)
  are both symmetric and \(\optHessian\) is positive definite,
  we can use a specialized algorithm like \texttt{LAPACK}'s \texttt{dsygv}
  and directly obtain real eigenvectors.
  Moreover, we can take advantage of the sparsity of these matrices and, e.g.,
  use a generalized eigenvalue algorithm aimed at band matrices like \texttt{dsbgv}.
  In contrast, explicit assembly of the term \(\velocityHessian\optHessian^{-1}\)
  might destroy the structure and result in a dense matrix even if the
  two factor matrices are sparse.
\end{remark}
For the partial-moment models, the eigen decomposition
can be done block-wise on the \(2\times2\) or \(4\times4\) matrix blocks
which reduces the cubic complexity of the eigen decomposition~\cite{Pan1999} to a linear
complexity
for increasing number of moments and thus greatly accelerates computations
for large problems.

\subsection{Realizability limiting}\label{sec:Limiting}
The linear reconstruction process in the finite volume scheme does not guarantee preservation of
realizability. Thus, we need an additional limiting step \eqref{eq:realizabilitylimiting} to
ensure that we are able to solve the optimization problem \eqref{eq:OptProblem} for the
reconstructed values. Since, in general, we cannot solve the integrals occurring in the
optimization
problem analytically and have to approximate them by a numerical quadrature
$\angularQuadrature$,
the admissible moment vectors are further restricted to the numerically realizable set
($\angularQuadrature$-realizable set)
\begin{equation}\label{eq:numericallyrealizableset2}
  \RQ{\basis} = \left\{\moments~:~\exists \distribution(\SC)\ge 0,\, \density =
  \intQ{\distribution} > 0,
  \text{ such that } \moments =\intQ{\basis\distribution} \right\} \subset
  \closure{\RD{\basis}{}},
\end{equation}
where
for an integrable function $\angularQuadratureFunction$, $\intQ{\angularQuadratureFunction}
  = \sum_{\angularQuadratureIndex=0}^{\angularQuadratureNumber-1}
  \angularQuadratureWeights{\angularQuadratureIndex}
  \angularQuadratureFunction(\angularQuadratureNodesS{\angularQuadratureIndex})\approx\ints{
    \angularQuadratureFunction}$ is the
approximation of the corresponding integral $\ints{\cdot}$ with
the quadrature rule $\angularQuadrature$. In general, the numerically realizable set is a strict
subset of the analytically realizable set.

The numerically realizable set can be described as the convex hull of the basis function values at
the quadrature nodes (see \cite{Alldredge2014} for the Legendre basis, the proof can be easily
adapted
for the other bases)
\begin{equation}\label{eq:convexhullrealizableset}
  \RQ{\basis}|_{\density = 1} =
  \interior{\convexhull{\{\basis(\SC_\angularQuadratureIndex)\}_{
  \angularQuadratureIndex=0}^{\angularQuadratureNumber-1}
  \}} }.
\end{equation}
If $\density$ depends linearly on $\moments$ it follows
\begin{equation}\label{eq:convexhullrealizablesetless1}
  \RQ{\basis}|_{\density < 1} =
  \interior{\convexhull{\mathbf{0}, \{\basis(\SC_\angularQuadratureIndex)\}_{
  \angularQuadratureIndex=0}^{\angularQuadratureNumber-1}
  \}} }.
\end{equation}
We do not want the limited moments to be too close to to the boundary of the numerically realizable
set as we are not able to solve the optimization problem \eqref{eq:OptProblem} in that case (see
\cite{Alldredge2012}).
Moving the limited value away from the boundary can be done in several ways.
A simple but often sufficient method can be employed for all limiters presented in this section. We
simply add a small parameter $\limiterepstilde$ to the final limiter variable
$\limitervariable$
\cite{Schneider2016}. A problem with this approach is that the connecting line between
$\moments$
and $\momentscellmean{}$ might be almost parallel to the boundary which possibly results in a
limited moment that is still too close to the boundary.
Another approach is to require a fixed distance $\limitereps$ to the boundary of
$\RQ{\basis}$,
i.e., to limit to the $(\angularQuadrature,\limitereps)$-realizable set
\begin{equation}\label{eq:epsnumericallyrealizableset}
  \RQeps{\basis} = \left\{\moments \in \RQ{\basis}
  \text{ such that } \distancebd(\moments, \partial\RQ{\basis}) \geq \limitereps \right\},
\end{equation}
where $\distancebd(\cdot, \partial\RQ{\basis})$ is the Euclidian distance to $\partial\RQ{\basis}$.
Limiting to this set is possible whenever $\momentscellmean{}$ is farther than
$\limitereps$ away
from the boundary. If $\momentscellmean{}$ is already in the $\limitereps$-range of
the boundary,
we disable reconstruction in that cell.

Unfortunately, checking whether a reconstructed value lies within the numerically realizable set is
not trivial in general. In the following, we detail the limiting procedure for the different
models. For the remainder of this section, let $\momentscellmean{}$ be the moment vector before
reconstruction and $\moments$ a reconstructed moment vector \footnote{In one dimension, there are
  always two reconstructed values per grid cell (one at each interface), so each of the limiters
  described in the following is applied to both values and the larger $\limitervariable$ is used
  in
  \eqref{eq:realizabilitylimiting}. In several dimensions, both reconstruction and limiting are
  performed independently for each coordinate direction.}.
Let further
$\momentcompcellmean{\basisindex}$ and $\momentcomp{\basisindex}$ be the $\basisindex$-th component of
$\momentscellmean{}$ and $\moments$, respectively.

\subsubsection{\texorpdfstring{$\Mn$ models}{MN models}}
In~\cite{Alldredge2015,Schneider2016},
the half space representation
for the convex hull \eqref{eq:convexhullrealizablesetless1}
was explicitly calculated before starting the time stepping,
yielding
\begin{equation}\label{eq:halfspacerealizableset}
  \RQone{\basis}
  = \Set{\moments\in\R^\momentnumber
    \given \RQonehalfspacematcol_{\facetindex} \cdot \moments <
    \RQonehalfspaceveccomp_\facetindex,\
    \facetindex \in \{0,\ldots,\numfacets-1\}},
\end{equation}
where \(\numfacets\) is the number of facets of the convex hull. During the time
stepping,
the intersection
of the connecting line between
\(\moments\) and \(\momentscellmean{}\) and each facet
can then be computed efficiently by
solving
\begin{equation*}
  \RQonehalfspacematcol_{\facetindex} \cdot \momentslimited{} = \RQonehalfspaceveccomp_{\facetindex}
\end{equation*}
(compare~\eqref{eq:realizabilitylimiting})
for the limiter variable. We thus obtain
\begin{align} \label{eq:halfspacelimiter}
  \limitervariable = \max_{\facetindex=0,\ldots,\numfacets-1} \limitervariable_\facetindex,
  \qquad\qquad
  \limitervariable_\facetindex & =
  \begin{cases} \frac{\RQonehalfspaceveccomp_\facetindex - \RQonehalfspacematcol_\facetindex \cdot \moments}{
      \RQonehalfspacematcol_\facetindex \cdot
      \left(\momentscellmean{} - \moments\right)}
      & \text{if } \frac{\RQonehalfspaceveccomp_\facetindex
      - \RQonehalfspacematcol_\facetindex \cdot \moments}{ \RQonehalfspacematcol_\facetindex \cdot
    \left(\momentscellmean{} - \moments\right)} \in [0,1],  \\
    0 & \text{else.}
  \end{cases}
\end{align}
If \(\density\left(\moments\right) \geq 1\) or
\(\density\left(\momentscellmean{}\right) \geq 1\), the moments can
simply be rescaled before applying the limiter~\cite{Alldredge2015,Schneider2016}.
Alternatively,
we can ignore the facet corresponding to the condition
\(\density = \multipliersone \cdot \moments \leq 1\) which gives the
half-space description
for the full numerically realizable set~\eqref{eq:convexhullrealizableset}.
In that case, no rescaling is necessary and we
can easily ensure a minimum distance of \(\limitereps\) to the realizable boundary by
moving each facet in normal direction before calculating the intersections,
resulting in
\begin{equation}\label{eq:halfspacelimitereps}
  \tilde{\RQonehalfspaceveccomp}_\facetindex  = \RQonehalfspaceveccomp_\facetindex - \limitereps
  \norm{\RQonehalfspacematcol_\facetindex}{2}
\end{equation}
instead of \(\RQonehalfspaceveccomp_\facetindex\) in~\eqref{eq:halfspacelimiter}. As for the other
limiters, we
disable reconstruction if \(\momentscellmean{}\) does not lie within the
\(\limitereps\)-realizable set, i.e.\ if
\begin{align}
  \exists\, \facetindex \text{ s.t. } \RQonehalfspacematcol_\facetindex \cdot \momentscellmean{} \geq
  \tilde{\RQonehalfspaceveccomp}_\facetindex.
\end{align}
However, explicit calculation of the convex hull is only viable for a relatively small number of
moments (such that the convex hull has to be calculated in a low-dimensional space) or very sparse
quadratures (such that the convex hull has to be calculated from a small number of points).
For a larger number of moments and a reasonable fine quadrature, the construction of the convex
hull
takes excessively long. Moreover, even when the convex hull is available, the performance of this
approach might be unacceptable as the number of facets grows rapidly with both the number of
moments
and the number of quadrature points~\cite{Schneider2016}. We thus use this approach
only for the partial moments (see \cref{sec:limitingpmmn}) where we only have
to calculate low-dimensional convex hulls.

For the \(\MN\) models, as proposed in \cite[Section~3.62]{Schneider2016}, we instead
utilize the quadrature
description \eqref{eq:numericallyrealizableset2} of the numerically realizable set and limit by
solving the linear program (LP)
\begin{subequations}
  \label{eq:convexhulllimiter}
  \begin{align}
    \min \limitervariable                 \\
    \text{s.t.} \sum\limits_{\angularQuadratureIndex=0}^{\angularQuadratureNumber}
    \tilde{\angularQuadratureWeights{\angularQuadratureIndex}}
    \basis\left(\SC_\angularQuadratureIndex\right) = (1-\limitervariable) \moments
    + \limitervariable \momentscellmean{} \\
    \label{eq:linearprogramlimiter3}
    \limitervariable \geq 0,\,
    \tilde{\angularQuadratureWeights{\angularQuadratureIndex}} >
    0,
  \end{align}
\end{subequations}
where \(\tilde{\angularQuadratureWeights{\angularQuadratureIndex}}\) should not be confused with
\(\angularQuadratureWeights{\angularQuadratureIndex}\)
but rather represents
\(\angularQuadratureWeights{\angularQuadratureIndex}\distribution(\SC_\angularQuadratureIndex)\)
for the sought representing distribution \(\distribution\).
This approach removes the prohibitively costly explicit calculation of the convex hull. However,
the runtime
cost during the time stepping algorithm might be considerably higher as a linear program has to be
solved for every reconstructed value.

Instead of using a single limiter variable $\limitervariable$,
principally, we can limit each component of $\moments$ independently. This has been
done, e.g., in
the context of the Euler equations in \cite{Zhang2010}. However, if the limiting is naively
performed in ordinary coordinates, spurious oscillations may occur, as the limiting in ordinary
coordinates may actually increase the slope in one of the characteristic components. In our
implementation, we thus limit each of the characteristic components independently. Let
$\eigenvectormatrix$ be the matrix of eigenvectors of the Jacobian $\Flux'(\momentscellmean{})$ and
let $\momentschar = \eigenvectormatrix^{-1} \moments$, be the respective moment vectors in
characteristic coordinates. Then we can find limiter variables $\boldsymbol{\limitervariable} =
  (\limitervariable_{1}, \ldots, \limitervariable_{\momentnumber})$ for each
characteristic component
by solving the LP
\begin{align}
  \label{eq:LPlimitercompwise}
  \begin{aligned}
    \min \mathbf{1} \cdot \boldsymbol{\limitervariable} \\
    \text{s.t.}
    \sum\limits_{\angularQuadratureIndex=0}^{\angularQuadratureNumber}
    \tilde{\angularQuadratureWeights{\angularQuadratureIndex}}
    \basis\left(\SC_\angularQuadratureIndex\right) =
    \eigenvectormatrix
    \begin{pmatrix}
      (1-\limitervariable_1) \momentscharcomp[1] + \limitervariable_1
      \momentscharmeancomp{1} \\
      \vdots                  \\
      (1-\limitervariable_{\momentnumber}) \momentscharcomp[\momentnumber] +
      \limitervariable_{\momentnumber}
      \momentscharmeancomp{\momentnumber}
    \end{pmatrix}                          \\
    \boldsymbol{\limitervariable} \geq \boldsymbol{0},\,
    \tilde{\angularQuadratureWeights{\angularQuadratureIndex}} >
    0
  \end{aligned}
  \qquad\Longleftrightarrow\qquad
  \begin{aligned}
    \min \mathbf{1} \cdot \boldsymbol{\limitervariable}                     \\
    \text{s.t.}
    \begin{pmatrix}
      \mathbf{B} & \tilde{\eigenvectormatrix}
    \end{pmatrix}
    \begin{pmatrix}
      \boldsymbol{\angularQuadratureWeights{}} \\
      \boldsymbol{\limitervariable}
    \end{pmatrix} = \eigenvectormatrix \momentschar = \moments \\
    \begin{pmatrix}
      \boldsymbol{\angularQuadratureWeights{}} \\
      \boldsymbol{\limitervariable}
    \end{pmatrix} \geq \mathbf{0}
  \end{aligned}
\end{align}
where the matrix $\tilde{\eigenvectormatrix}$ is defined as $\tilde{V}_{ij} =
  V_{ij} \left(\momentscharcomp[j] - \momentscharmeancomp{j}\right)$ and the
$i$-th column of
$\mathbf{B}$ is $\basis\left(\SC_\angularQuadratureIndex\right)$.

For the LP-based limiter, it is not clear how to ensure a fixed distance
$\limitereps$ to the boundary. We thus use the method of adding a small parameter
$\limiterepstilde$ to the final limiter variable by replacing $\limitervariable
  \geq 0$ by
$\limitervariable \geq -\limiterepstilde$ in
\eqref{eq:linearprogramlimiter3} and using $\limitervariable+ \limiterepstilde$
instead of $\limitervariable$ if it is in the interval $[-\limiterepstilde, 1 -
      \limiterepstilde]$.

For checking realizability of $\moments$, which is needed for the stopping
criterion \eqref{eq:thirdstoppingcriterion} in the optimization algorithm, we
solve the simpler LP
\begin{subequations}
  \begin{align}
    \label{eq:LPrealizabilitycheck}
    \min 0                                                         \\
    \text{s.t.} \
    \mathbf{B} \boldsymbol{\angularQuadratureWeights{}} = \moments \\
    \boldsymbol{\angularQuadratureWeights{}} \geq \mathbf{0}.
  \end{align}
\end{subequations}
Note that the limiters presented thus far are not restricted to the $\Mn$ models
but could be used
for any basis $\basis$. However, for the $\HFMN$ and
$\PMMN$ models, due to the simpler
realizability conditions, limiters that are both faster and easier to implement can be used.

\subsubsection{\texorpdfstring{$\HFMN$ models}{HFMN models}}
For the hat functions the numerically realizable set and the realizable set agree for suitable
quadratures \cite{SchneiderLeibner2020}. As a consequence, we can use a limiter based on the
analytical
realizability conditions which only require component-wise positivity (see
\cref{sec:realizability}). We thus calculate the
limiter
variable $\limitervariable$ (limiting to $\RQeps{\hfbasis})$ as
\begin{align}\label{eq:positivitylimiter}
  \limitervariable_{\basisindex} & =
  \begin{cases}
    1 & \text{if } \momentcompcellmean{\basisindex} < \limitereps \\
    \frac{\limitereps-
      \momentcomp{\basisindex}}{\momentcompcellmean{\basisindex}-
      \momentcomp{\basisindex}}
      & \text{else if }
    \frac{\limitereps-
      \momentcomp{\basisindex}}{\momentcompcellmean{\basisindex}-
    \momentcomp{\basisindex}} \in [0,1]                           \\
    0 & \text{else}
  \end{cases},\qquad\qquad
  \limitervariable = \max_i \limitervariable_{\basisindex}.
\end{align}

\subsubsection{\texorpdfstring{$\PMMN$ models}{PMMN models}}\label{sec:limitingpmmn}

In one dimension, $\RQ{\pmbasis} = \RD{\pmbasis}{}$ for suitable quadratures, so
a limiter based on the analytical realizability conditions \eqref{eq:pmrealizabilityconds}
can be used. We use a limiter variable $\limitervariable_{\cellindex}$ per interval
$\cell{\cellindex} =
  [\SCheight_{\cellindex}, \SCheight_{\cellindex+1}]$. If we require a distance of at least $\limitereps$
to the boundary, the realizability conditions \eqref{eq:pmrealizabilityconds} become
\begin{equation}\label{eq:realizabilitylimitingpm1deps}
  \momentcomp{0,\cellindex} \geq \limitereps \quand
  \SCheight_{\cellindex}{\momentcomp{0,\cellindex}} + \limitereps
  \sqrt{\SCheight_{\cellindex}^2+1} \leq \momentcomp{1,\cellindex} \leq
  \SCheight_{\cellindex+1}{\momentcomp{0,\cellindex}} - \limitereps
  \sqrt{\SCheight_{\cellindex+1}^2+1}.
\end{equation}
If $\momentscellmean{\cell{\cellindex}}$ is already not $\limitereps$-realizable, we
disable reconstruction for that interval. This results in the following limiter for one-dimensional
partial moments
\begin{align}\label{eq:dglimiter1d}
  \limitervariable_{\cell{\cellindex}}
   & =
  \begin{cases}
    1       & \text{if } \momentscellmean{\cell{\cellindex}} \text{ does not fulfill }
    \eqref{eq:realizabilitylimitingpm1deps}                                            \\
    \max\left(\limitervariable_{\cell{\cellindex}}^0,
    \limitervariable_{\cell{\cellindex}}^1,
    \limitervariable_{\cell{\cellindex}}^2
    \right) & \text{else}
  \end{cases}
\end{align}
where
\begin{align*}
  \limitervariable_{\cell{\cellindex}}^0
   & =
  \begin{cases}
    \frac{\limitereps-\momentcomp{\cellindex,0}}{
      \momentcompcellmean{\cellindex,0}-\momentcomp{\cellindex,0}}
      & \text{if }
    \frac{\limitereps-\momentcomp{\cellindex,0}}{
    \momentcompcellmean{\cellindex,0}-\momentcomp{\cellindex,0}} \in [0,1] \\
    0 & \text{else}
  \end{cases} \\
  \limitervariable_{\cell{\cellindex}}^1
   & =
  \begin{cases}
    \frac{\momentcomp{\cellindex,0}
      \SCheight_{\cellindex} - \momentcomp{\cellindex,1} + \limitereps
      \sqrt{\SCheight_{\cellindex}^2+1}}{
      (\momentcompcellmean{\cellindex,1}-\momentcomp{\cellindex,1})
      -(\momentcompcellmean{\cellindex,0}-\momentcomp{\cellindex,0})
      \SCheight_{\cellindex}}
      & \text{if }
    \frac{\momentcomp{\cellindex,0}
      \SCheight_{\cellindex} - \momentcomp{\cellindex,1} + \limitereps
      \sqrt{\SCheight_{\cellindex}^2+1}}{
      (\momentcompcellmean{\cellindex,1}-\momentcomp{\cellindex,1})
      -(\momentcompcellmean{\cellindex,0}-\momentcomp{\cellindex,0})
    \SCheight_{\cellindex}} \in [0,1] \\
    0 & \text{else}
  \end{cases} \\
  \limitervariable_{\cell{\cellindex}}^2
   & =
  \begin{cases}
    \frac{\momentcomp{\cellindex,0}
      \SCheight_{\cellindex+1} - \momentcomp{\cellindex,1} - \limitereps
      \sqrt{\SCheight_{\cellindex+1}^2+1}}{
      (\momentcompcellmean{\cellindex,1}-\momentcomp{\cellindex,1})
      -(\momentcompcellmean{\cellindex,0}-\momentcomp{\cellindex,0})
      \SCheight_{\cellindex+1}}
      & \text{if }
    \frac{\momentcomp{\cellindex,0}
      \SCheight_{\cellindex+1} - \momentcomp{\cellindex,1} - \limitereps
      \sqrt{\SCheight_{\cellindex+1}^2+1}}{
      (\momentcompcellmean{\cellindex,1}-\momentcomp{\cellindex,1})
      -(\momentcompcellmean{\cellindex,0}-\momentcomp{\cellindex,0})
    \SCheight_{\cellindex+1}} \in [0,1] \\
    0 & \text{else}
  \end{cases}
\end{align*}
For the partial moment basis in three dimensions, the analytical and numerical
realizable set differ. However, note that \eqref{eq:convexhullrealizablesetless1} holds separately
for each spherical triangle (see \cite[Lemma 5.13]{SchneiderLeibner2020}), so we can explicitly
calculate the half space representation \eqref{eq:halfspacerealizableset} for each
spherical triangle.
Instead of calculating a convex hull in $\momentnumber$ dimensions, as would be needed
for the full moment models, we only
have to calculate
$\frac{\momentnumber}{4}$ convex hulls in $4$ dimensions, which is considerably
faster and usually finished within a few seconds in our implementation (remember that this
calculation has to be done only once before the time stepping).

\subsection{Implementation of quadrature rules}
\label{sec:QuadratureRules}
In one dimension, we use Gauss-Lobatto quadratures on each interval. These
quadratures include the endpoints of the interval, which ensures that the
numerically realizable set (see \eqref{eq:numericallyrealizableset2}) equals the
analytically realizable set for hat functions and partial moments, see
\cite{SchneiderLeibner2020}.
To choose a suitable quadrature order,
we solved some of our numerical test
cases for different quadrature orders and calculated the errors with respect to
the reference solution (see Section S2 in the supplementary materials).
As suggested by this analysis,
for the first-order models, we use a quadrature of order 15
per interval of the partition \(\generalpartition\).
For the full moment \(\MN\) models, we split
the domain in the two intervals \([-1,0]\) and
\([0,1]\) that
are needed for calculation of the kinetic flux and
use a quadrature of order \(2\momentorder+40\) on each interval.

In three dimensions, for partial moments and hatfunctions, we are using Fekete
quadratures \cite{Taylor2000} (from the \texttt{TRIANGLE\_FEKETE\_RULE} library
\cite{jburkardtfekete}) mapped to the spherical triangles. The library provides
seven Fekete quadratures of order \(3\), \(6\), \(9\), \(12\), \(12\), \(15\) and \(18\),
using \(10\), \(28\), \(55\),
\(91\), \(91\), \(136\), \(190\) quadrature points, respectively. The second rule of
order \(12\) contains some negative quadrature weights, so we do not use
that quadrature. If we want to improve the approximation, we subdivide each
spherical triangle in several smaller ones as in \cite{boal2004adaptive} and use
the mapped Fekete quadrature on each subtriangle. The Fekete rules correspond to
Gauss-Lobatto rules on the triangle edges and thus also include the vertices of
each triangle \cite{Taylor2000}, which simplifies the realizability
preservation (see \cite{SchneiderLeibner2020}).
As suggested by our quadrature sensitivity analysis (see Section S2),
we use a quadrature order of \(15\) for the \(\HFMN[6]\) and
\(\PMMN[32]\)
models and a quadrature order of \(9\) for the other \(\HFMN\) and \(\PMMN\)
models. For the \(\MN\) models, we use tensor-product
quadrature rules of order \(2\momentorder+8\) on the octants of the sphere.

For the hat function basis in one dimension, we alternatively explicitly calculate all
integrals needed in the Newton algorithm using the analytical formulas
and Taylor expansion at the numerical singularities of the analytical formulas
(see \cite[Appendix A1]{Leibner2021} for the explicit formulas).
The
Taylor expansion is performed in a neighborhood of radius $0.1$ around the
singularity, up to vanishing remainder or a maximal order of \(200\). This
completely removes the need for quadrature rules. Note that the same approach
could be used for the partial moments in one dimension. However, as the
quadrature-based adaptive-change-of-basis algorithm is very efficient for
partial moments, we did not implement the analytical formulas for partial
moments.

For the hat function basis in three dimensions, integrals cannot be evaluated analytically anymore.
We experimented with an approach where the integrals are expanded in a Taylor series representation
(see Section S1 in the supplementary materials). However, it turned out that the Taylor series had
to
be computed up to a prohibitively high order. For this reason, we dismissed that approach
and use the Fekete quadrature approach described above also for the
hat functions in three dimensions.

\subsection{Implementation of initial and boundary conditions}\label{sec:initialandboundaryfvscheme}
The initial values for the finite volume scheme
are computed by integration of the kinetic equation's initial values \eqref{eq:kineticequationinitial}.
\begin{equation*}
  \moments[\gridindex]^{0} =
  \frac{1}{\gridwidthz}\int\limits_{\z_{\cellindex-\frac12}}^{\z_{\cellindex+\frac12}}
  \ints{\distributiontzero(\spatialvar,\SC) \basis} \mathrm{d}\spatialvar
\end{equation*}
Since the initial values in our test cases are isotropic (see \cref{sec:results}), i.e.\
\(\distributiontzero(\spatialvar,\SC) = \distributiontzero(\spatialvar)\), we only have
to compute the velocity integral of the basis \(\ints{\basis}\). For this integral, we
use the
same quadratures as in \cref{sec:QuadratureRules} to ensure that the result is numerically
realizable.
Except for the plane-source and point-source tests, the initial values are constant in each grid
cell,
so we use the midpoint quadrature to evaluate the spatial integral.
For the plane-source test, we always use an even number of grid cells and distribute
the Dirac delta at \(\spatialvar=0\) into the two adjacent grid cells, i.e.\ the initial
value in these grid cells is set
to the constant \(\distributiontzero(\spatialvar) = \distributionvacuum + \frac{1}{2\gridwidth}\).
For the point-source test, we use a Gauss-Legendre tensor product quadrature of order 20 to
evaluate the spatial integrals for the initial values.

Boundary conditions for the moment equations are implemented
by replacing the ansatz function
\(\ansatz[{\moments[\gridindexalt]}]\) belonging to a grid cell \({\gridindexalt}\)
outside of the computational domain
(such cells often called ``ghost cells'') by the boundary
condition \(\distributionboundary\) of the kinetic equation \eqref{eq:kineticequationboundary}
in the computation of the kinetic flux \(\eqref{eq:kineticFlux}\).

\section{Numerical results}\label{sec:results}
We want to apply our moment models to several test cases in the one- and
three-dimensional setting.
We follow the FAIR guiding principles for scientific research \cite{Wilkinson2016} and publish
the
code that generates the following results in \cite{SchneiderLeibner2019DataversePart2}.
As already mentioned (see \cref{section:ImplementationDetails}), the scheme was implemented
in the \texttt{DUNE} generic discretization toolbox
\texttt{dune-gdt}~\cite{Schindler2017}.
The computations were done on a varying number of nodes of a distributed
memory computer cluster\footnote{Each node encloses two Intel
  Intel Skylake Xeon Gold 6140 CPUs (\(2 \times 18\) cores) and \(92\)GB RAM.}.
Communication between nodes was done via MPI (Message Passing Interface)~\cite{MPIF2015}.
On each node, we used a work-stealing task-based shared-memory parallelization
which was implemented using Intel \texttt{TBB}~\cite{Intel2020}.

\subsection{Slab geometry (1D)}
\subsubsection{Plane source}
\label{sec:Planesource}
In this test case an isotropic distribution with all mass concentrated in the middle of an infinite
domain $\z \in
  (-\infty, \infty)$ is defined as initial condition, i.e.
\begin{equation*}
  \distributiontzero(\z, \SCheight) = \distributionvacuum + \delta(\z),
\end{equation*}
where the small parameter $\distributionvacuum = 0.5 \cdot 10^{-8}$ is used to
approximate a vacuum.
In practice, a bounded domain must be used which is large
enough that the boundary should have only negligible effects on the
solution. For the final time $\tf = 1$, the domain is set to $\domain = [-1.2, 1.2]$
(recall that for all presented models the maximal speed of propagation is bounded in absolute value
by one).

At the boundary the vacuum approximation
\begin{equation*}
  \distributionboundary(\timevar,\zL,\SCheight) \equiv \distributionvacuum \quand
  \distributionboundary(\timevar,\zR,\SCheight) \equiv \distributionvacuum
\end{equation*}
is used again. Furthermore, the physical coefficients are set to
$\scattering \equiv 1$, $\absorption \equiv 0$ and $\source \equiv 0$.

All solutions are computed with an even number of cells, so the initial Dirac
delta lies on a cell boundary. Therefore it is approximated by splitting it into
the cells immediately to the left and right. In all figures below, only positive
$\z$ are shown since the solutions are always symmetric around
$\z = 0$.

Note that since the method of moments is indeed a type of spectral method, it
can be expected that due to the non-smoothness of the initial condition the
convergence towards the kinetic solution of this test case is slow (note that
$\distributiontzero(\cdot, \SCheight)\notin\Lp{p}$ for any $p$).
Nevertheless, it is an often-used benchmark revealing many properties of a
moment model (see e.g.\ \cite{Garrett2014}).

Some exemplary solutions at the final time are shown in
\cref{fig:Planesource1,fig:Planesource2,fig:Planesource3}.
Remember that the full-moment models are indexed by the basis order \(\momentorder\)
while the piecewise linear bases are indexed by
the number of moments \(\momentnumber\).
Since, in one dimension, a basis of the space of polynomials
up to order \(\momentorder\) has \(\momentnumber = \momentorder + 1\)
elements, we compare the \(\PN\) and \(\MN\) models to the piecewise linear models
with \(\momentorder + 1\) moments.

As expected, there are strong oscillations
about the reference solution (the analytical solution from~\cite{Ganapol2001}) for all tested
models.
With increasing number of moments,
the number of peaks increases while their height decreases.
The oscillations are considerably stronger for the linear models
than for the corresponding minimum-entropy models.
The \(\MN\) models are closest to the reference solution, particularly for the low-order models.
\begin{figure}
  \centering
  \externaltikz{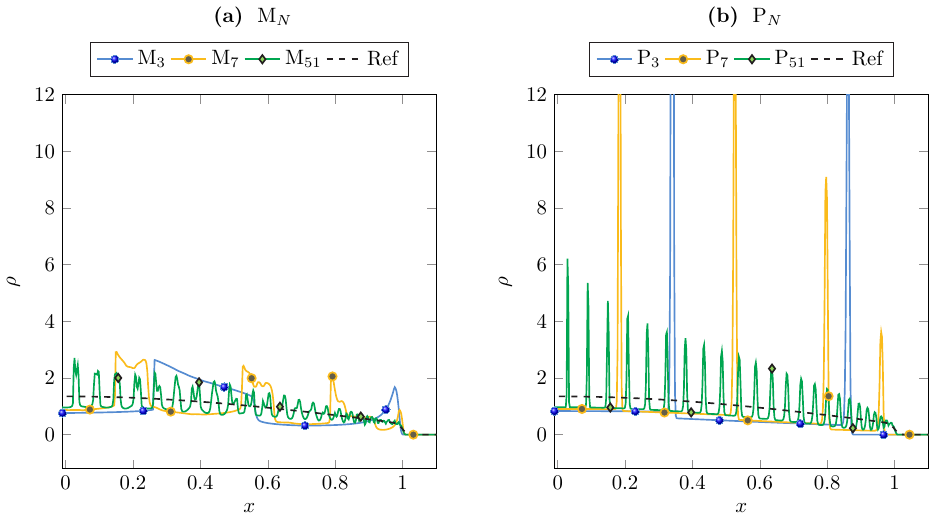}{\input{Images/PlanesourceIsotropicCutsFull}}
  \caption[Local particle density in the plane-source test for full-moment models]{Local particle
    density \(\density\) in the plane-source test case at time \(\tf = 1\) for different orders of the
    full-moment models.}\label{fig:Planesource1}
\end{figure}
\begin{figure}
  \centering
  \externaltikz{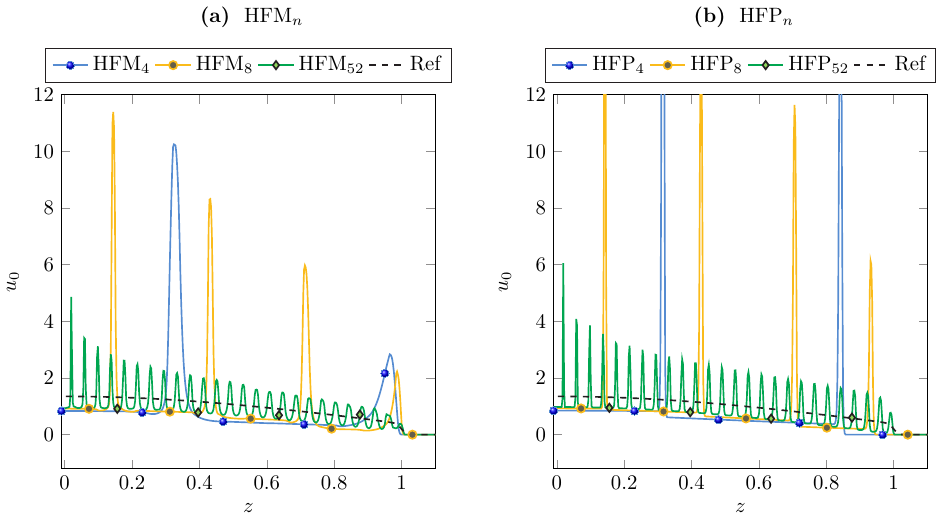}{\input{Images/PlanesourceIsotropicCutsHF}}
  \caption[Local particle density in the plane-source test for hat function moment models]{Local particle
    density \(\density\) in the plane-source test case at time \(\tf = 1\) for different orders of the
    hat function moment models.}\label{fig:Planesource2}
\end{figure}
\begin{figure}
  \centering
  \externaltikz{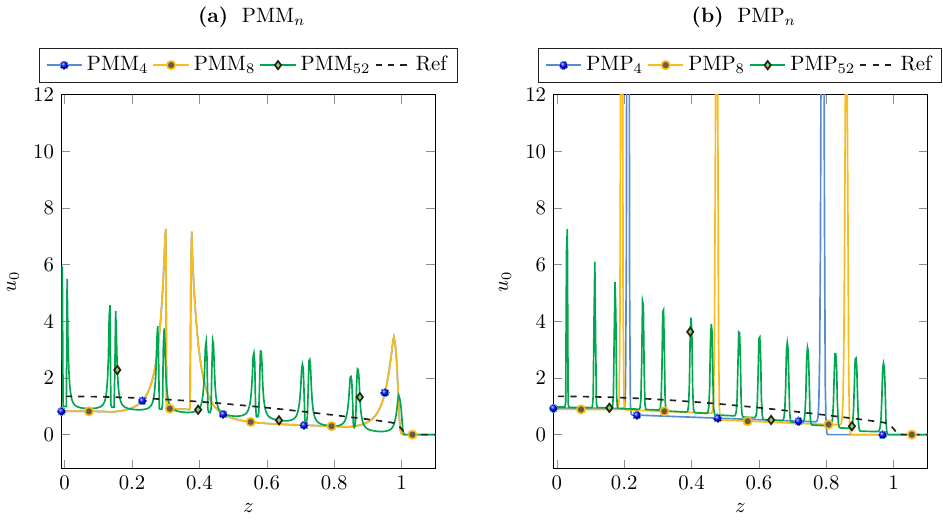}{\input{Images/PlanesourceIsotropicCutsPM}}
  \caption[Local particle density in the plane-source test for partial-moment models]{Local particle
    density \(\density\) in the plane-source test case at time \(\tf = 1\)
    for different orders of the partial-moment models.}\label{fig:Planesource3}
\end{figure}
\begin{figure}
  \centering
  \externaltikz{PlanesourceError}{\input{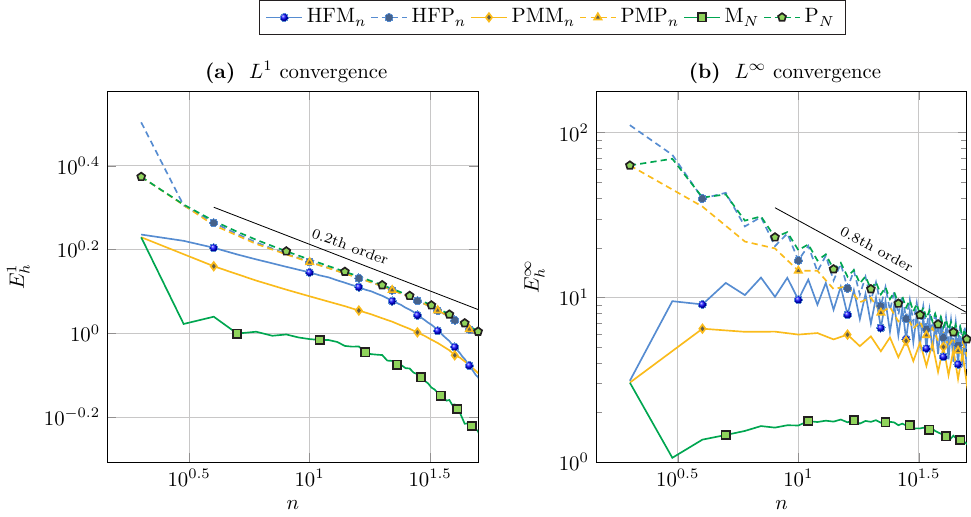}}
  \caption{Convergence of the local particle density $\density$ in the
    plane-source test case for different models.}
  \label{fig:PlanesourceError}
\end{figure}

This is also reflected in the convergence results, which can be found in
\cref{fig:PlanesourceError}. Depicted are the \(\Lp{1}\) and
\(\Lp{\infty}\) errors
between the local particle densities \(\density(\moments)\) of the moment
models and the analytic reference solution from~\cite{Ganapol2001} at the final time \(\tf\).
As expected, overall convergence is slow. The \(\HFPN\), \(\PMPN\) and \(\PN\) models
show very similar \(\Lp{1}\) errors at all orders.
With respect to \(\Lp{\infty}\) norm, the \(\PMPN\) models are slightly better than the
\(\PN\) models. As observed before for the \(\PN\) models~\cite{Schneider2016b,Schneider2016},
\(\HFPN\) models with odd \(n\) show a higher
\(\Lp{\infty}\) error than models with even \(n\) due to a zero eigenvalue of the flux
jacobian. A similar but much less pronounced behaviour can be seen for the
\(\PMPN\) models (for odd and even number of intervals \(\frac{\momentnumber}{2}\)). For odd
\(\momentnumber\),
\(\HFPN\) \(\Lp{\infty}\) errors are close to the corresponding \(\PN\) error. For even
\(\momentnumber\)
the \(\HFPN\) models perform better than the \(\PN\) models and similar to the
\(\PMPN\) model.

The entropy-based \(\HFMN\), \(\PMMN\) and \(\Mn\) models show lower errors than
their linear counterparts both in \(\Lp{1}\) and \(\Lp{\infty}\) norm.
The \(\PMMN\) models
perform slightly better than \(\HFMN\) models of the same order. The difference
between odd and even orders/number of intervals is much more pronounced than for
the \(\HFPN\) and \(\PMPN\) models. The \(\Mn\) models give the lowest errors of all
tested models. However, the errors are still high and the rate of convergence
is equally bad for all models although the convergence rate seems to
improve with higher orders, especially for the minimum-entropy-based models.

\subsubsection{Source beam}
\label{sec:SourceBeam}
The discontinuous version of the source-beam problem from
\cite{Hauck2013} is presented.
The spatial domain is $\domain = [0,3]$, and
\begin{gather*}
  \absorption(\z) = \begin{cases}
    1 & \text{ if } \z\leq 2, \\
    0 & \text{ else},
  \end{cases} \quad
  \scattering(\z) = \begin{cases}
    0  & \text{ if } \z\leq 1,   \\
    2  & \text{ if } 1<\z\leq 2, \\
    10 & \text{ else}
  \end{cases} \quad
  \source(\z) = \begin{cases}
    \frac12 & \text{ if } 1\leq \z\leq 1.5, \\
    0       & \text{ else},
  \end{cases}
\end{gather*}
with initial and boundary conditions
\begin{gather*}
  \distributiontzero(\z, \SCheight) \equiv \distributionvacuum, \\
  \distributionboundary(\timevar,\zL,\SCheight) = \cfrac{e^{-10^5(\SCheight-1)^2}}{\ints{e^{-10^5(\SCheight-1)^2}}}
  \quand
  \distributionboundary(\timevar,\zR,\SCheight) \equiv \distributionvacuum.
\end{gather*}
The final time is $\tf = 2.5$ and the same vacuum approximation $\distributionvacuum$
as in the plane-source problem is used.

\begin{figure}[htbp]
  \centering
  \externaltikz{SourceBeamError}{\input{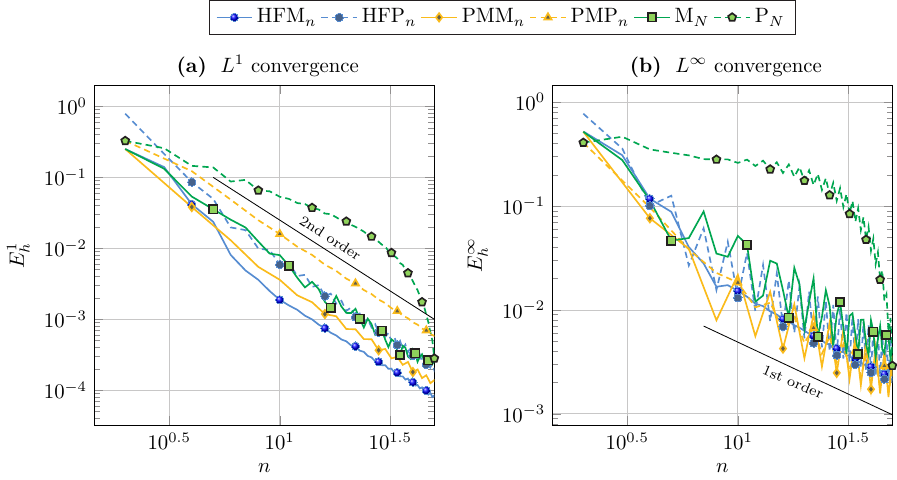}}
  \caption{Convergence of the local particle density $\density$ in the source beam test case for
    different models.}
  \label{fig:SourceBeamError}
\end{figure}

Convergence results can be found in \cref{fig:SourceBeamError}.
The reference solution for this test case is computed from a direct finite difference
discretization of the kinetic
equation on a grid with \(21000\times14000\) elements.
As expected due to the higher regularity of the test case\footnote{\(\distribution\) is ``only'' a discontinuous function
  compared the distributional setting in the plane-source test.}, the convergence
for all tested
models is much better than in the plane-source test.
As a consequence, most of the models are relatively
close to the reference solution which is why we do not show exemplary solutions plots.
The
piecewise linear models and the \(\MN\) models converge with second and first
order in \(\Lp{1}\) and \(\Lp{\infty}\) norm, respectively.
The \(\PN\) models show spectral convergence which
is reflected in comparatively high errors and a slow rate of convergence for the lower-order models
and then a rapid convergence for the high-order models.
Despite the eventual high rate of convergence, for the maximal moment number
\(\momentnumber=50\) regarded here the \(\PN\) models are
mostly outperformed by the other models.
Again, the minimum-entropy-based models perform better than their linear
counterparts although the rate of convergence is the same.
In \(\Lp{\infty}\) norm, the \(\MN\), partial-moment and \(\HFPN\)
models again show a zig-zag pattern where, e.g.,
the \(\HFPN\) models using an odd number
of intervals (even number of moments \(\momentnumber\)) perform better
than the ones using an even interval number (odd \(\momentnumber\)).
For the partial-moment models, an even number of intervals gives better results.
Notably, the \(\HFMN\) do not show such an alternating behaviour. \enlargethispage{1\baselineskip}

\subsection{Three dimensions}
We now consider numerical results in three spatial dimensions with velocities on the unit sphere.
\subsubsection{Point source}
The point-source test is the three-dimensional analogue of the plane-source test
(\secref{sec:Planesource}) in slab geometry. Due to the limitations in the
resolution we use a smoothed version of the initial Dirac delta:
\begin{equation*}
  \distributiontzero(\spatialvar, \SC) = \distributionvacuum +
  \frac{1}{4\pi^4\sigma^3}\exp\left(-\frac{\abs{\spatialvar}^2}{\pi\sigma^2}
  \right),
\end{equation*}
where $\sigma = 0.03$, $\distributionvacuum = \frac{10^{-8}}{4\pi}$. As before, we choose
$\scattering\equiv 1$, $\absorption\equiv 0$ and $\source \equiv 0$.
All models are calculated on $\domain = [-1,1]^3$ to the final time $\tf = 0.75$. The
grid size is
chosen to be $\gridwidthx=\gridwidthy=\gridwidthz = 0.02.$
The point-source test is well-suited to demonstrate symmetries (or symmetry breaks) appearing in
the
solution. We show some selected models in Figures \ref{fig:Pointsource1} and
\ref{fig:Pointsource2},
where we use the endcap geometry $[0,1]\times[-1,1]\times[-1,1]$ for the isosurfaces.
\begin{figure}
  \centering
  \externaltikz{Pointsource1}{\input{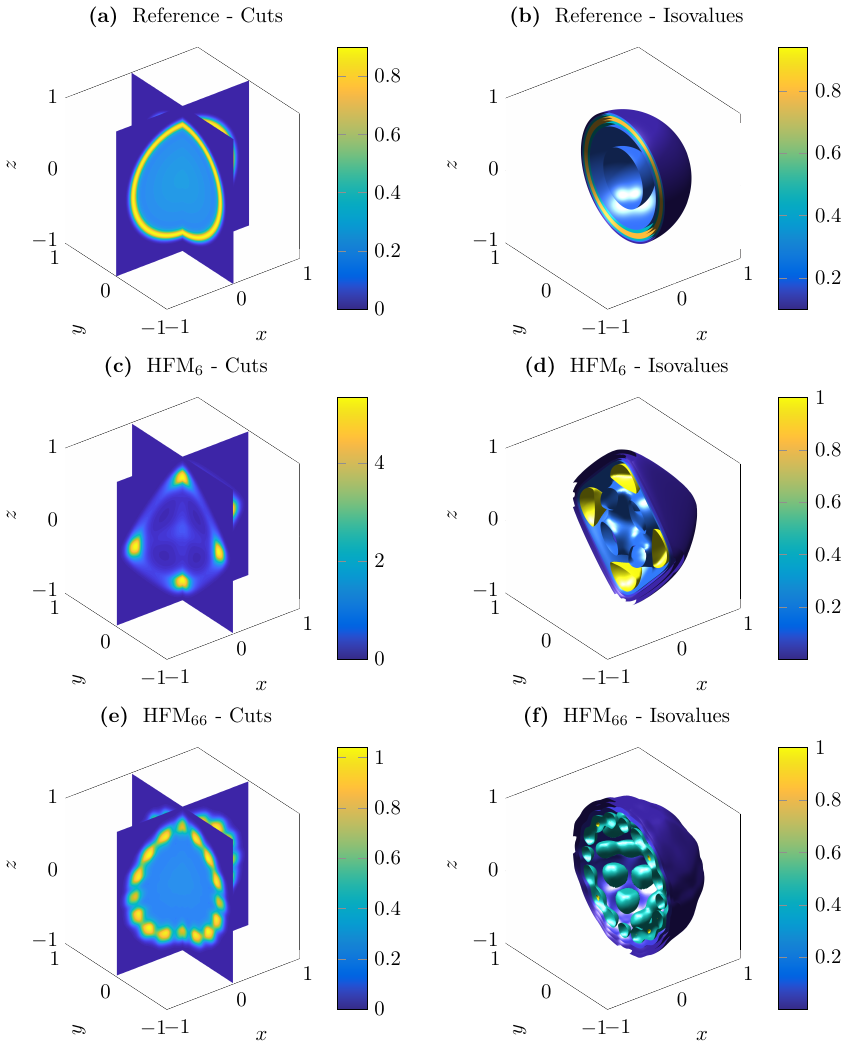}}
  \caption{Two-dimensional cuts and selected isosurfaces for some models in the point-source test.}
  \label{fig:Pointsource1}
\end{figure}
\begin{figure}
  \centering
  \externaltikz{Pointsource2}{\input{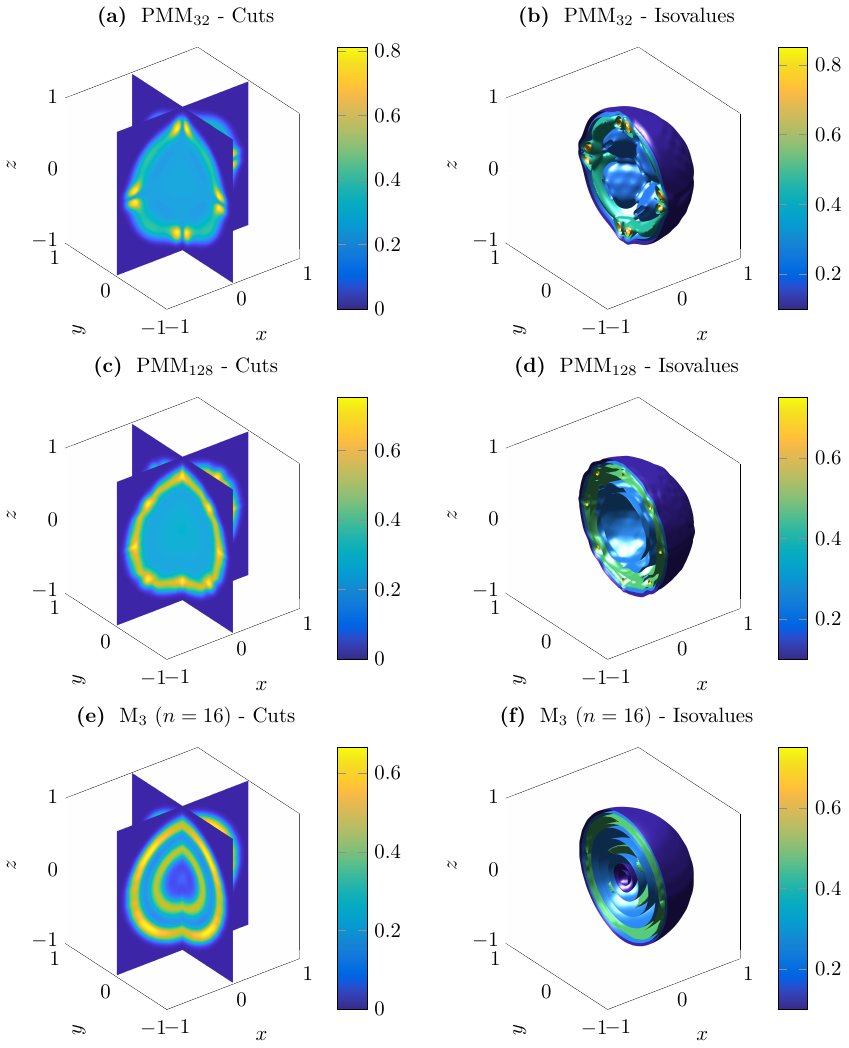}}
  \caption{Two-dimensional cuts and selected isosurfaces for some models in the point-source test.}
  \label{fig:Pointsource2}
\end{figure}

The reference solution itself is rotationally symmetric and can be computed analytically using the
formulas by Ganapol \cite{Ganapol2001}. It can be observed that the hat functions have a
preferred directions of propagation, directly related to the position of the vertices in the
spherical triangulation (e.g., the octahedron that defines the $\HFMN[6]$ basis can be
easily identified in \figref{fig:Pointsource1}). Similar effects occur for the partial moments.
However, the discontinuity of their basis is also reflected in the peaks along the boundaries of
the spherical triangles (compare $\PMMN[32]$). In contrast to this, the full-moment
models preserve the rotational symmetry (compare $\Mn[3]$, where small irregularities
in the solution arise due to the spherical quadrature rule) but adding more waves to the solution.

Finally, we show error plots for our models in \figref{fig:PointsourceError}. The models show the
expected slow convergence in the $\Lp{1}$-norm, similar to the plane-source test
(\secref{sec:Planesource}). All first-order models show roughly order $\frac12$,
whereas the full-moment models have varying convergence rates. In the $\Lp{\infty}$-norm,
the first-order models show order $1$ convergence in the beginning, which then
slows down to order $\frac12$ as well. The full-moment models are showing no (or very
slow) convergence, which is the well-known Gibbs phenomenon.

Note that the \(\PMMN\) models clearly outperform the other methods. In particular,
they are as good as or even slightly better
than the \(\Mn\) models whose calculation is significantly more
expensive for the same degrees of freedom (see \cref{sec:Timings}).

\begin{figure}
  \centering
  \externaltikz{PointsourceError}{\input{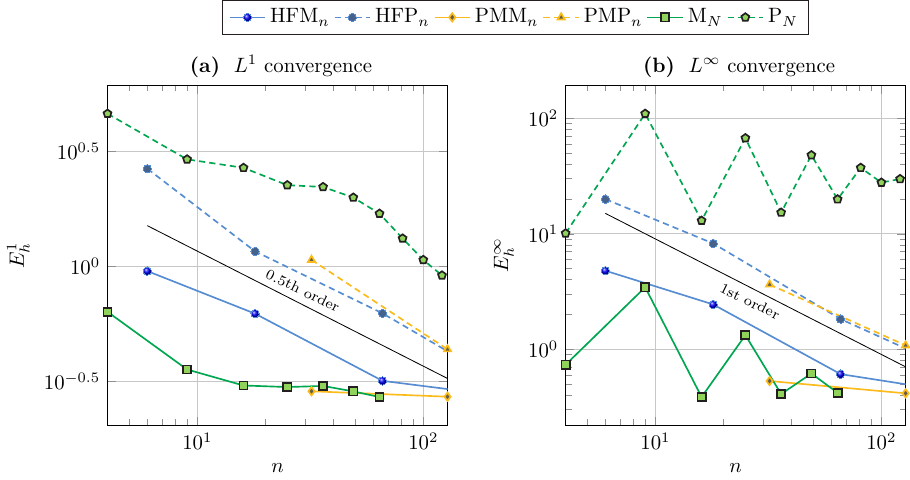}}
  \caption{Convergence of the local particle density $\density$ in the
    point-source test for different models.}
  \label{fig:PointsourceError}
\end{figure}

\subsubsection{Checkerboard}
The checkerboard test case is a lattice problem which is loosely based on a part of a reactor core
\cite{Brunner2005b}. We extend it in a straightforward manner to the three-dimensional case.
The used geometry is shown in \figref{fig:CheckerboardGeometry}. There are scattering (orange and green) and
highly absorbing (black) regions. The parameters are chosen to be the following.
\begin{itemize}
  \item Domain: $\domain = [0,7]^3$, subdivided into the three regimes
        \begin{align*}
          \domain_a         & = \Set*{\spatialvar = (\x,\y,\z)^T \in[1,6]^3 \given \begin{aligned}        & (\floor{\x}+\floor{\y}+\floor{\z})\bmod 2=1,              \\
                       & \spatialvar\notin [3,4]^3\cup [3,4]\times[5,6]\times[3,4]\end{aligned}} , \\
          \domain_{s}       & = \domain\setminus\domain_a,                                                       \\
          \domain_{\source} & = [3,4]^3,
        \end{align*}
  \item Final time: $\tf = 3.2$,
  \item Parameters (compare \figref{fig:CheckerboardGeometry}):
        \begin{align*}
          \scattering(\spatialvar) = \begin{cases}
            1 & \text{ if } \spatialvar\in\domain_s, \\
            0 & \text{ else},
          \end{cases},~
          \absorption(\spatialvar) = \begin{cases}
            0  & \text{ if } \spatialvar\in\domain_s, \\
            10 & \text{ else},
          \end{cases},~
          \source(\spatialvar) = \begin{cases}
            \frac{1}{4\pi} & \text{ if } \spatialvar\in \domain_{\source}, \\
            0              & \text{ else}.
          \end{cases}
        \end{align*}
  \item Initial condition: $\distributiontzero(\spatialvar,\SC) = \distributionvacuum := \cfrac{10^{-8}}{4\pi}$ (approx. vacuum),
  \item Boundary conditions: $\distributionboundary(\timevar,\spatialvar,\SC) = \distributionvacuum$.
\end{itemize}

\begin{figure}
  \centering
  \externaltikz{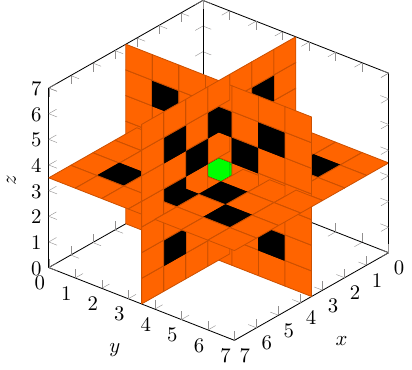}{
    \input{Images/CheckerboardGeometry.tex}
  }
  \caption{Geometry of the checkerboard test case. Orange and green spots are scattering, black spots are absorbing. The source is located in the green spot.}
  \label{fig:CheckerboardGeometry}
\end{figure}
Due to the discontinuous nature of the physical parameters, this test case is a
challenging task for a numerical solver. We align our grid with the
discontinuities of the parameters by using a multiple of $7$ (usually
$70$)
regularly spaced grid points in each direction.

Solution plots for selected models can be found in the supplementary materials (Figures
S3.1--S3.3).
The \(\PMPN[32]\) model and the \(\PN\) models
of order \(\momentorder\in\{2,\ldots,9\}\) (only \(\momentorder=9\) shown) have negative particle
densities \(\density\).
Surprisingly, the hat function basis \(\HFPN\) has positive densities for all \(\momentnumber\)
that we calculated.

We compare our models to a discrete ordinate implementation~\cite{larsen2010advances,Garrett2014} of second
order. \(\Lp{1}\)- and
\(\Lp{\infty}\)-errors of the local particle density \(\density\) can be found
in \cref{fig:CheckerboardError}. All tested models converge with about first order both in
\(\Lp{1}\) and \(\Lp{\infty}\) norm. Again, \(\PMMN\) and \(\HFMN\)
models are comparable to the \(\Mn\)
models.
\begin{figure}
  \centering
  \externaltikz{CheckerboardError}{\input{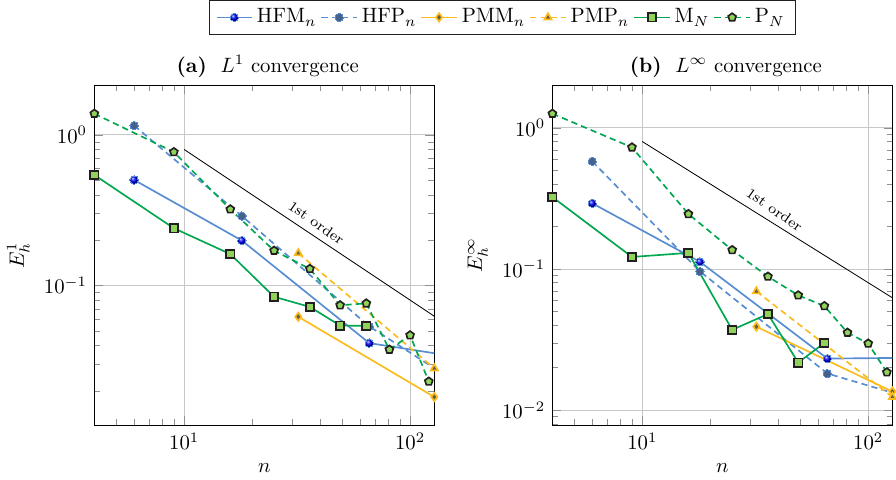}}
  \caption{Convergence of the local particle density $\density$ in the
    checkerboard test case for different models.}
  \label{fig:CheckerboardError}
\end{figure}

\subsubsection{Shadow}
The shadow test case represents a particle stream that is partially blocked by
an absorber, resulting in a shadowed region behind the absorber. The used
geometry is shown in \figref{fig:ShadowGeometry}.
\begin{figure}[htbp]
  \centering
  \externaltikz{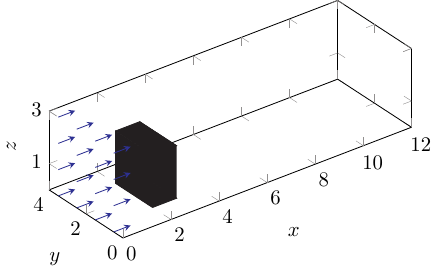}{
    \input{Images/ShadowGeometry.tex}
  }
  \caption{Setup of the shadow test. The absorbing region is depicted in
    black.}
  \label{fig:ShadowGeometry}
\end{figure}
The parameters are chosen to
be the following.
\begin{itemize}
  \item Domain: $\domain = [0,12]\times[0,4]\times[0,3]$
  \item Final time: $\tf = 20$,
  \item Parameters:
        \begin{align*}
          \scattering(\spatialvar) & = \source(\spatialvar) = 0   \\
          \absorption(\spatialvar) & = \begin{cases}
            50 & \text{ if } \spatialvar\in[2,3]\times[1,3]\times[0,2] \\
            0  & \text{ else},
          \end{cases}
        \end{align*}
  \item Initial condition: $\distributiontzero(\spatialvar,\SC) =
          \distributionvacuum := \cfrac{10^{-8}}{4\pi}$ (approx. vacuum),
  \item The isotropic particle stream with density $\density = 2$ enters the
        region via the boundary condition at $\x = 0$. At all other boundaries, vacuum
        boundary conditions are used.
        \begin{equation*}
          \distributionboundary(\timevar,\spatialvar,\SC) = \begin{cases}
            \cfrac{2}{4\pi}     & \text{ if } \x = 0 \\
            \distributionvacuum & \text{ else. }
          \end{cases}
        \end{equation*}
\end{itemize}
We show slices and isovalues of several models at the final time in the supplementary materials
(Figures S.4--S.6).
Again, several of the linear models (e.g.\ \(\PMPN[32]\)
or \(\PN[22]\)) show negative values (depicted in red). As in the previous test
case, the partial
moments perform very well. Compare, for example, the
linear \(\PMPN[512]\) model to \(\PN[22]\) (which has roughly the
same number of degrees of freedom). The partial-moment model approximates the reference much
better,
especially in the far field where also the small
oscillations are captured accurately.
A similar tendency can be observed for the hat function model \(\HFPN[258]\).

Both hat function and discontinuous minimum-entropy models show a good approximation of the
absorber (compare \(\PMMN[32]\) and
\(\HFMN[6]\)). However, they are not able to provide a reasonable approximation in
the far field. Further repartitioning of the
sphere yields much better results in this case.

Investigating again the convergence towards the reference solution (see \cref{fig:ShadowError}),
we see that the full-moment models are slightly superior in the beginning, but convergence slows
down for higher \(\momentnumber\). Both \(\HFMN\) as well as \(\PMMN\) show a similar convergence
behaviour. Again, taking running time into account, both models
outperform the classical \(\Mn\) model in
terms
of efficiency (see \cref{sec:Timings}). Note that we only
computed \(\MN\) models up to a order of \(\momentorder = 4\)
since the higher-order models did not finish in
the available computation time. The same is true for the \(\HFMN\)
models with \(\momentnumber > 66\).
Distributing the workload among more nodes of the distributed cluster
did not improve
computation times for these models, which is
probably due to load-balancing issues.
While the task-stealing algorithm (see above)
ensures that the work is distributed evenly on each node,
there is no load-balancing \emph{between} nodes.
In our implementation, the grid is distributed to the nodes using a
decomposition of the domain in connected parts.
Since the optimization problems are particularly challenging
near the absorber (due to low
particle densities and very anisotropic distributions),
the node(s) 
containing the absorber often
need much longer than the other nodes to solve the optimization problems.
Here, an MPI-based load-balancing implementation may be required
which would, however, significantly increase the communication overhead.
As an alternative, in \cite{Leibner2020},
we investigate a numerical scheme that avoids
the non-linear optimization problems and thus does
not show the same load-balancing issues.
\begin{figure}
  \centering
  \externaltikz{ShadowError}{\input{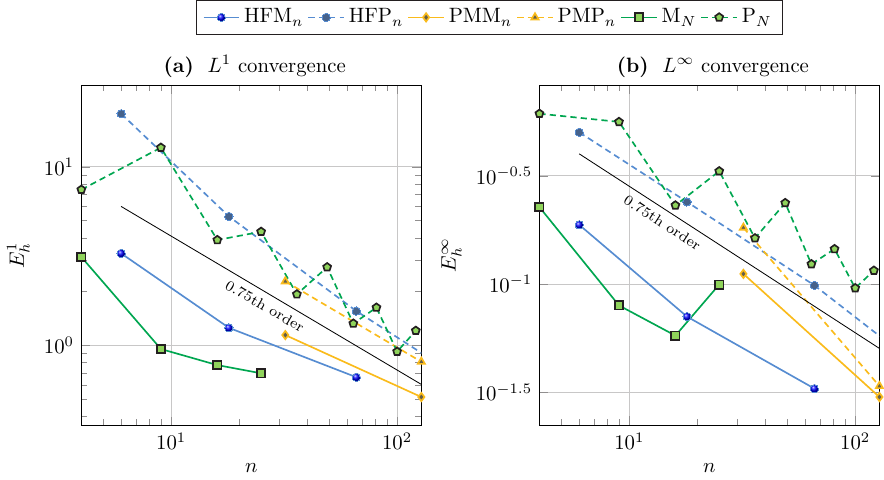}}
  \caption{Convergence of the local particle density $\density$ in the
    shadow test case for different models.}
  \label{fig:ShadowError}
\end{figure}

In conclusion, moment models based on piecewise first-order continuous (\(\HFPN\), \(\HFMN\))
or discontinuous (\(\PMPN\), \(\PMMN\)) basis functions often approximate the true solution as
good as or even better than the standard models (\(\PN\), \(\Mn\)) using polynomial bases.
In contrast to the standard models, however, these models can be implemented
very efficiently. This is especially true for the entropy-based models since
the necessary
realizability limiting can be based on the analytical realizability conditions (see
\cref{sec:QuadratureRules}). In addition, in case of the discontinuous models,
all matrix operations can be performed on small matrix blocks which provides further
performance advantages (also for the \(\PMPN\) models, see \cref{sec:Timings}).

\subsection{Timings}\label{sec:Timings}
Performance measurements can be found in \cref{fig:Timings2}.
The times were measured without parallelization. Displayed times are the minimum of three runs.
Quadratures were chosen as described in \cref{sec:QuadratureRules}. Measurements were done
both for the first-order scheme without linear reconstruction and
for the realizability-preserving second-order scheme (see \cref{sec:scheme}). Profiling
shows that the first-order scheme spends most of the time solving the optimization problems. For
the
second-order scheme, solving the eigen problems for the reconstruction in characteristic
coordinates
also has a large impact on the execution time.
Here, computation times could probably be improved by using a generalized eigen solver which
takes the structure of the Jacobians into account (see \cref{sec:eigenvalueproblems}).
Both the adaptive-change-of-basis
scheme and the eigensolver have third-order complexity. We thus asymptotically expect third-order
complexity in \(\momentnumber\) for both the first-order and the second-order scheme for the
\(\Mn\)
models. For the \(\PMMN\) models, all operations (including the solution of the eigen problems) can
be done
block-wise, so we expect first-order complexity in that case. Regarding only the
optimization problem, the same is true for \(\HFMN\) models as all matrices involved are
tridiagonal
(in slab geometry) or very sparse (in three dimensions). However, the Jacobian of
the flux function is not sparse in general for the \(\HFMN\) models,
so the eigen problems are currently
solved with a standard third-order-complex eigensolver.
These models would particularly benefit from an improved implementation
of the eigensolver which exploits
the fact that the Jacobians are products of two sparse symmetric matrices (see
\cref{sec:eigenvalueproblems}).

In slab geometry, we used a reduced version of the plane-source test case (\(1000\) grid cells,
final
time \(\tf = 0.1\)). For the \(\HFMN\) models, two different implementations were tested: the
backtracking
Newton solver without change of basis (see \cref{sec:adaptivechangeofbasis})
using quadratures to calculate the integrals and the same
backtracking Newton solver where all needed integrals were solved using the analytical formulas
(and
Taylor expansion at the singularities, see \cref{sec:QuadratureRules}). In three dimensions, we used a
reduced version of the point-source test case (\(\gridsize = 10^3\) grid cells, single Runge-Kutta
step).

As can be seen in \cref{fig:Timings2}, for the first-order scheme, results are as expected
except
that the \(\Mn\) models show second-order complexity in slab geometry, probably because the
matrices
are relatively small here and thus the third-order matrix operations do not dominate the execution
time. The \(\HFMN\) implementation using analytic integrals is faster than the quadrature version
but the difference is negligible in practice.
Given that the implementation using analytic integrals is considerably more complex
and that analytic realizability conditions can also be used for the quadrature-based version,
we suggest to generally use a quadrature-based implementation also for the \(\HFMN\) models.
\begin{figure}[htbp]
  \centering
  \externaltikz{Timings}{\input{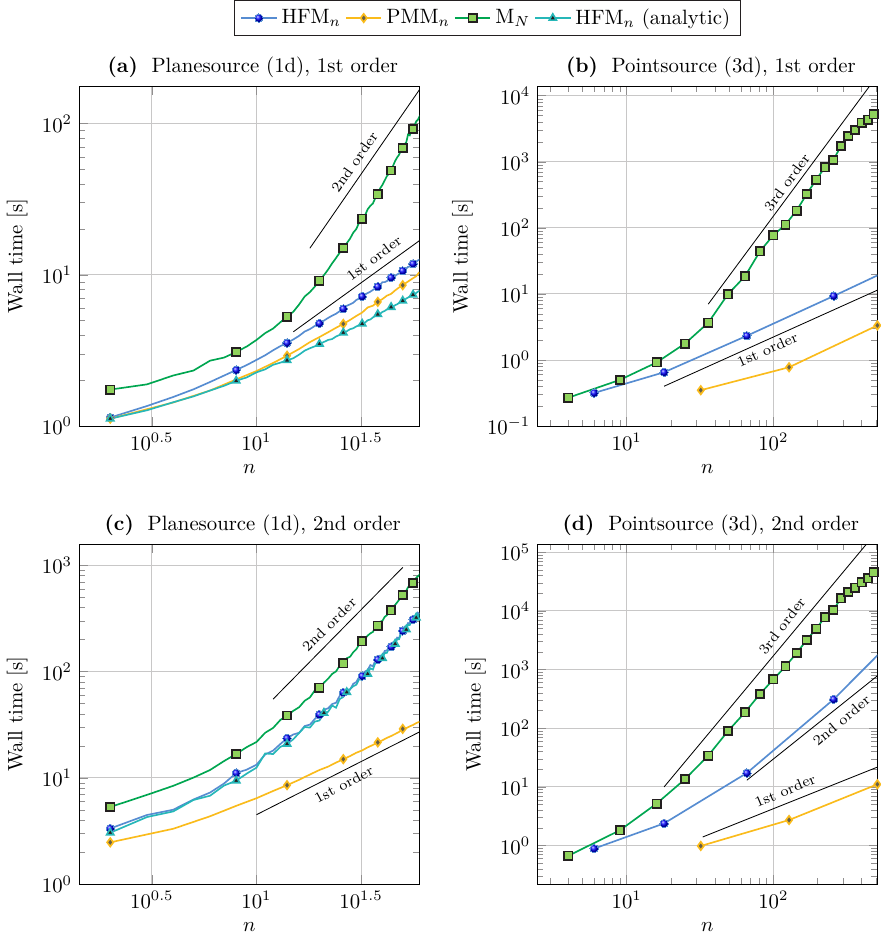}}
  \caption{Execution wall time for the minimum-entropy models in one and
    three dimensions. Times were measured in serial computations (no
    parallelization). Plotted is the minimum of three runs. a) Plane-source test case
    ($1000$ grid cells, $\tf = 0.1$). b) Point-source test case ($10^3$ grid
    cells,
    single Runge-Kutta step).}
  \label{fig:Timings2}
\end{figure}

For the second-order scheme, as expected, the \(\PMMN\) models show first-order complexity both
in slab geometry and in three dimension and thus are several orders of magnitudes faster than the
other models. Curiously, the \(\HFMN\) models are close to second-order complexity also in three
dimensions. For even larger
\(\momentnumber\), we expect the \(\HFMN\) models to also increase with third-order due to the
eigensolver
but the results show that the \(\HFMN\) models are much faster than the \(\Mn\)
models for a long time.\enlargethispage{1\baselineskip}

Note that though the \(\PMMN[2]\) and \(\Mn[2]\) models are
equivalent, the measured times are different
as the \(\Mn[2]\) model uses the convex-hull based realizability limiter while the
\(\PMMN[2]\) model
uses the limiter based on the analytical realizability conditions to be consistent with the models
with higher \(\momentnumber\).

\section{Conclusions and outlook}
We derived two classes of minimum-entropy moment models based on a continuous finite element basis
as well as a discontinuous piece-wise linear basis. Both types of models are realizable, i.e.,
generated by a non-negative ansatz, such that important physical properties like positivity of mass
are preserved. We demonstrated in various numerical tests in one and three dimensional geometry
that
those models are qualitatively competitive with the classical full-moment $\Mn$
models of the same
number of degrees of freedom if the solution of the kinetic equation only has a limited smoothness
(since otherwise the $\Mn$ models typically show spectral convergence).
Additionally, the new models
are much cheaper (with respect to running time) than the full moment models since the non-linear
problems that have to be solved locally
are much smaller and typically much easier to solve as well. Consequently, the new
models are considerably more efficient in the sense that they reach the same
approximation error with much less computation effort.
In particular, the partial moments
$\PMMN$ show a linear relation between wall time and number of moments, which is also
true for the hat function basis if only a first-order scheme is used. If a higher-order
discretization in space and time is required, the partial moments appear to be the model of choice.
However, in some cases the discontinuity in the basis functions may lead to severe problems, for
example when collision is modeled with the Laplace-Beltrami operator \cite{Schneider2014}. In such
cases, $\HFMN$ might be favorable.

We provided a second-order realizability-preserving scheme by using a splitting technique and
analytic
solutions of the stiff part, combined with a realizability-preserving reconstruction scheme.
Higher-order variants of this scheme can in principle be derived similarly, but we emphasize that
we
strictly focused on non-smooth problems, where the sense in applying schemes with (much) more than
second order is questionable.

If the underlying problem admits more smoothness (especially in the velocity domain), higher-order
moment models might be more appropriate to enhance the speed of convergence towards the kinetic
solution. While this is rather straight-forward to define both in slab as well as three-dimensional
geometry (partial moments can be constructed immediately while the hat-function basis can be
extended to higher-order splines on the unit interval/unit sphere,respectively
\cite{alfeld1996bernstein}), special care is required since the realizability conditions are needed
in order to use our realizability-preserving scheme. Up to our knowledge, the corresponding
realizability problems are only solved for partial moments (of arbitrary order) in slab geometry
\cite{Curto1991}, while first approaches are given for second-order partial moments on
quadrants/octants of the sphere~\cite{Schneider2017}.

\bibliographystyle{siam}
\bibliography{bibliography}

\end{document}


\maketitle

\noindent


\def\tikzpath{Images/}

\setcounter{equation}{0}
\setcounter{figure}{0}

\section{Evaluating the \texorpdfstring{\(\HFMN\)}{hat function} integrals in three dimensions using Taylor series}
For the hat function basis in three dimensions, integrals cannot be evaluated analytically anymore.
However, the integrals can be expanded in a Taylor series representation.
Though we dismissed that approach (since it turned out that the Taylor series had to
be computed up to a prohibitively high order) we describe it here for future reference.

Let
$\sphericaltriangle$ be a spherical triangle with vertices $\vertexA$,
$\vertexB$ and $\vertexC$ on the unit sphere $\mathcal{S}_2$. Let
$\hfbasis[\sphericaltriangle] = (\hfbasiscomp[1], \hfbasiscomp[2],
  \hfbasiscomp[3])^T$ be the barycentric basis functions on $\sphericaltriangle$.
We are interested in the integral
\begin{equation*}
  \int_{\sphericaltriangle} f(\SC) \exp(\hfbasis[\sphericaltriangle](\SC)\cdot\multipliers)~d\SC
\end{equation*}
for some function \(f\).
Since we cannot evaluate this integral analytically, we write the exponential function in its
Taylor series representation.
Let $\SC = (\SCx, \SCy, \SCz)^T$ and $\multipliers = \left(\multiplierscomp{1},\ldots,\multiplierscomp{3}\right)^T$. Using the partition
of unity property of the barycentric basis functions, we get
\begin{align*}
  \exp(\hfbasis[\sphericaltriangle]\cdot\multipliers)
   & = \exp(\sum_{i=1}^3 \hfbasiscomp[i]\multiplierscomp{i})
  = \exp(\hfbasiscomp[1]\multiplierscomp{1}+\hfbasiscomp[2]\multiplierscomp{2}+(1-\hfbasiscomp[1]-\hfbasiscomp[2])\multiplierscomp{3})               \\
   & = e^{\multiplierscomp{3}} e^{\hfbasiscomp[1](\multiplierscomp{1}-\multiplierscomp{3})+\hfbasiscomp[2](\multiplierscomp{2}-\multiplierscomp{3})}
\end{align*}
Expanding the second exponential in a Taylor series representation gives
\begin{equation*}
  \exp(\hfbasis[\sphericaltriangle](\SC)\cdot\multipliers)
  = \exp(\multiplierscomp{3}) \sum_{k=0}^{\infty} \cfrac{\left(\sum_{i=1}^2 \hfbasiscomp[i](\multiplierscomp{i}-\multiplierscomp{3})\right)^k}{k!}
  = \sum_{k=0}^{\infty} \sum_{k_1+k_2=k}\prod_{i=1}^{2}\cfrac{(\hfbasiscomp[i](\multiplierscomp{i}-\multiplierscomp{3}))^{k_i}}{k_i!}
\end{equation*}
where we used the multinomial theorem for the last equality.
Interchanging summation and integration yields
\begin{equation*}
  \int_{\sphericaltriangle} f(\SC) \exp(\hfbasis[\sphericaltriangle](\SC)\cdot\multipliers)~d\SC
  = \exp(\multiplierscomp{3}) \sum_{k=0}^{\infty} \sum_{k_1+k_2=k}
  \left(\prod_{i=1}^{2}\frac{(\multiplierscomp{i}-\multiplierscomp{3})^{k_i}}{k_i!}\right)
  \int_{\sphericaltriangle} f(\SC) \prod_{i=1}^{2} {\hfbasiscomp[i]}^{k_i} ~d\SC
\end{equation*}
The integrals
$\int_{\sphericaltriangle} f(\SC) \prod_{i=1}^{2} {\hfbasiscomp[i]}^{k_i} ~d\SC$
are independent of \(\multipliers\)
can be precomputed once and for all (up to some maximal order), if the
spherical triangle does not change. For the optimization algorithm, we need
to calculate these integrals for
\begin{equation*}
  f(\SC) \in \{1, \hfbasiscomp[i](\SC), {\SC}_k
  \hfbasiscomp[i](\SC), \hfbasiscomp[i](\SC) \hfbasiscomp[j](\SC),
  {\SC}_k \hfbasiscomp[i](\SC) \hfbasiscomp[j](\SC)
  \text{ for } i, j, k \in \{1, 2, 3\}\}.
\end{equation*}
\begin{remark}
  Note that, if the spherical triangle is contained in an octant of the sphere, none of these
  possible choices for $f(\SC)$ changes sign over the
  domain of integration. Thus, if we order the multipliers such that $\multiplierscomp{3} = \min_{i=1}^3 \multiplierscomp{i}$, all terms
  of the Taylor expansion have the same
  sign. Hence, if we precalculate the integrals for all three possible choices of the basis function
  $\hfbasiscomp[3]$, we can calculate the Taylor series without numerical cancelation.
\end{remark}
This procedure allows for using high-order quadratures to precompute the
integrals up to some maximal order (which is reasonably fast even for very fine
quadratures). When solving the optimization problems, we only have to evaluate
the Taylor series, which is considerably faster than using a quadrature of
comparable order. We tested this procedure with a maximal order of $250$ and it
indeed worked quite well for the vast majority of optimization problems.
However, for moments corresponding to anisotropic distributions,
$(\multiplierscomp{i}-\multiplierscomp{3})$ may become arbitrarily large such
that the Taylor series has to evaluated up to a prohibitively high order. This
leads to additional regularization for these moments which introduces additional errors.

\section{Quadrature sensitivity}\label{sec:appendixQuadratures}
For the minimum entropy models, we usually cannot calculate the
integrals occurring in the
minimum-entropy optimization problems analytically but have to use a
quadrature. Choosing an appropriate quadrature is not trivial as it has a great
influence on both realizability and performance. Due to realizability
considerations, in one dimension, we chose
Gauss-Lobatto quadratures. For the $\HFMN$ and $\PMMN$ models, we use
one
quadrature per interval $\cell{\cellindex}$. The $\Mn$ models do not use a
subdivision of the quadrature domain in intervals. However, for the kinetic
flux we integrate over the intervals $[-1, 0]$ and $[0, 1]$, so we use
one
quadrature on each of these two intervals. For the one-dimensional $\PMMN$ and
$\HFMN$ models, the integrals can be solved analytically, which we did for the
$\HFMN$ models in our implementation. However, for the sake of completeness, we
also tested the $\HFMN$ models using quadratures instead of the analytical
formulas.

To test which quadrature order to use and how sensitive the different models
are with respect to the quadrature order, we solved some of our numerical test
cases for different quadrature orders and calculated the errors with respect to
the reference solution. For high quadrature orders, the integrals should be
approximated very good such that the error with respect to the reference
solution should
only be due to the moment approximation, not due to errors in evaluating the
integrals. We thus expect the errors to converge to a model-dependent limit
value with increasing quadrature order.

The results for one dimension can be found in Figures \ref{fig:PMMQuadratures},
\ref{fig:HFMQuadratures} and \ref{fig:LEMQuadratures}. The $\HFMN$ and
$\PMMN$
models show similar behavior, which is expected as they are both first-order
models and thus similar integrals have to be approximated. In the source-beam
test case, both in $\Lp{1}$ and $\Lp{\infty}$ norm, the error mostly reaches
its limit value already for a quadrature order of $5$. The plane-source test
case is more sensitive to badly approximated integrals and needs a quadrature
order of about $11$ to reach its limit. Unsurprisingly, models with fewer
intervals are more sensitive to low-order quadratures. Given these results, we use a quadrature
order of 15 for the $\HFMN$ and $\PMMN$ models in our numerical
tests.

For the $\Mn$ models, obviously, higher-order models need higher-order
quadratures. For the plane-source test case, a quadrature order of $2\momentorder+40$ seems
appropriate. The source-beam test case is a special case due to the approximate
Dirac boundary value. For the $\HFMN$ and $\PMMN$ models, the
boundary value
can be evaluated analytically, yielding a (numerically) realizable moment
vector as the numerically realizable set equals the analytically realizable
set. For the Legendre basis, however, we cannot evaluate the boundary integrals
analytically and we cannot use an arbitrary high-order quadrature for the
boundary-value only as the resulting moment vector might not be numerically
realizable. We thus use the same quadrature to evaluate the boundary value as
we use in the optimization problem. To fully resolve the boundary value, we
have to use a much higher quadrature order than we need for the plane-source
test case. To be on the safe side, we use a quadrature with $100$ quadrature
points per half interval (order $197$) in this test case.

\begin{figure}[htbp]
  \centering
  \externaltikz{1dPmmQuadratures}{\input{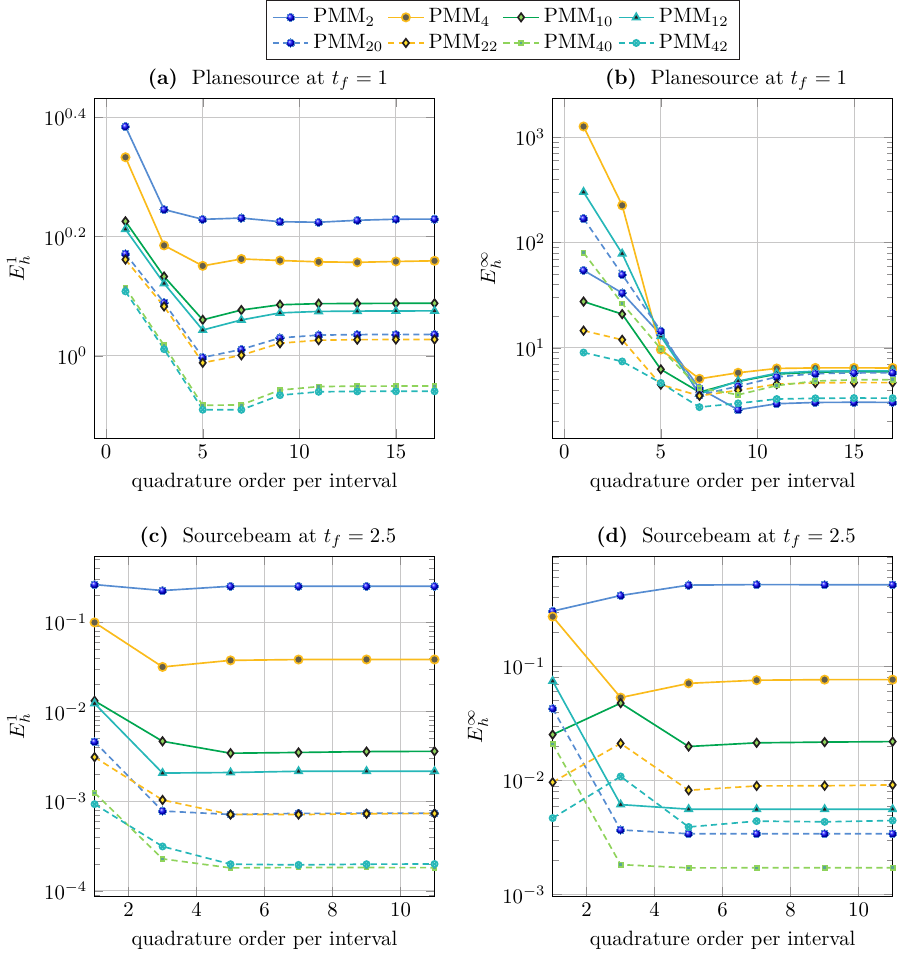}}
  \caption{Analysis of the quadrature dependency of the $\PMMN$ models.}
  \label{fig:PMMQuadratures}
\end{figure}

\begin{figure}[htbp]
  \centering
  \externaltikz{1dHfmQuadratures}{\input{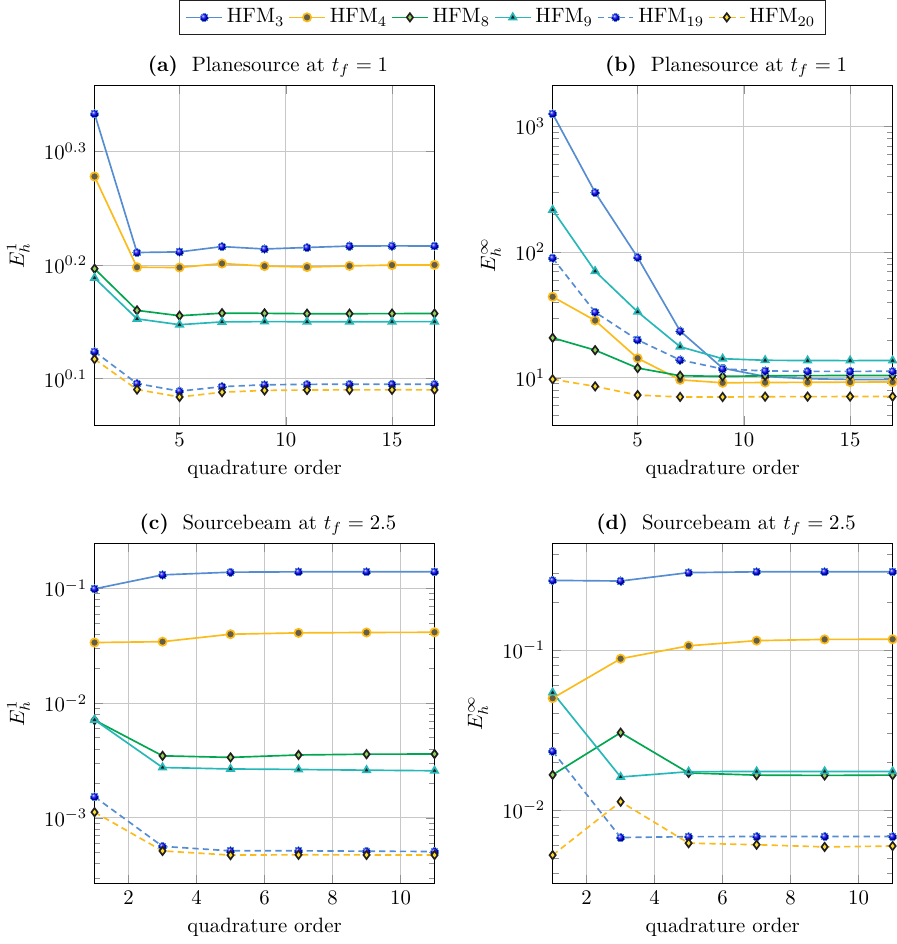}}
  \caption{Analysis of the quadrature dependency of the $\HFMN$ models.}
  \label{fig:HFMQuadratures}
\end{figure}

\begin{figure}[htbp]
  \centering
  \externaltikz{1dLemQuadratures}{\input{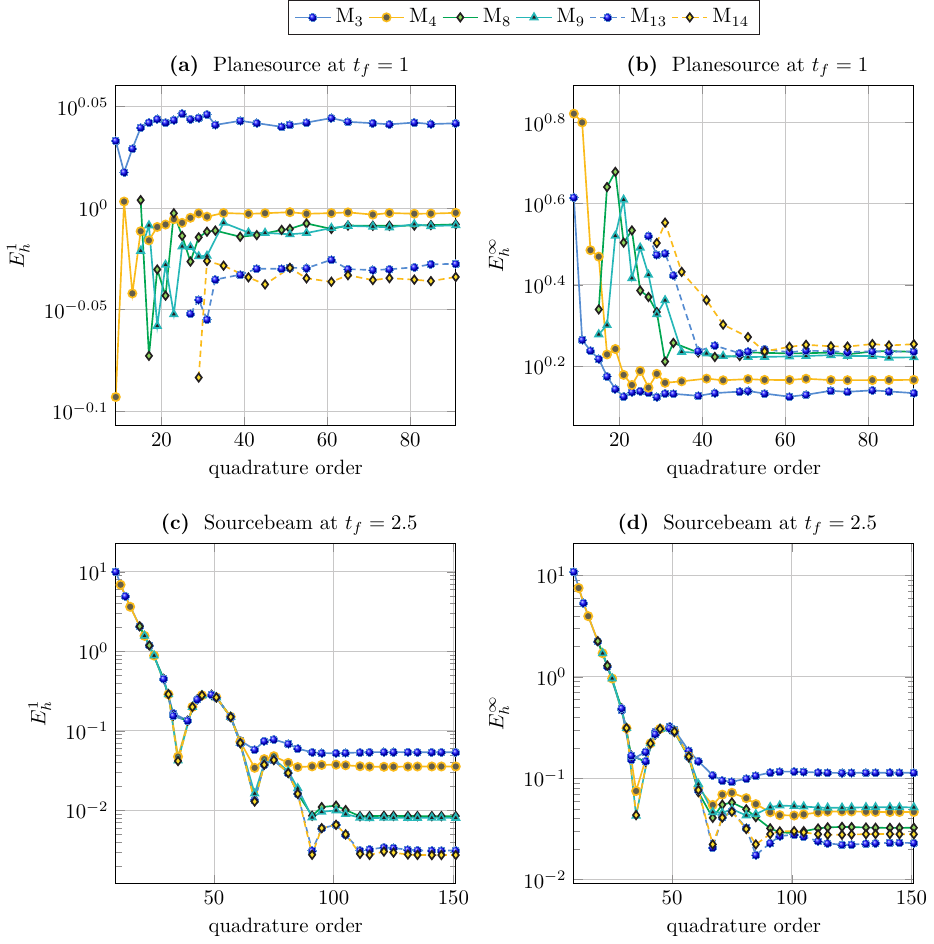}}
  \caption{Analysis of the quadrature dependency of the $\Mn$ models.}
  \label{fig:LEMQuadratures}
\end{figure}

In three dimension, we use Fekete quadratures on each spherical triangle for the
$\HFMN$ and $\PMMN$ models and tensor-product rules on the octants of
the sphere
for the $\Mn$ models. To test the influence
of the quadrature, we solve the pointsource problem (using $50^3$ grid cells)
for each quadrature and calculate the error with respect to the analytical
solution of the kinetic equation. The
results for the $\HFMN$ and $\PMMN$ can be found in
\figref{fig:3dPmmAndHfmQuadratures}. The models with $8$ spherical triangles
($\HFMN[6], \PMMN[32]$) give significantly different results when a low-order
quadrature is used. A quadrature order of about $12$ is needed to fully resolve
the structure in these models. In contrast, the error graphs for the
higher-order $\HFMN$ and $\PMMN$ are mostly flat, so these models do
not profit
from quadratures with degree larger than $6$. Apparently, the finer
triangulation of the quadrature domain in these models is sufficient to properly
approximate the integrals even with low-order quadratures on each triangle. For
the following numerical experiments, we thus use a quadrature order of $15$ for
$\HFMN[6]$ and $\PMMN[32]$ and order $9$ for the other
models.

For the $\Mn$ models, the results can be found in
\figref{fig:3dShmQuadratures}). Obviously, higher-order models need
higher-order quadratures. We use a quadrature order of $2N+8$ in the numerical
experiments.

\begin{figure}
  \centering
  \externaltikz{3dPmmAndHfmQuadratures}{\input{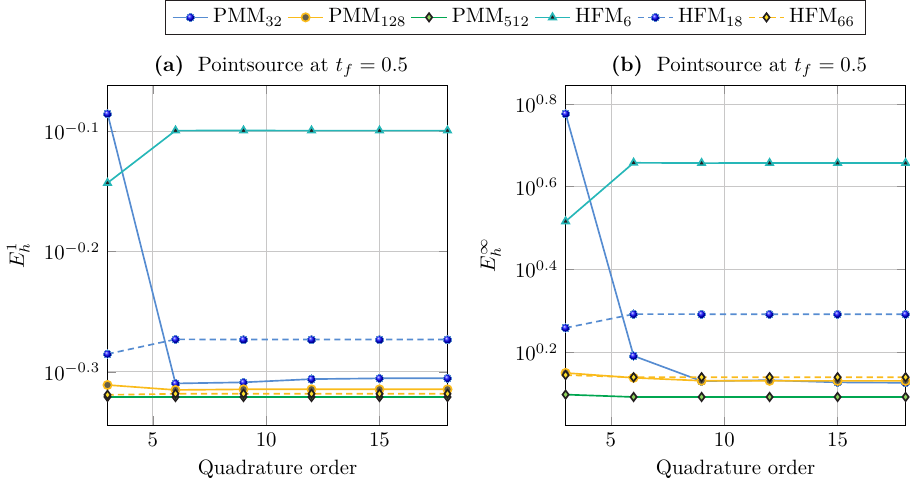}}
  \caption{$\Lp{1}$ and $\Lp{\infty}$ errors of $\PMMN$ and
    $\HFMN$
    models for different quadratures. A Fekete quadrature rule of the respective
    order is used on each spherical triangle.}
  \label{fig:3dPmmAndHfmQuadratures}
\end{figure}

\begin{figure}
  \centering
  \externaltikz{3dShmQuadratures}{\input{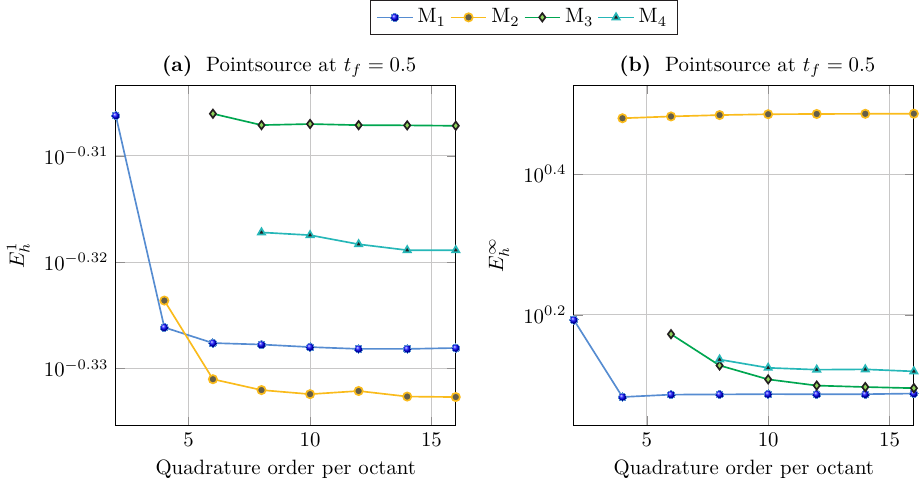}}
  \caption{Analysis of the quadrature dependency of the $\Mn$ models.}
  \label{fig:3dShmQuadratures}
\end{figure}

\pagebreak
\section{Supplementary figures}
\setcounter{figure}{0}
We here include some plots of the solutions of the Checkerboard and Shadow test cases.
All models are shown as two-dimensional slices through the spatial domain, as well as isosurfaces
in an endcap geometry (i.e., some portion of the surfaces is removed to get some insight into the
interior of the solution).
\begin{figure}[hbp]
  \centering
  \externaltikz{Checkerboard1}{\input{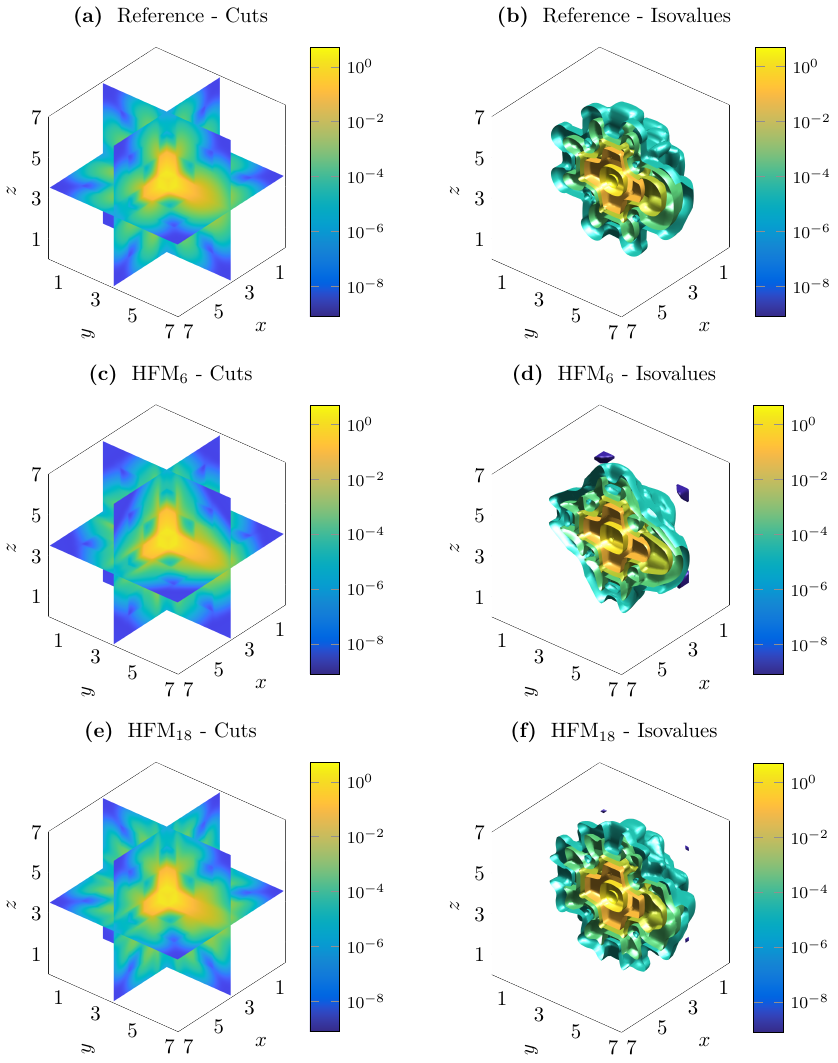}}
  \caption{Two-dimensional cuts and selected isosurfaces for some models in the checkerboard
    test, logarithmic scale.}
  \label{fig:Checkerboard1}
\end{figure}

\begin{figure}
  \centering
  \externaltikz{Checkerboard2}{\input{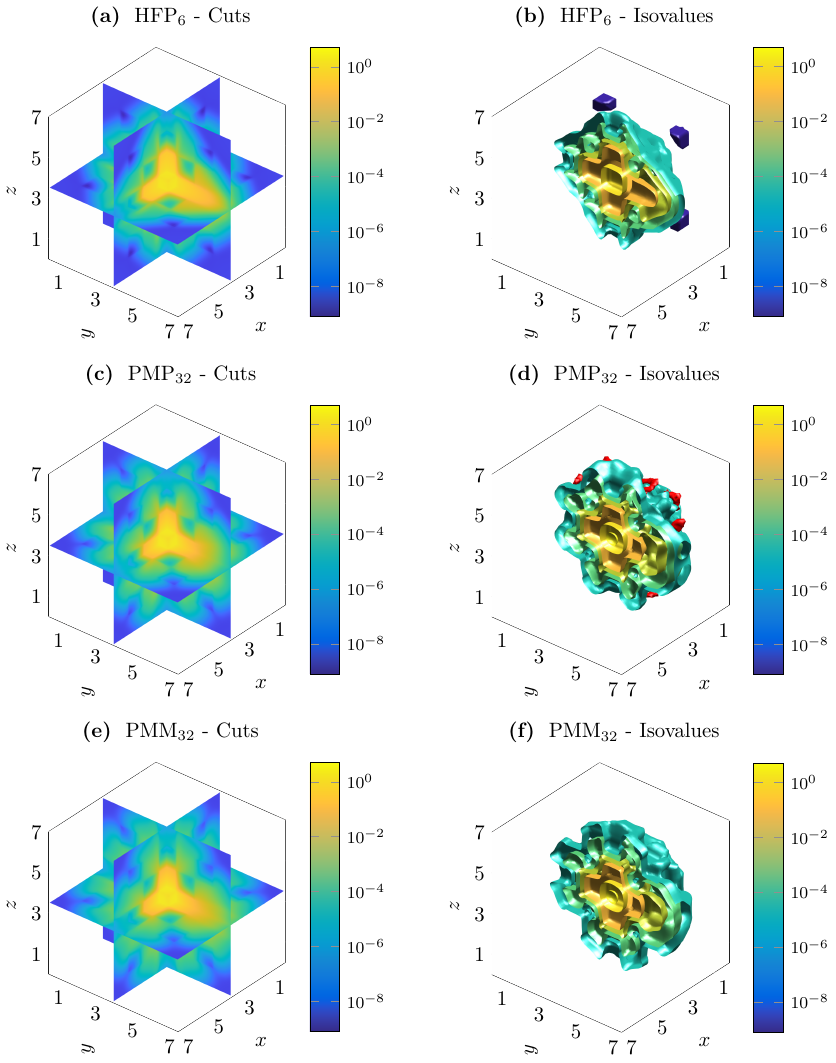}}
  \caption{Two-dimensional cuts and selected isosurfaces for some models in the checkerboard
    test, logarithmic scale. Negative values are shown in red.}
  \label{fig:Checkerboard2}
\end{figure}
\begin{figure}
  \centering
  \externaltikz{Checkerboard3}{\input{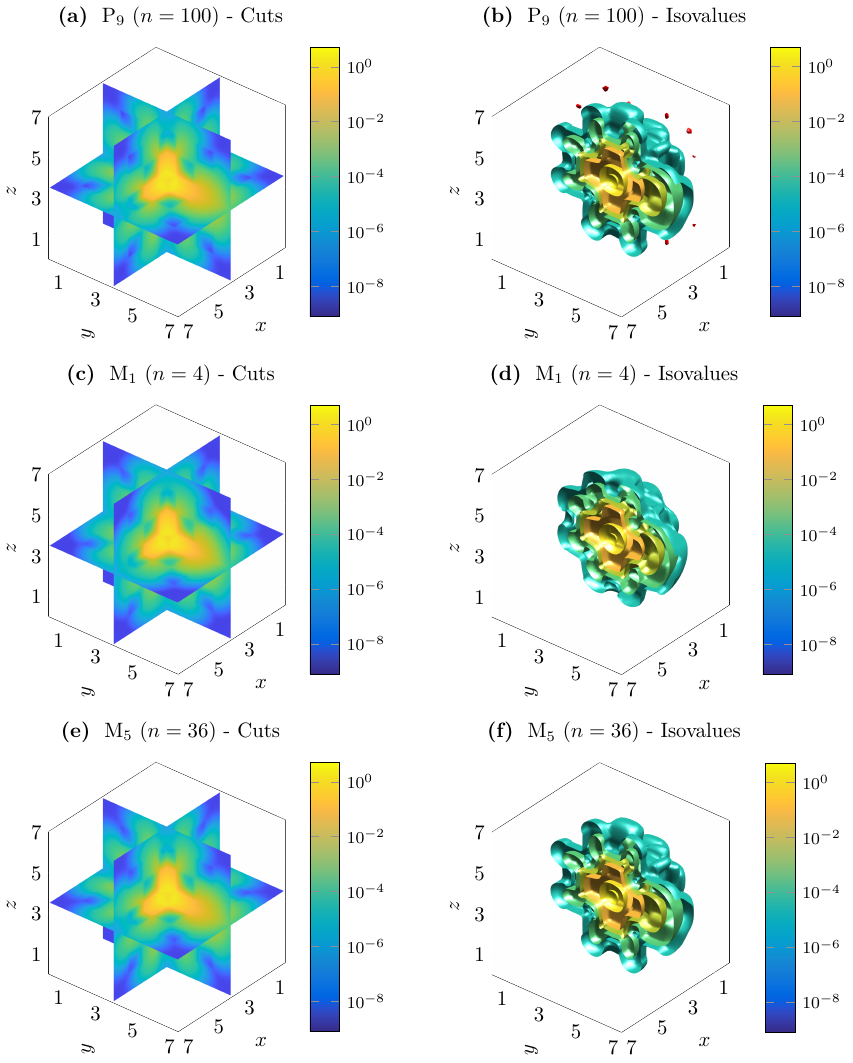}}
  \caption{Two-dimensional cuts and selected isosurfaces for some models in the checkerboard
    test, logarithmic scale. Negative values are shown in red.}
  \label{fig:Checkerboard3}
\end{figure}

\begin{figure}
  \centering
  \externaltikz{Shadow1}{\input{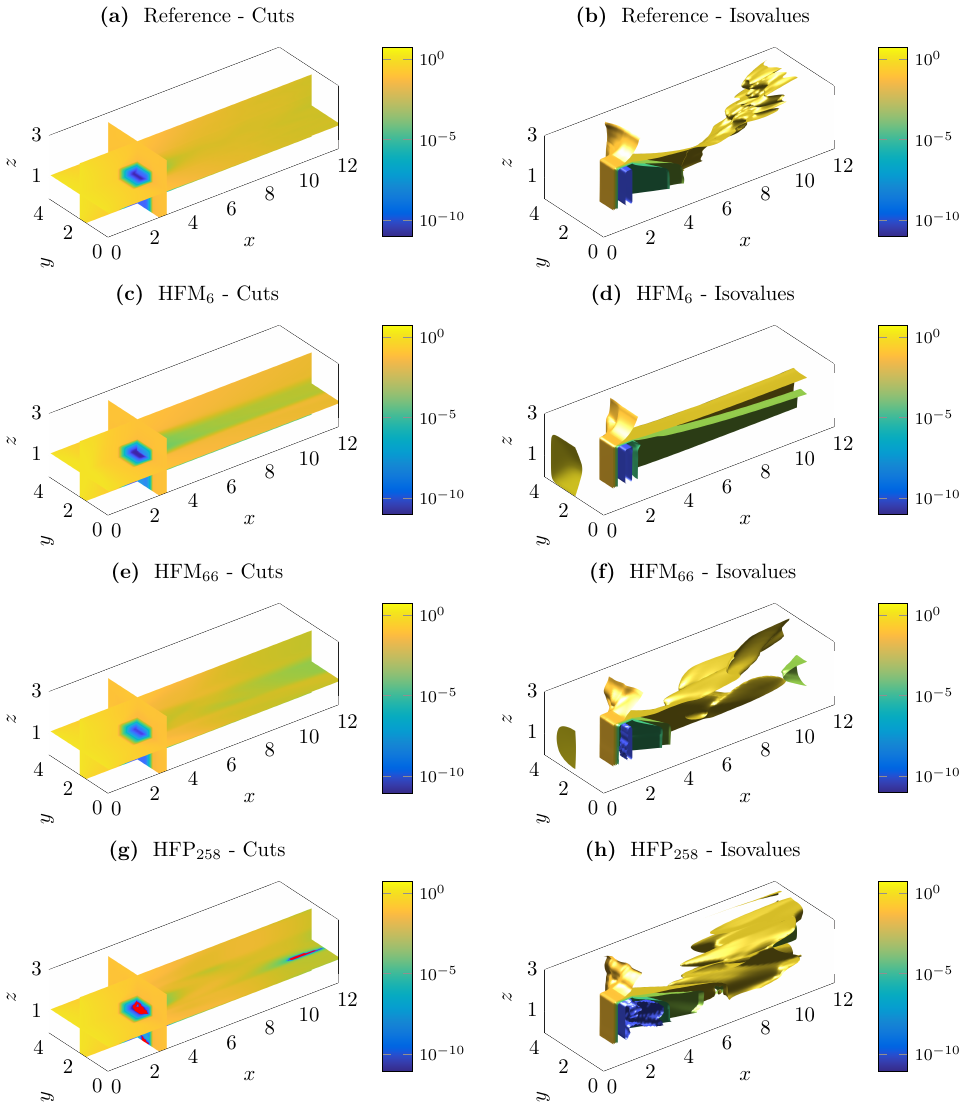}}
  \caption{Two-dimensional cuts and selected isosurfaces for some models in the shadow test, logarithmic scale. Negative values are shown in red.}
  \label{fig:Shadow1}
\end{figure}
\begin{figure}
  \centering
  \externaltikz{Shadow2}{\input{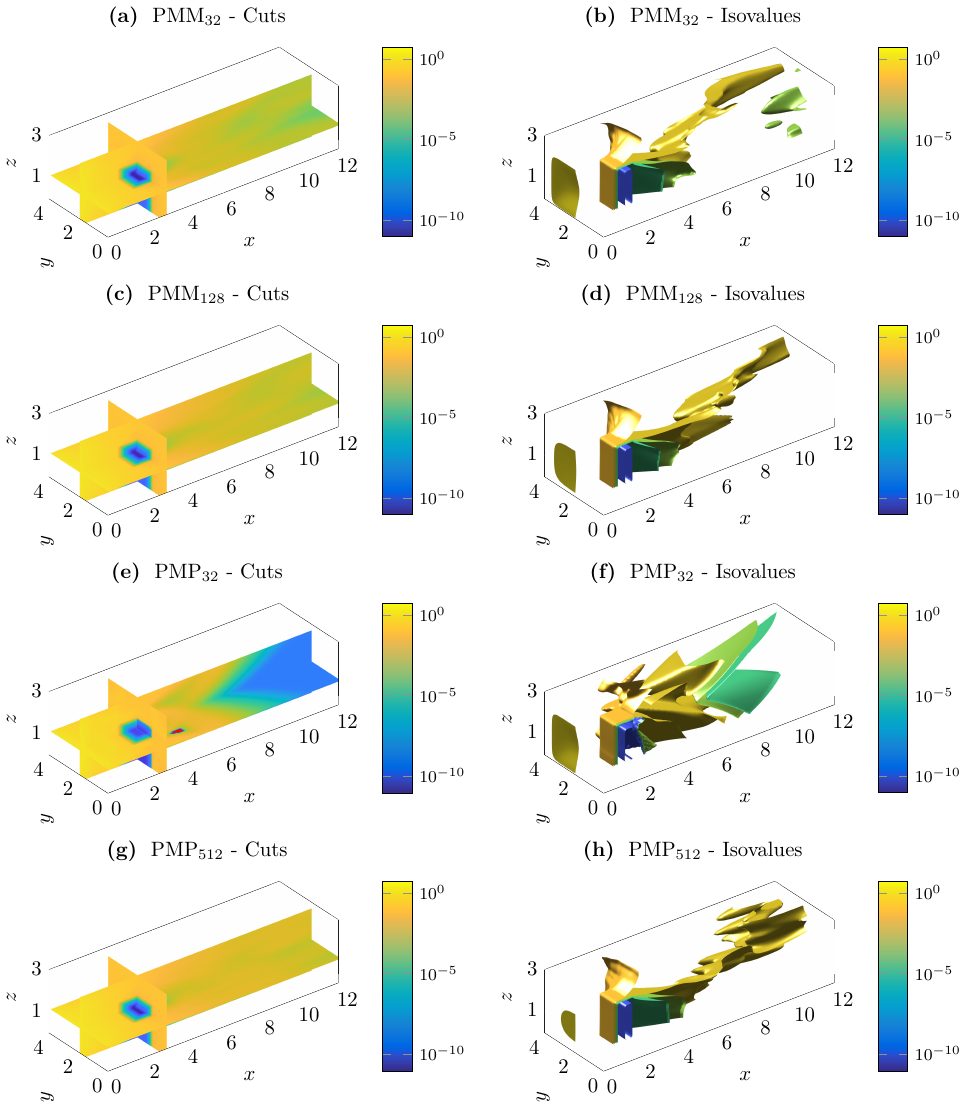}}
  \caption{Two-dimensional cuts and selected isosurfaces for some models in the shadow test, logarithmic scale. Negative values are shown in red.}
  \label{fig:Shadow2}
\end{figure}

\begin{figure}
  \centering
  \externaltikz{Shadow3}{\input{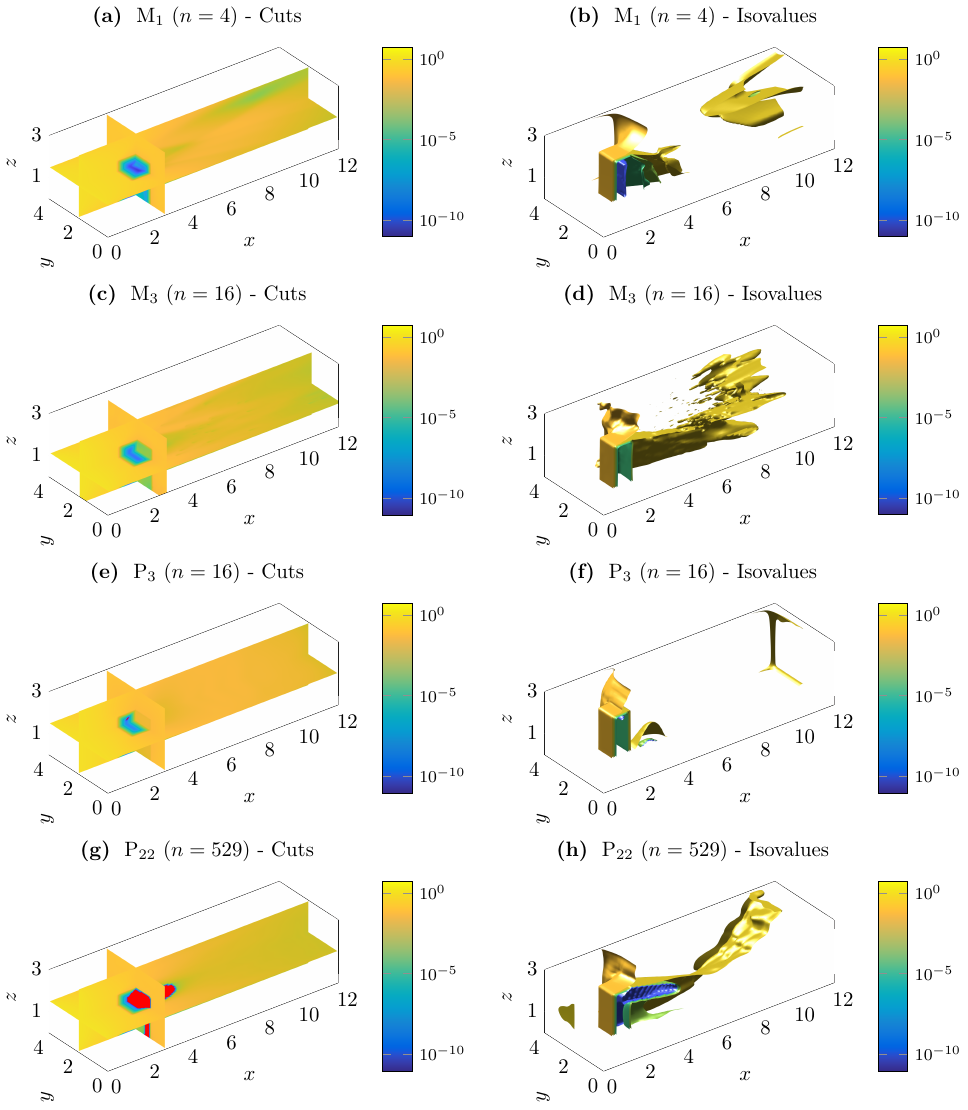}}
  \caption{Two-dimensional cuts and selected isosurfaces for some models in the shadow test, logarithmic scale. Negative values are shown in red.}
  \label{fig:Shadow3}
\end{figure}